%% file: IM_APAL2_arxiv.tex
\documentclass[11pt]{amsart}
\usepackage{amsfonts}
\usepackage{amssymb}
\usepackage{amsmath}
\usepackage{amsthm}
\usepackage[dvipdfm]{graphicx,color}
\usepackage{ascmac}
\usepackage{moreverb}
\usepackage{fancybox}
\usepackage{fancyvrb}
\usepackage{booktabs}
\usepackage{comment}
\usepackage{stmaryrd}
\usepackage{txfonts}
\usepackage[all]{xy}
\usepackage{calc}

\newlength{\hoffsettmp}
\newlength{\voffsettmp}
\theoremstyle{plain}
\newtheorem{theorem}{Theorem}
\newtheorem{lemma}[theorem]{Lemma}
\newtheorem{cor}[theorem]{Corollary}
\newtheorem{prop}[theorem]{Proposition}

\newtheorem{question}[theorem]{Question}

\theoremstyle{definition}
\newtheorem{definition}[theorem]{Definition}
\newtheorem{example}[theorem]{Example}

\theoremstyle{definition}
\newtheorem*{remark}{Remark}
\newtheorem*{claim}{Claim}

\newtheorem*{req}{Requirements}

\newtheorem*{construction}{Construction}

\newtheorem*{ack}{Acknowledgements}

\newcommand{\lrangle}[1]{\langle #1 \rangle}
\newcommand{\pair}[1]{( #1 )}
\newcommand{\res}{\upharpoonright}
\newcommand{\dg}[1]{\mathbf{#1}}
\newcommand{\fr}{\mbox{}^{\smallfrown}}

\newcommand{\lsup}{\otimes}
\newcommand{\linf}{\oplus}

\newcommand{\binf}{\bigoplus}

\newcommand{\nn}{\mathbb{N}}

\newcommand{\cross}{\dagger}
\newcommand{\tie}{\triangledown}
\newcommand{\btie}{\bigtriangledown}
\newcommand{\htie}{\blacktriangledown}
\newcommand{\bhtie}{\mbox{\Large{$\blacktriangledown$}}}
\newcommand{\lcm}{\tie_{\omega}}

\newcommand{\cls}{\tie_\infty}
\newcommand{\bcls}{\btie_\infty\mbox{}}
\newcommand{\hjump}[1]{\htie(#1)}

\newcommand{\cmeet}{\binf\mbox{}^{\longrightarrow}}
\newcommand{\ntie}{\fr}
\newcommand{\shft}{\mbox{}\leftharpoonup}
\newcommand{\concat}{\bigsqcap}

\newcommand{\bhk}[1]{\llbracket #1 \rrbracket}

\makeindex

\title[Inside the Muchnik Degrees II]{Inside the Muchnik Degrees II: The Degree Structures induced by the Arithmetical Hierarchy of Countably Continuous Functions}

\author{K.~Higuchi}
\address{Department of Mathematics and Informatics, Chiba University, 1-33 Yayoi-cho, Inage, Chiba, Japan}
\email{khiguchi@g.math.s.chiba-u.ac.jp}

\author{T.~Kihara}
\address{School of Information Science, Japan Advanced Institute of Science and Technology, Nomi 923-1292, Japan}
\email{kihara@jaist.ac.jp}


\begin{document}



\begin{abstract}
It is known that infinitely many Medvedev degrees exist inside the Muchnik degree of any nontrivial $\Pi^0_1$ subset of Cantor space.
We shed light on the fine structures inside these Muchnik degrees related to learnability and piecewise computability.
As for nonempty $\Pi^0_1$ subsets of Cantor space, we show the existence of a finite-$\Delta^0_2$-piecewise degree containing infinitely many finite-$(\Pi^0_1)_2$-piecewise degrees, and a finite-$(\Pi^0_2)_2$-piecewise degree containing infinitely many finite-$\Delta^0_2$-piecewise degrees (where $(\Pi^0_n)_2$ denotes the difference of two $\Pi^0_n$ sets), whereas the greatest degrees in these three ``finite-$\Gamma$-piecewise'' degree structures coincide.
Moreover, as for nonempty $\Pi^0_1$ subsets of Cantor space,  we also show that every nonzero finite-$(\Pi^0_1)_2$-piecewise degree includes infinitely many Medvedev (i.e., one-piecewise) degrees, every nonzero countable-$\Delta^0_2$-piecewise degree includes infinitely many finite-piecewise degrees, every nonzero finite-$(\Pi^0_2)_2$-countable-$\Delta^0_2$-piecewise degree includes infinitely many countable-$\Delta^0_2$-piecewise degrees, and every nonzero Muchnik (i.e., countable-$\Pi^0_2$-piecewise) degree includes infinitely many finite-$(\Pi^0_2)_2$-countable-$\Delta^0_2$-piecewise degrees.
Indeed, we show that any nonzero Medvedev degree and nonzero countable-$\Delta^0_2$-piecewise degree of a nonempty $\Pi^0_1$ subset of Cantor space have the strong anticupping properties.
Finally, we obtain an elementary difference between the Medvedev (Muchnik) degree structure and the finite-$\Gamma$-piecewise degree structure of all subsets of Baire space by showing that none of the finite-$\Gamma$-piecewise structures are Brouwerian, where $\Gamma$ is any of the Wadge classes mentioned above.
\end{abstract}

\maketitle

\input{NRMP_fullproof1_II.tex}

\tableofcontents

\input{NRMP_fullproof2.tex}
\input{NRMP_fullproof3a.tex}

\input{NRMP_fullproof3b.tex}
\input{NRMP_fullproof5_II.tex}

\begin{ack}\upshape
The authors were partially supported by Grant-in-Aid for JSPS fellows.
The second author (Kihara) would like to thank Douglas Cenzer, Hajime Ishihara, Dick de Jongh, Arno Pauly, and Albert Visser, for valuable comments and helpful discussion, and the second author also would like to thank Makoto Tatsuta and Yoriyuki Yamagata for introducing him to the syntactical study on Limit Computable Mathematics.
The second author is also grateful to Sam Sanders who helped his English writing.
Finally, the authors would like to thank the anonymous referees for their valuable comments and suggestions.
\end{ack}

\addcontentsline{toc}{section}{Bibliography}
\bibliographystyle{plain}
\bibliography{IMref}

\printindex

\end{document}

%% file: NRMP_fullproof1_II.tex
\section{Summary}

\subsection{Introduction}

This paper is a continuation of Higuchi-Kihara \cite{HK_PartI}.
Our objective in this paper is to investigate the degree structures induced by intermediate notions between the Medvedev reduction (uniformly computable function) and Muchnik reduction (nonuniformly computable function).
We will shed light on a hidden, but extremely deep, structure inside the Muchnik degree of each $\Pi^0_1$ subset of Cantor space.

In 1963, Albert Muchnik \cite{Muc} introduced the notion of Muchnik reduction as a partial function on Baire space that is decomposable into countably many computable functions.
Such a reduction is also called a {\em countably computable} function, {\em $\sigma$-computable} function, or {\em nonuniformly computable} function.
The notion of Muchnik reduction has been a powerful tool for clarifying the noncomputability structure of the $\Pi^0_1$ subsets of Cantor space \cite{Sim,Sim2,Sim5,Simta}.
Muchnik reductions have been classified in Part I \cite{HK_PartI} by introducing the notion of piecewise computability.

Remarkably, many descriptive set theorists have recently focused their attention on the concept of {\em piecewise definability} of functions on Polish spaces, in association with the Baire hierarchy of Borel measurable functions (see \cite{MR2,MRpre,Sem}).
Roughly speaking, if $\mathbf{\Gamma}$ is a pointclass (in the Borel hierarchy) and $\mathbf{\Lambda}$ is a class of functions (in the Baire hierarchy), a function is said to be $\mathbf{\Gamma}$-piecewise $\mathbf{\Lambda}$ if it is decomposable into countably many $\mathbf{\Lambda}$-functions with $\mathbf{\Gamma}$ domains.
If $\mathbf{\Gamma}$ is the class of all closed sets and $\mathbf{\Lambda}$ is the class of all continuous functions, it is simply called {\em piecewise continuous} (see for instance \cite{JR,MRS,MRSK13,PBreta}).
The notion of piecewise continuity is known to be equivalent to the $\mathbf{\Delta}^0_2$-measurability \cite{JR}.
If $\mathbf{\Gamma}$ is the class of all sets and $\mathbf{\Lambda}$ is the class of all continuous functions, it is also called {\em countably continuous} \cite{MRpre} or {\em $\sigma$-continuous} (\cite{Sabok09}).
Nikolai Luzin was the first to investigate the notion of countable-continuity, and today, many researchers have studied this concept, in particular, with an important dichotomy theorem (\cite{Sole98,PawSab}).

Our concepts introduced in Part I \cite{HK_PartI}, such as {\em $\Delta^0_2$-piecewise computability}, are indeed the lightface versions of piecewise definability.
This notion is also known to be equivalent to the effective $\Delta^0_2$-measurability (\cite{PBreta}).
See also \cite{Bra1,deBre13,Kihta} for more information on effective Borel measurability.

To gain a deeper understanding of piecewise definability, we investigate the Medvedev/Muchnik-like degree structures induced by piecewise computable notions.
This also helps us to understand the notion of relative learnability since we have observed a close relationship between lightface piecewise definability and algorithmic learning in Part I \cite{HK_PartI}.

In Part II, we restrict our attention to the local substructures consisting of the degrees of all $\Pi^0_1$ subsets of Cantor space.
This indicates that we consider the relative piecewise computably (or learnably) solvability of {\em computably-refutable problems}.
When a scientist attempts to verify a statement $P$, his verification will be algorithmically refuted whenever it is incorrect.
This {\em falsifiability principle} holds only when $P$ is represented as a $\Pi^0_1$ subset of a space.
Therefore, the restriction to the $\Pi^0_1$ sets can be regarded as an analogy of {\em Popperian learning} \cite{CN79} because of the falsifiability principle.
From this perspective, the universe of the $\Pi^0_1$ sets is expected to be a good playground of Learning Theory \cite{JORS}.

The restriction to the $\Pi^0_1$ subsets of Cantor space $2^\nn$ is also motivated by several other arguments.
First, many mathematical problems can be represented as $\Pi^0_1$ subsets of certain topological spaces (see Cenzer and Remmel \cite{CR}).
The $\Pi^0_1$ sets in such spaces have become important notions in many branches of Computability Theory, such as {\em Recursive Mathematics} \cite{HRM}, {\em Reverse Mathematics} \cite{SimRM}, {\em Computable Analysis} \cite{Wei}, {\em Effective Randomness} \cite{NieR,DHra}, and {\em Effective Descriptive Set Theory} \cite{MosDS}.
For these reasons, degree structures on $\Pi^0_1$ subsets of Cantor space $2^\nn$ are widely studied from the viewpoint of {\em Computability Theory} and {\em Reverse Mathematics}.

In particular, many theorems have been proposed on the algebraic structure of the Medvedev degrees of $\Pi^0_1$ subsets of Cantor space, such as density \cite{CH1}, embeddability of distributive lattices \cite{BS}, join-reducibility \cite{Bin}, meet-irreducibility \cite{Al1}, noncuppability \cite{CKWW}, non-Brouwerian property \cite{Higuchi12}, decidability \cite{CK}, and undecidability \cite{Sha} (see also \cite{Hin,Sim,Sim2,Sim5,Simta} for other properties on the Medvedev and Muchnik degree structures).
The $\Pi^0_1$ sets have also been a key notion (under the name of {\em closed choice}) in the study of the structure of the Weihrauch degrees, which is an extension of the Medvedev degrees (see \cite{BMP,BGa,BG}).

Among other results, Cenzer and Hinman \cite{CH1} showed that the Medvedev degrees of the $\Pi^0_1$ subsets of Cantor space are dense, and Simpson \cite{Sim} questioned whether the Muchnik degrees of $\Pi^0_1$ subsets of Cantor space are also dense.
However, this question remains unanswered.
We have limited knowledge of the Muchnik degree structure of the $\Pi^0_1$ sets because the Muchnik reductions are very difficult to control.
What we know is that as shown by Simpson-Slaman \cite{SiSl} and Cole-Simpson \cite{CoSi}, there are infinitely many Medvedev degrees in the Muchnik degree of any nontrivial $\Pi^0_1$ subsets of Cantor space.
Now, it is necessary to clarify the internal structure of the Muchnik degrees.
In Part II, we apply the disjunction operations introduced in Part I \cite{HK_PartI} to understand the inner structures of the Muchnik degrees induced by various notions of piecewise computability.

\subsection{Results}

In Part I \cite{HK_PartI}, the notions of piecewise computability and the induced degree structures are introduced.
Our objective in Part II is to study the interaction among the structures $\mathcal{P}/\mathcal{F}$ of $\mathcal{F}$-degrees of nonempty $\Pi^0_1$ subsets of Cantor space for notions $\mathcal{F}$ of piecewise computability listed as follows.
\index{$\mathcal{P}/\mathcal{F}$}%

\begin{itemize}
\item
\index{${\rm dec}^{<\omega}_{\rm p}[\Pi^0_1]$}%
${\rm dec}^{<\omega}_{\rm p}[\Pi^0_1]$ also denotes the set of all partial computable functions on $\nn^\nn$ that are decomposable into finitely many partial computable functions with $\Pi^0_1$ domains.
\item
\index{${\rm dec}^{<\omega}_{\rm d}[\Pi^0_1]$}%
${\rm dec}^{<\omega}_{\rm d}[\Pi^0_1]$ denotes the set of all partial functions on $\nn^\nn$ that are decomposable into finitely many partial computable functions with $(\Pi^0_1)_2$ domains, where a $(\Pi^0_1)_2$ set is the difference of two $\Pi^0_1$ sets.
\item ${\rm dec}^{<\omega}_{\rm p}[\Delta^0_2]$ denotes the set of all partial functions on $\nn^\nn$ that are decomposable into finitely many partial computable functions with $\Delta^0_2$ domains.
\item
\index{${\rm dec}^{\omega}_{\rm p}[\Delta^0_2]$}%
${\rm dec}^{\omega}_{\rm p}[\Delta^0_2]$ denotes the set of all partial functions on $\nn^\nn$ that are decomposable into countably many partial computable functions with $\Delta^0_2$ domains.
\item
\index{${\rm dec}^{<\omega}_{\rm d}[\Pi^0_2]$}%
${\rm dec}^{<\omega}_{\rm d}[\Pi^0_2]$ denotes the set of all partial functions on $\nn^\nn$ that are decomposable into finitely many partial computable functions with $(\Pi^0_2)_2$ domains.
\item
\index{${\rm dec}^{<\omega}_{\rm d}[\Pi^0_2]{\rm dec}^{\omega}_{\rm p}[\Delta^0_2]$}%
${\rm dec}^{<\omega}_{\rm d}[\Pi^0_2]{\rm dec}^{\omega}_{\rm p}[\Delta^0_2]$ denotes the set of all partial functions on $\nn^\nn$ that are decomposable into finitely many partial $\Delta^0_2$-piecewise computable functions with $(\Pi^0_2)_2$ domains, where a $(\Pi^0_2)_2$ set is the difference of two $\Pi^0_2$ sets.
\item
\index{${\rm dec}^{\omega}_{\rm p}[\Pi^0_2]$}%
${\rm dec}^{\omega}_{\rm p}[\Pi^0_2]$ denotes the set of all partial functions on $\nn^\nn$ that are decomposable into countably many partial computable functions with $\Pi^0_2$ domains.
\end{itemize}

\begin{center}\footnotesize
\begin{tabular}{ccccc}
 &\rotatebox[origin=r]{30}{---}&$\mathcal{P}/{\rm dec}^{<\omega}_{\rm d}[\Pi^0_2]$&\rotatebox{-30}{---}& \\
$\mathcal{P}/{\rm dec}^{<\omega}_{\rm p}[\Pi^0_1]\;$ --- $\;\mathcal{P}/{\rm dec}^{<\omega}_{\rm d}[\Pi^0_1]\;$ --- $\;\mathcal{P}/{\rm dec}^{<\omega}_{\rm p}[\Delta^0_2]$& & & & $\mathcal{P}/{\rm dec}^{<\omega}_{\rm d}[\Pi^0_2]{\rm dec}^{\omega}_{\rm p}[\Delta^0_2]\;$ --- $\;\mathcal{P}/{\rm dec}^{\omega}_{\rm p}[\Pi^0_2]$\\
 &\rotatebox[origin=r]{-30}{---}&$\mathcal{P}/{\rm dec}^{\omega}_{\rm p}[\Delta^0_2]$&\rotatebox{30}{---}& 
\end{tabular}
\end{center}

In Part I \cite{HK_PartI}, we observed that these degree structure are exactly those induced by the $(\alpha,\beta|\gamma)$-computability.
\begin{itemize}
\item
\index{$[\mathfrak{C}_T]^1_1$}%
$[\mathfrak{C}_T]^1_1$ denotes the set of all partial computable functions on $\nn^\nn$.
\item
\index{$[\mathfrak{C}_T]^1_{<\omega}$}%
$[\mathfrak{C}_T]^1_{<\omega}$ denotes the set of all partial functions on $\nn^\nn$ learnable with bounded mind changes.
\item
\index{$[\mathfrak{C}_T]^1_{\omega|<\omega}$}%
$[\mathfrak{C}_T]^1_{\omega|<\omega}$ denotes the set of all partial functions on $\nn^\nn$ learnable with bounded errors.
\item
\index{$[\mathfrak{C}_T]^1_{\omega}$}%
$[\mathfrak{C}_T]^1_{\omega}$ denotes the set of all partial learnable functions on $\nn^\nn$.
\item
\index{$[\mathfrak{C}_T]^{<\omega}_{1}$}%
$[\mathfrak{C}_T]^{<\omega}_{1}$ denotes the set of all partial $k$-wise computable functions on $\nn^\nn$ for some $k\in\nn$.
\item
\index{$[\mathfrak{C}_T]^{<\omega}_{\omega}$}%
$[\mathfrak{C}_T]^{<\omega}_{\omega}$ denotes the set of all partial functions on $\nn^\nn$ learnable by a team.
\item
\index{$[\mathfrak{C}_T]^\omega_1$}%
$[\mathfrak{C}_T]^\omega_1$ denotes the set of all partial nonuniformly computable functions on $\nn^\nn$ (i.e., all functions $f$ satisfying $f(x)\leq_Tx$ for any $x\in{\rm dom}(f)$).
\end{itemize}

As in Part I \cite{HK_PartI}, each degree structure $\mathcal{P}/[\mathfrak{C}_T]^\alpha_{\beta|\gamma}$ is abbreviated as $\mathcal{P}^\alpha_{\beta|\gamma}$.
\index{$\mathcal{P}^\alpha_{\beta\mid\gamma}$}%

\begin{center}\footnotesize
\begin{tabular}{ccccc}
 &\rotatebox[origin=r]{30}{---}&$\mathcal{P}^{<\omega}_1$&\rotatebox{-30}{---}& \\
$\mathcal{P}^1_1\;\;$ --- $\;\;\mathcal{P}^1_{<\omega}\;\;$ --- $\;\;\mathcal{P}^1_{\omega|<\omega}$& & & & $\mathcal{P}^{<\omega}_\omega\;\;$ --- $\;\;\mathcal{P}^\omega_1$\\
 &\rotatebox[origin=r]{-30}{---}&$\mathcal{P}^1_\omega$&\rotatebox{30}{---}& 
\end{tabular}
\end{center}

We will see that all of the above inclusions are proper.
Beyond the properness of these inclusions, there are four LEVELs signifying the differences between two classes $\mathfrak{F}$ and $\mathfrak{G}$ of partial functions on $\nn^\nn$ (lying between $[\mathfrak{C}_T]^1_1$ and $[\mathfrak{C}_T]^\omega_1$) listed as follows.
\begin{enumerate}
\item
\index{LEVEL 1 separation}%
There is a function $\Gamma\in \mathfrak{F}\setminus \mathfrak{G}$.
\item
\index{LEVEL 2 separation}%
There are sets $X,Y\subseteq\nn^\nn$ such that $\mathfrak{F}$ has a function $\Gamma_\mathfrak{F}:X\to Y$, but $\mathfrak{G}$ has {\em no} function $\Gamma_\mathfrak{G}:X\to Y$.
\item
\index{LEVEL 3 separation}%
There are $\Pi^0_1$ sets $X,Y\subseteq 2^\nn$ such that $\mathfrak{F}$ has a function $\Gamma_\mathfrak{F}:X\to Y$, but $\Gamma_\mathfrak{G}$ has {\em no} function $\Gamma_\mathfrak{G}:X\to Y$.
\item
\index{LEVEL 4 separation}%
For every special $\Pi^0_1$ set $Y\subseteq 2^\nn$, there is a $\Pi^0_1$ set $X\subseteq 2^\nn$ such that $\mathfrak{F}$ has a function $\Gamma_\mathfrak{F}:X\to Y$, but $\mathfrak{G}$ has {\em no} function $\Gamma_\mathfrak{G}:X\to Y$.
\end{enumerate}

The LEVEL 1 separation just represents $\mathfrak{F}\not\subseteq\mathfrak{G}$.
Clearly, $4\rightarrow 3\rightarrow 2\rightarrow 1$.
Note that the LEVEL 2 separation holds for {\em no} $\Sigma^0_1$ sets $X,Y\subseteq\nn^\nn$, since $\Pi^0_1$ is the first level in the arithmetical hierarchy which can define a nonempty set $S\subseteq\nn^\nn$ without computable element.
Such a $\Pi^0_1$ set is called {\em special}, i.e., a subset of Baire space is special if it is nonempty and contains no computable points.
\index{Pi01 set@$\Pi^0_1$ set!special}
As mentioned before, Simpson-Slaman \cite{SiSl} (see Cole-Simpson \cite{CoSi}) showed that the {\rm LEVEL} $4$ separation holds between $[\mathfrak{C}_T]^1_1$ and $[\mathfrak{C}_T]^\omega_1$, that is, every nonzero Muchnik degree $\dg{a}\in\mathcal{P}^\omega_1$ contains infinitely many Medvedev degrees $\dg{b}\in\mathcal{P}^1_1$. 

\medskip

In section $2$, we use the consistent two-tape disjunction operations on $\Pi^0_1$ subsets of Cantor space introduced in Part I \cite{HK_PartI} to obtain LEVEL 3 separation results.
\begin{itemize}
\item
\index{$\tie_{n}$}%
$\tie_{n}$ is the disjunction operation on $\Pi^0_1$ sets induced by the two-tape BHK-interpretation with mind-changes $<n$.
\item
\index{$\tie_\omega$}%
 $\tie_\omega$ is the disjunction operation on $\Pi^0_1$ sets induced by the two-tape BHK-interpretation with finitely many mind-changes.
\item
\index{$\tie_\infty$}%
 $\tie_\infty$ is the disjunction operation on $\Pi^0_1$ sets induced by the two-tape BHK-interpretation permitting unbounded mind-changes.
\end{itemize}

By using these operations, we obtain the LEVEL 3 separations results for $[\mathfrak{C}_T]^1_1$, $[\mathfrak{C}_T]^1_{<\omega}$, $[\mathfrak{C}_T]^1_{\omega|<\omega}$, and $[\mathfrak{C}_T]^{<\omega}_1$.
We show that there exist $\Pi^0_1$ sets $P,Q\subseteq 2^\nn$ such that all of the following conditions are satisfied.
\begin{enumerate}
\item
\begin{enumerate}
\item There is {\em no} computable function $\Gamma^1_1:P\tie_2 Q\to P\tie_1 Q$;
\item There is a function $\Gamma^1_{<\omega}:P\tie_2 Q\to P\tie_1 Q$ learnable with bounded mind-changes.
\end{enumerate}
\item
\begin{enumerate}
\item There is {\em no} function $\Gamma^1_{<\omega}:P\lcm Q\to P\tie_1 Q$ learnable with bounded mind-changes;
\item There is a function $\Gamma^1_{\omega|<\omega}:P\lcm Q\to P\tie_1 Q$ learnable with bounded errors.
\end{enumerate}
\item
\begin{enumerate}
\item There is {\em no} function $\Gamma^1_{\omega|<\omega}:P\cls Q\to P\tie_1 Q$ learnable with bounded errors;
\item There is a $2$-wise computable function $\Gamma^{<\omega}_1:P\cls Q\to P\tie_1 Q$.
\end{enumerate}
\end{enumerate}
The above conditions also suggest how does degrees of difficulty of our disjunction operations behave.

\medskip

In contrast to the above results, in section $3$, we will see that the hierarchy between $[\mathfrak{C}_T]^1_{<\omega}$ and $[\mathfrak{C}_T]^{<\omega}_{1}$ collapses for {\em homogeneous} $\Pi^0_1$ subsets of Cantor space $2^\nn$.
In other words, the LEVEL 4 separations {\em fail} for $[\mathfrak{C}_T]^1_{<\omega}$, $[\mathfrak{C}_T]^1_{\omega|<\omega}$, and $[\mathfrak{C}_T]^{<\omega}_1$.
For other classes, is the LEVEL 4 separation successful?

To archive the LEVEL 4 separations, we use dynamic disjunction operations developed in Part I \cite{HK_PartI}.
\begin{enumerate}
\item The concatenation $P\mapsto P\fr P$ of two $\Pi^0_1$ sets $P\subseteq 2^\nn$ indicates the mass problem ``solve $P$ by a learning proof process with mind-change-bound $2$''.
\item Every iterated concatenation along a well-founded tree indicates a learning proof process with an ordinal bounded mind changes.
\item The hyperconcatenation $P\mapsto P\htie P$ of two $\Pi^0_1$ sets $P\subseteq 2^\nn$ is defined as the iterated concatenation of $P$ along the corresponding ill-founded tree of $P$.
\end{enumerate}

These operations turn out to be extremely useful to establish the LEVEL 4 separation results.
Some of these results will be proved by applying priority argument {\em inside} some learning proof model of $P$.
\begin{enumerate}
\item The LEVEL 4 separation succeeds for $[\mathfrak{C}_T]^1_1$ and $[\mathfrak{C}_T]^1_{<\omega}$, via the map $P\mapsto P\fr P$.
\item The LEVEL 4 separation succeeds for $[\mathfrak{C}_T]^{<\omega}_1$ and $[\mathfrak{C}_T]^1_{\omega}$, via the map
\[P\mapsto\bigcup_{m\in\nn}(P\fr P\fr\dots\text{($m$ times)}\dots\fr P\fr P).\]
\item The LEVEL 4 separation succeeds for $[\mathfrak{C}_T]^1_\omega$ and $[\mathfrak{C}_T]^{<\omega}_{\omega}$, via the map $P\mapsto P\htie P$.
\item The LEVEL 4 separation succeeds for $[\mathfrak{C}_T]^{<\omega}_\omega$ and $[\mathfrak{C}_T]^\omega_1$, via the map $P\mapsto\widehat{\rm Deg}(P)$, where $\widehat{\rm Deg}(P)$ denotes the Turing upward closure of $P$.
\end{enumerate}

The method that we use to show the first and the third items also implies that any nonzero $\dg{a}\in\mathcal{P}^1_1$ and $\dg{a}\in\mathcal{P}^1_\omega$ have {\em the strong anticupping property}, i.e., for every nonzero $\dg{a}\in\mathcal{P}$, there is a nonzero $\dg{b}\in\mathcal{P}$ below $\dg{a}$ such that $\dg{a}\leq\dg{b}\vee\dg{c}$ implies $\dg{a}\leq\dg{c}$.
\index{strong anticupping property}%
Indeed, these strong anticupping results are established via concatenation $\fr$ and hyperconcatenation $\htie$.
\begin{enumerate}
\item $\mathcal{P}^1_1\models(\forall \dg{a},\dg{c})\;(\dg{a}\leq(\dg{a}\fr\dg{a})\vee\dg{c}\;\rightarrow\;\dg{a}\leq\dg{c})$.
\item $\mathcal{P}^1_\omega\models(\forall \dg{a},\dg{c})\;(\dg{a}\leq(\dg{a}\htie\dg{a})\vee\dg{c}\;\rightarrow\;\dg{a}\leq\dg{c})$.
\end{enumerate}

\medskip

In section 5, we apply our results to sharpen Jockusch's theorem \cite{Joc} and Simpson's Embedding Lemma \cite{Sim2}.
Jockusch showed the following nonuniform computability result for ${\sf DNR}_k$, the set of all $k$-valued {\em diagonally noncomputable functions}.
\begin{enumerate}
\item There is {\em no} (uniformly) computable function $\Gamma^1_1:{\sf DNR}_3\to{\sf DNR}_2$.
\item There is a nonuniformly computable function $\Gamma^\omega_1:{\sf DNR}_3\to{\sf DNR}_2$.
\end{enumerate}

This result will be sharpened by using our learnability notions as follows.

\begin{enumerate}
\item There is {\em no} learnable function $\Gamma^1_\omega:{\sf DNR}_3\to{\sf DNR}_2$.
\item There is {\em no} $k$-wise computable function $\Gamma^{<\omega}_1:{\sf DNR}_3\to{\sf DNR}_2$ for $k\in\nn$.
\item There is a (uniformly) computable function $\Gamma^1_1:{\sf DNR}_3\to{\sf DNR}_2\htie{\sf DNR}_2$.
Hence, there is a function $\Gamma^{<\omega}_\omega:{\sf DNR}_3\to{\sf DNR}_2$ learnable by a team of two learners.
\end{enumerate}

Finally, we employ concatenation and hyperconcatenation operations to show that neither $\mathcal{D}^1_{<\omega}$ nor $\mathcal{D}^{<\omega}_{1}$ nor $\mathcal{D}^{<\omega}_{\omega}$ are Brouwerian.
Hence, these degree structures are not elementarily equivalent to the Medvedev (Muchnik) degree structure.

\subsection{Notations and Conventions}

For any sets $X$ and $Y$, we say that {\em $f$ is a function from $X$ to $Y$} (written $f:X\to Y$) if the domain ${\rm dom}(f)$ of $f$ includes $X$, and the range ${\rm range}(f)$ of $f$ is included in $Y$.
We also use the notation $f:\subseteq X\to Y$ to denote that $f$ is a partial function from $X$ to $Y$, i.e., the domain ${\rm dom}(f)$ of $f$ is included in $X$, and the range ${\rm rng}(f)$ of $f$ is also included in $Y$.

\index{$\sigma^-$}\index{tree!immediate predecessor in}\index{$[\sigma]$}\index{tree}\index{$[T]$}%
\index{tree!extendible in}\index{$T^{ext}$}\index{tree!leaf of}\index{tree!dead end of}%
For basic terminology in Computability Theory, see Soare \cite{Soa}.
For $\sigma\in\nn^{<\nn}$, we let $|\sigma|$ denote the length of $\sigma$.
For $\sigma\in\nn^{<\nn}$ and $f\in\nn^{<\nn}\cup\nn^\nn$, we say that $\sigma$ is {\em an initial segment} of $f$ (denoted by $\sigma\subset f$) if $\sigma(n)=f(n)$ for each $n<|\sigma|$.
Moreover, $f\res n$ denotes the unique initial segment of $f$ of length $n$.
let $\sigma^-$ denote an immediate predecessor node of $\sigma$, i.e. $\sigma^-=\sigma\res (|\sigma|-1)$.
We also define $[\sigma]=\{f\in \nn^\nn:f\supset\sigma\}$.
A {\em tree} is a subset of $\nn^{<\nn}$ closed under taking initial segments.
For any tree $T\subseteq \nn^{<\nn}$, we also let $[T]$ be the set of all infinite paths of $T$, i.e., $f$ belongs to $[T]$ if $f\res n$ belongs to $T$ for each $n\in\nn$. 
A node $\sigma\in T$ is {\em extendible} if $[T]\cap[\sigma]\not=\emptyset$.
Let $T^{ext}$ denote the set of all extendible nodes of $T$.
We say that $\sigma\in T$ is {\em a leaf} or {\em a dead end} if there is no $\tau\in T$ with $\tau\supsetneq\sigma$.

\index{$X^{<\nn}$}\index{concatenation}\index{$\fr$}\index{$\bigsqcap$}%
For any set $X$, the tree $X^{<\nn}$ of finite words on $X$ forms a monoid under concatenation $\fr$.
Here {\em the concatenation of $\sigma$ and $\tau$} is defined by $(\sigma\fr \tau)(n)=\sigma(n)$ for $n<|\sigma|$ and $(\sigma\fr\tau)(|\sigma|+n)=\tau(n)$ for $n<|\tau|$.
We use symbols $\fr$ and $\bigsqcap$ for the operation on this monoid, where $\bigsqcap_{i\leq n}\sigma_i$ denotes $\sigma_0\fr\sigma_1\fr\dots\fr\sigma_n$.
To avoid confusion, the symbols $\times$ and $\prod$ are only used for a product of sets.
We often consider the following three left monoid actions of $X^{<\nn}$:
The first one is the set $X^\nn$ of infinite words on $X$ with an operation $\fr:X^{<\nn}\times X^\nn\to X^\nn$; $(\sigma\fr f)(n)=\sigma(n)$ for $n<|\sigma|$ and $(\sigma\fr f)(|\sigma|+n)=f(n)$ for $n\in\nn$.
The second one is the set $\mathcal{T}(X)$ of subtrees $T\subseteq X^{<\nn}$ with an operation $\fr:X^{<\nn}\times\mathcal{T}(X)\to\mathcal{T}(X)$; $\sigma\fr T=\{\sigma\fr\tau:\tau\in T\}$.
The third one is the power set $\mathcal{P}(X^\nn)$ of $X^{\nn}$ with an operation $\fr:X^{<\nn}\times\mathcal{P}(X^\nn)\to\mathcal{P}(X^\nn)$; $\sigma\fr P=\{\sigma\fr f:f\in P\}$.

\index{Pi01 set@$\Pi^0_1$ set}\index{Pi01 set@$\Pi^0_1$ set!effective enumeration of@effective enumeration of}\index{Pi01 set@$\Pi^0_1$ set!corresponding tree of@corresponding tree of}\index{$\Phi(\sigma;n)$}%
\index{Pi01 set@$\Pi^0_1$ set!computable sequence of}\index{Pi01 set@$\Pi^0_1$ set!uniform sequence of}%
\index{Pi01 set@$\Pi^0_1$ set!special}\index{$f\oplus g$}\index{$P\linf Q$}\index{$P\lsup Q$}%
We say that a set $P\subseteq\nn^\nn$ is $\Pi^0_1$ if there is a computable relation $R$ such that $P=\{f\in\nn^\nn:(\forall n)R(n,f)\}$ holds.
Equivalently, $P=[T_P]$ for some computable tree $T_P\subseteq\nn^\nn$.
Let $\{\Phi_e\}_{e\in\nn}$ be an effective enumeration of all Turing functionals (all partial computable functions\footnote{In some context, a function $\Phi$ is sometimes called partial computable if it can be extended to some $\Phi_e$. In this paper, however, we do not need to distinguish our definition as being different from this definition.}) on $\nn^\nn$.
Then the $e$-th $\Pi^0_1$ subset of $2^\nn$ is defined by $P_e=\{f\in 2^\nn:\Phi_e(f;0)\uparrow\}$.
Note that $\{P_e\}_{e\in\nn}$ is an effective enumeration of all $\Pi^0_1$ subsets of Cantor space $2^\nn$.
If (an index $e$ of) a $\Pi^0_1$ set $P_e\subseteq 2^\nn$ is given, then $T_e=\{\sigma\in 2^{<\nn}:\Phi_e(\sigma;0)\uparrow\}$ is called {\em the corresponding tree for $P_e$}.
Here $\Phi(\sigma;n)$ for $\sigma\in\nn^{<\nn}$ and $n\in\nn$ denotes the computation of $\Phi$ with an oracle $\sigma$, an input $n$, and step $|\sigma|$.
Whenever a $\Pi^0_1$ set $P$ is given, we assume that an index $e$ of $P$ is also given.
If $P\subseteq 2^\nn$ is $\Pi^0_1$, then the corresponding tree $T_P\subseteq 2^{<\nn}$ of $P$ is computable, and $[T_P]=P$.
Moreover, the set $L_P$ of all leaves of the computable tree $T_P$ is also computable.
We also say that a sequence of $\{P_i\}_{i\in I}$ of $\Pi^0_1$ subsets of a space $X$ is {\em computable} or {\em uniform} if the set $\{(i,f)\in I\times X:f\in P_i\}$ is again a $\Pi^0_1$ subset of the product space $I\times X$.
A set $P\subseteq\nn^\nn$ is {\em special} if $P$ is nonempty and $P$ has no computable member.
For $f,g\in\nn^\nn$, $f\oplus g$ is defined by $(f\oplus g)(2n)=f(n)$ and $(f\oplus g)(2n+1)=g(n)$ for each $n\in\nn$.
For $P,Q\subseteq\nn^\nn$, put $P\linf Q=(\lrangle{0}\fr P)\cup(\lrangle{1}\fr Q)$ and $P\lsup Q=\{f\oplus g:f\in P\;\&\;g\in Q\}$.

\subsection{Notations from Part I}

\subsubsection{Functions}

Every partial function $\Psi:\subseteq\nn^{<\nn}\to\nn$ is called a {\em learner}.
\index{learner}%
In Part II \cite[Proposition 1]{HK_PartI}, it is shown that we may assume that $\Psi$ is total, and we fix an effective enumeration $\{\Psi_e\}_{e\in\nn}$ of all learners.
For any string $\sigma\in\nn^{<\nn}$, the set of {\em mind-change locations of a learner $\Psi$ on the informant $\sigma$} is defined by
\index{learner!mind-change location of}%
\index{${\tt mcl}_\Psi$}%
\[{\tt mcl}_\Psi(\sigma)=\{n<|\sigma|:\Psi(\sigma\res n+1)\not=\Psi(\sigma\res n)\}.\]
We also define ${\tt mcl}_\Psi(f)=\bigcup_{n\in\nn}{\tt mcl}_\Psi(f\res n)$ for any $f\in\nn^{\nn}$.
Then, $\#{\tt mcl}_\Psi(f)$ denotes the {\em number of times that the learner $\Psi$ changes her/his mind on the informant $f$}.
\index{$\#{\tt mcl}_\Psi$}%
Moreover, the set of {\em indices predicted by a learner $\Psi$ on the informant $\sigma$} is defined by
\index{${\tt indx}_\Psi$}%
\[{\tt indx}_\Psi(\sigma)=\{\Psi(\sigma\res n):n\leq|\sigma|\}.\]
We also define ${\tt indx}_\Psi(f)=\bigcup_{n\in\nn}{\tt indx}_\Psi(f\res n)$ for any $f\in\nn^{\nn}$.
We say that {\em a partial function $\Gamma:\subseteq\nn^\nn\to\nn^\nn$ is identified by a learner $\Psi$ on $g\in\nn^\nn$} if $\lim_n\Psi_e(g\res n)$ converges, and $\Phi_{\lim_n\Psi_e(g\res n)}(g)=\Gamma(g)$.
We also say that a partial function $\Gamma$ is identified by a learner $\Psi$ if it is identified by $\Psi$ on every $g\in{\rm dom}(\Gamma)$.
In Part I \cite[Definition 2]{HK_PartI}, we introduced the seven notions of $(\alpha,\beta|\gamma)$-computability for a partial function $\Gamma:\subseteq\nn^\nn\to\nn^\nn$ listed as follows:
\index{computability!$(\alpha,\beta\mid\gamma)$-}%

\begin{enumerate}
\item $\Gamma$ is {\em $(1,1)$-computable} if it is computable.
\item $\Gamma$ is {\em $(1,<\omega)$-computable} if it is identified by a learner $\Psi$ with $\sup\{\#{\tt mcl}_\Psi(g):g\in{\rm dom}(\Gamma)\}<\omega$.
\item $\Gamma$ is {\em $(1,\omega|<\omega)$-computable} if it is identified by a learner $\Psi$ with $\sup\{\#{\tt indx}_\Psi(g):g\in{\rm dom}(\Gamma)\}<\omega$.
\item $\Gamma$ is {\em $(1,\omega)$-computable} if it is identified by a learner.
\item $\Gamma$ is {\em $(<\omega,1)$-computable} if there is $b\in\nn$ such that for every $g\in{\rm dom}(\Gamma)$, $\Gamma(g)=\Phi_e(g)$ for some $e<b$.
\item $\Gamma$ is {\em $(<\omega,\omega)$-computable} if there is $b\in\nn$ such that for every $g\in{\rm dom}(\Gamma)$, $\Gamma$ is identified by $\Psi_e$ for some $e<b$ on $g$.
\item $\Gamma$ is {\em $(\omega,1)$-computable} if it is nonuniformly computable, i.e., $\Gamma(g)\leq_Tg$ for every $g\in{\rm dom}(\Gamma)$.
\end{enumerate}

\index{$[\mathfrak{C}_T]^\alpha_{\beta}$}\index{$[\mathfrak{C}_T]^\alpha_{\beta\mid\gamma}$}%
Let $[\mathfrak{C}_T]^\alpha_{\beta}$ (resp.\ $[\mathfrak{C}_T]^\alpha_{\beta|\gamma}$) denote the set of all $(\alpha,\beta)$-computable (resp.\ $(\alpha,\beta|\gamma)$-computable) functions.
If $\mathcal{F}$ be a monoid consisting of partial functions under composition, $\mathcal{P}(\nn^\nn)$ is preordered by the relation $P\leq_\mathcal{F}Q$ indicating the existence of a function $\Gamma\in\mathcal{F}$ from $Q$ into $P$, that is, $P\leq_\mathcal{F}Q$ if and only if there is a partial function $\Gamma:\subseteq\nn^\nn\to\nn^\nn$ such that $\Gamma\in\mathcal{F}$ and $\Gamma(g)\in P$ for every $g\in Q$.
\index{$P\leq_\mathcal{F}Q$}%
Let $\mathcal{D}/\mathcal{F}$ and $\mathcal{P}/\mathcal{F}$ denote the quotient sets $\mathcal{P}(\nn^\nn)/\equiv_\mathcal{F}$ and $\Pi^0_1(2^\nn)/\equiv_\mathcal{F}$, respectively.
\index{$\mathcal{D}/\mathcal{F}$}\index{$\mathcal{P}/\mathcal{F}$}%
Here, $\Pi^0_1(2^\nn)$ denotes the set of all nonempty $\Pi^0_1$ subsets of $2^\nn$.
For $P\in\mathcal{P}(\nn^\nn)$, the equivalence class $\{Q\subseteq\nn^\nn:Q\equiv_\mathcal{F}P\}\in\mathcal{D}/\mathcal{F}$ is called {\em the $\mathcal{F}$-degree} of $P$.
\index{degree}%
If $\mathcal{F}=[\mathfrak{C}_T]^\alpha_{\beta|\gamma}$ for some $\alpha,\beta,\gamma\in\{1,<\omega,\omega\}$, we write $\leq^\alpha_{\beta|\gamma}$, $\mathcal{D}^\alpha_{\beta|\gamma}$, and $\mathcal{P}^\alpha_{\beta|\gamma}$ instead of $\leq_\mathcal{F}$, $\mathcal{D}/\mathcal{F}$ and $\mathcal{P}/\mathcal{F}$.
\index{$\leq^\alpha_{\beta\mid\gamma}$}\index{$\mathcal{D}^\alpha_{\beta\mid\gamma}$}%
\index{$\mathcal{P}^\alpha_{\beta\mid\gamma}$}%
The preorderings $\leq^1_1$ and $\leq^\omega_1$ are equivalent to the Medvedev reducibility \cite{Med} and the Muchnik reducibility \cite{Muc}, respectively.
In Part I \cite[Theorem 26 and Proposition 27]{HK_PartI}, we showed the following equivalences:
\begin{align*}
&\mathcal{P}^1_{<\omega}=\mathcal{P}/{\rm dec}^{<\omega}_{\rm d}[\Pi^0_1] & 
&\mathcal{P}^1_{\omega|<\omega}=\mathcal{P}/{\rm dec}^{<\omega}_{\rm p}[\Delta^0_2] &
&\mathcal{P}^1_{\omega}=\mathcal{P}/{\rm dec}^{\omega}_{\rm p}[\Delta^0_2] \\
&\mathcal{P}^{<\omega}_{1}=\mathcal{P}/{\rm dec}^{<\omega}_{\rm d}[\Pi^0_2] &
&\mathcal{P}^{<\omega}_{\omega}=\mathcal{P}/{\rm dec}^{<\omega}_{\rm d}[\Pi^0_2]{\rm dec}^{\omega}_{\rm p}[\Delta^0_2] &
&\mathcal{P}^{\omega}_{1}=\mathcal{P}/{\rm dec}^{\omega}_{\rm p}[\Pi^0_2]
\end{align*}

\subsubsection{Sets}

To define the disjunction operations in Part I \cite[Definition 29]{HK_PartI}, we introduced some auxiliary notions.
Let $I\subseteq\nn$ be a set.
Fix $\sigma\in(I\times\nn)^{<\nn}$, and $i\in I$.
Then {\em the $i$-th projection of $\sigma$} is inductively defined as follows.
\index{i-th projection@$i$-th projection}\index{${\tt pr}_i(\sigma)$}%
\begin{align*}
{\tt pr}_i(\lrangle{})=\lrangle{},& & {\tt pr}_i(\sigma)=
\begin{cases}
{\tt pr}_i(\sigma^-)\fr n, \mbox{ if } \sigma=\sigma^-\fr\lrangle{\pair{i,n}},\\
{\tt pr}_i(\sigma^-), \mbox{ otherwise.}
\end{cases}
\end{align*}
Moreover, {\em the number of times of mind-changes of (the process reconstructed from a record) $\sigma\in (I\times\nn)^{<\nn}$} is given by
\index{number of times of mind-changes}\index{${\tt mc}$}%
\[{\tt mc}(\sigma)=\#\{n<|\sigma|-1:(\sigma(n))_0\not=(\sigma(n+1))_0\}.\]
Here, for $x=\pair{x_0,x_1}\in I\times\nn$, the first (second, resp.) coordinate $x_0$ ($x_1$, resp.) is denoted by $(x)_0$ ($(x)_1$, resp.).
Furthermore, for $f\in(I\times\nn)^{\nn}$, we define ${\tt pr}_i(f)=\bigcup_{n\in\nn}{\tt pr}_i(f\res n)$ for each $i\in I$, and ${\tt mc}(f)=\lim_n{\tt mc}(f\res n)$, where if the limit does not exist, we write ${\tt mc}(f)=\infty$.

In Part I \cite[Definition 33, 36 and 55]{HK_PartI}, we introduced the disjunction operations.
Fix a collection $\{P_i\}_{i\in I}$ of subsets of Baire space $\nn^\nn$.
\begin{enumerate}
\item
\index{$\bhk{\bigvee_{i\in I}P_i}_{\sf Int}$}%
$\bhk{\bigvee_{i\in I}P_i}_{\sf Int}=\{f\in(I\times\nn)^\nn:((\exists i\in I)\;{\tt pr}_i(f)\in P_i)\;\&\;{\tt mc}(f)=0\}$.
\item
\index{$\bhk{\bigvee_{i\in I}P_i}_{\sf LCM[n]}$}%
 $\bhk{\bigvee_{i\in I}P_i}_{\sf LCM[n]}=\{f\in(I\times\nn)^\nn:((\exists i\in I)\;{\tt pr}_i(f)\in P_i)\;\&\;{\tt mc}(f)<n\}$.
\item
\index{$\bhk{\bigvee_{i\in I}P_i}_{\sf CL}$}%
 $\bhk{\bigvee_{i\in I}P_i}_{\sf CL}=\{f\in(I\times\nn)^\nn:(\exists i\in I)\;{\tt pr}_i(f)\in P_i\}$.
\end{enumerate}

As in Part I, we use the notation ${\tt write}(i,\sigma)$ for any $i\in\nn$ and $\sigma\in\nn^{<\nn}$.
\index{${\tt write}$}%
\[{\tt write}(i,\sigma)=i^{|\sigma|}\oplus\sigma=\lrangle{\pair{i,\sigma(0)},\pair{i,\sigma(1)},\pair{i,\sigma(2)},\dots,\pair{i,\sigma(|\sigma|-1)}}.\]
This string indicates the {\em instruction to write the string $\sigma$ on the $i$-th tape} in the one/two-tape model.
We also use the notation ${\tt write}(i,f)=\bigcup_{n\in\nn}{\tt write}(i,f\res n)=i^\nn\oplus f$ for any $f\in\nn^{\nn}$.

In Part II, we are mostly interested in the degree structures of $\Pi^0_1$ subsets of $2^\nn$.
As mentioned in Part I \cite{HK_PartI}, the consistent disjunction operations are useful to study such local degree structures.
{\em The consistency set ${\rm Con}(T_i)_{i\in I}$ for a collection $\{T_i\}_{i\in I}$ of trees} is defined as follows.
\index{consistency set}\index{${\rm Con}(T_i)_{i\in I}$}%
\[{\rm Con}(T_i)_{i\in I}=\{f\in(I\times\nn)^\nn:(\forall i\in I)(\forall n\in\nn)\;{\tt pr}_i(f\res n)\in T_i\}.\]
Then we use the following modified definitions.
Fix a collection $\{P_i\}_{i\in I}$ of $\Pi^0_1$ subsets of Baire space $\nn^\nn$ and $n\in\omega\cup\{\omega\}$.
\begin{enumerate}
\item
\index{$\left[\btie_n\right]_{i\in I}P_i$}%
$\left[\btie_n\right]_{i\in I}P_i=\bhk{\bigvee_{i\in I}P_i}_{{\sf LCM}[n]}\cap{\rm Con}(T_{P_i})_{i\in I}$.
\item
\index{$\left[\bcls\right]_{i\in I}P_i$}%
$\left[\bcls\right]_{i\in I}P_i=\bhk{\bigvee_{i\in I}P_i}_{\sf CL}\cap{\rm Con}(T_{P_i})_{i\in I}$.
\end{enumerate}
Here $T_{P_i}$ is the corresponding tree for $P_i$ for every $i\in I$.
If $i\in\{0,1\}$, then we simply write $P_0\tie_nP_1$, $P_0\tie_\omega P_1$, and $P_0\tie_\infty P_1$ for these notions.
\index{$P_0\tie_nP_1$}\index{$P_0\tie_\omega P_1$}\index{$P_0\tie_\infty P_1$}%
In Part II, we use the following notion.

\begin{definition}
Pick any $*\in\nn\cup\{{\omega}\}\cup\{\infty\}$.
For each disjunctive notions $\tie_*$ and collection $\{P_i\}_{i\in I}$ of subsets of $\nn^\nn$, fix the corresponding tree $T_{P_i}\subseteq\nn^{<\nn}$ of $P_i$ for every $i\in I$ and we may also associate a tree $T_*$ with (the closure of) $P_0\tie_*P_1$.
Then {\em the heart of $P_0\tie_*P_1$} is defined by $T_*^\heartsuit=\{\sigma\in T_*:(\forall i\in I)\;{\tt pr}_i(\sigma)\in T_{P_i}^{ext}\}$.
\index{heart}\index{$T_*^\heartsuit$}%
\end{definition}

Note that every $\sigma\in T_*^\heartsuit$ is extendible in $T_*$, since $T_*^\heartsuit\subseteq\{\sigma\in T_*:(\exists i\in I)\;{\tt pr}_i(\sigma)\in T_{P_i}^{ext}\}$.

Let $L_P$ denote the set of all leaves of the corresponding tree for a nonempty $\Pi^0_1$ set $P$ (where recall that such a tree is assumed to be uniquely determined when an index of $P$ is given).
Then {\em the (non-commutative) concatenation of $P$ and $Q$} is defined as follows.
\index{concatenation}\index{concatenation!non-commutative}\index{$P\ntie Q$}%
\[P\ntie Q=P\cup\bigcup_{\rho\in L_P}\rho\fr Q.\]
Moreover, the commutative concatenation $P\tie Q$ is defined as $(P\ntie Q)\oplus(Q\ntie P)$.
\index{concatenation!commutative}\index{$P\tie Q$}%
Let $P$ and $\{Q_n\}_{n\in\nn}$ be computable collection of $\Pi^0_1$ subsets of $2^\nn$, and let $\rho_n$ denote the length-lexicographically $n$-th leaf of the corresponding computable tree of $P$.
Then, we define the {\em infinitary concatenation} and {\em recursive meet} \cite{BS} as follows:
\index{$P\ntie\{Q_i\}_{i\in\nn}$}\index{concatenation!infinitary}%
\index{$\cmeet_{i\in\nn}Q_i$}\index{recursive meet}%
\begin{align*}
P\ntie\{Q_i\}_{i\in\nn}=P\cup\bigcup_n\rho_n\fr Q_n,& &\cmeet_{i\in\nn}Q_i={\sf CPA}\ntie\{Q_i\}_{i\in\nn}.
\end{align*}
Here, recall that ${\sf CPA}$ is a Medvedev complete set, which consists of all {\em complete consistent extensions of Peano Arithmetic}.
\index{${\sf CPA}$}%
The Medvedev completeness of ${\sf CPA}$ ensures that for any nonempty $\Pi^0_1$ subset $P\subseteq 2^\nn$, a computable function $\Phi:{\sf CPA}\to P$ exists.

In Part I, we studied the disjunction and concatenation operations along graphs.
For nonempty $\Pi^0_1$ subsets $P$ and $Q$ of $2^\nn$, the {\em hyperconcatenation $Q\htie P$ of $Q$ and $P$} is defined by the iterated concatenation of $P$'s along the ill-founded tree $T_Q$, that is,
\index{concatenation!hyper-}\index{hyperconcatenation}\index{$Q\htie P$}%
\[Q\htie P=\left[\bigcup_{\tau\in T_Q}\left(\concat_{i<|\tau|}T_P\fr\lrangle{\tau(i)}\right)\fr T_P\right].\]

%% file: NRMP_fullproof2.tex
\section{Degrees of Difficulty of Disjunctions}

The main objective in this section is to establish {\rm LEVEL} $3$ separation results among our classes of nonuniformly computable functions by using disjunction operations introduced in Part I \cite{HK_PartI}.
We have already seen the following inequalities for $\Pi^0_1$ subsets $P,Q\subseteq 2^\nn$ in Part I \cite{HK_PartI}.
\[P\oplus Q\geq^1_1P\cup Q\geq^1_1P\tie Q\geq^1_1P\lcm Q\geq^1_1P\cls Q.\]

As observed in Part I \cite{HK_PartI}, these binary disjunctions are closely related to the reducibilities $\leq^1_1$, $\leq^{<\omega}_{tt,1}$, $\leq^1_{<\omega}$, $\leq^1_{\omega|<\omega}$, and $\leq^{<\omega}_1$, respectively.
We employ rather exotic $\Pi^0_1$ sets constructed by Jockusch and Soare to separate the strength of these disjunctions.
We say that a set $A\subseteq\nn^\nn$ is {\em an antichain} if it is an antichain with respect to the Turing reducibility $\leq_T$.
\index{Pi01 set@$\Pi^0_1$ set!antichain}%
In other words, $f$ is Turing incomparable with $g$, for any two distinct elements $f,g\in A$.
A nonempty closed set $A\subseteq\nn^\nn$ is {\em perfect} if it has no isolated point.
\index{Pi01 set@$\Pi^0_1$ set!perfect}%

\begin{theorem}[Jockusch-Soare \cite{JS}]\label{thm:JS}
There exists a perfect $\Pi^0_1$ antichain in $2^\nn$.
\end{theorem}

A stronger condition is sometimes required.
For a set $P\subseteq\nn^\nn$ and an element $g\in\nn^\nn$, let $P^{\leq_Tg}$ denote the set of all element of $P$ which are Turing reducible to $g$.
\index{$P^{\leq_Tg}$}%
Then, a set $A\subseteq\nn^\nn$ is antichain if and only if $A^{\leq_Tg}=\{g\}$ for every $g\in A$.
A set $P\subseteq\nn^\nn$ is {\em independent} if $P^{\leq_T\bigoplus D}=D$ for every finite subset $D\subset P$.
\index{Pi01 set@$\Pi^0_1$ set!independent}%

\begin{theorem}[see Binns-Simpson \cite{BS}]
There exists a perfect independent $\Pi^0_1$ subset of $2^\nn$.
\end{theorem}

On the other hand, in the later section, we will see that our hierarchy of disjunctions collapses for homogeneous sets, which may be regarded as an opposite notion to antichains and independent sets.

\subsection{The Disjunction $\linf$ versus the Disjunction $\cup$}

We first separate the strength of the coproduct (the intuitionistic disjunction) $\linf$ and the union (the classical one-tape disjunction) $\cup$.
This automatically establish the LEVEL $3$ separation result between $[\mathfrak{C}_{T}]^1_1$ and $[\mathfrak{C}_{tt}]^{<\omega}_1$.
Recall that a set $P\subseteq\nn^\nn$ is {\em special} if it is nonempty and it contains no computable points.
\index{Pi01 set@$\Pi^0_1$ set!special}%

\begin{lemma}\label{lem:2:separ}
Let $P_0,P_1$ be $\Pi^0_1$ subsets of $2^\nn$, and let $Q$ be a special $\Pi^0_1$ subset of $2^\nn$.
Assume that there exist $f\in P_0$ and $g\in P_1$ with $Q^{\leq_Tf\oplus g}=Q^{\leq_Tf}\cup Q^{\leq_Tg}$ such that $Q^{\leq_Tf}$ and $Q^{\leq_Tg}$ are separated by open sets.
Then $Q\not\leq^{1}_{1}(P_0\lsup 2^\nn)\cup(2^\nn\lsup P_1)$.
\end{lemma}

\begin{proof}\upshape
Suppose that $Q\leq^{1}_{1}(P_0\lsup 2^\nn)\cup(2^\nn\lsup P_1)$ via a computable functional $\Phi$.
Then $f\oplus g\in(P_0\lsup 2^\nn)\cup(2^\nn\lsup P_1)$.
By our choice of $f$ and $g$, $\Phi(f\oplus g)$ must belong to $Q^{\leq_Tf\oplus g}=Q^{\leq_Tf}\cup Q^{\leq_Tg}$.
By our assumption, $Q^{\leq_Tf}$ and $Q^{\leq_Tg}$ are separated by two disjoint open sets $U,V\subseteq 2^\nn$.
That is, $Q^{\leq_Tf}\subseteq U$, $Q^{\leq_Tg}\subseteq V$, and $U\cap V=\emptyset$.
Therefore, either $\Phi(f\oplus g)\in Q\cap U$ or $\Phi(f\oplus g)\in Q\cap V$ holds.
In any case, there exists an open neighborhood $[\sigma]\ni\Phi(f\oplus g)$ such that $[\sigma]\subseteq U$ or $[\sigma]\subseteq V$.
Without loss of generality, we can assume $[\sigma]\subseteq U$.
We pick initial segments $\tau_0\subset f$ and $\tau_1\subset g$ with $\Phi(\tau_0\oplus\tau_1)\supseteq\sigma$.
Then $(\tau_0\fr 0^\nn)\oplus g\in(P_0\lsup 2^\nn)\cup(2^\nn\lsup P_1)$, and it is Turing equivalent to $g$.
However this is impossible because $\Phi(\tau_0\fr 0^\nn\oplus g)\in [\sigma]$, and $[\sigma]\cap Q^{\leq_Tg}\subseteq U\cap Q^{\leq_Tg}=\emptyset$.
\end{proof}

\begin{cor}\label{cor:2:separ}
\begin{enumerate}
\item There are $\Pi^0_1$ sets $P,Q\subseteq 2^\nn$ such that $P\cup Q<^{1}_{1}P\linf Q$.
\item There are $\Pi^0_1$ sets $P,Q\subseteq 2^\nn$ such that $P\equiv^{<\omega}_{tt,1}Q$ and $P<^{1}_{1}Q$.
\end{enumerate}
\end{cor}

\begin{proof}\upshape
(1)
Let $R$ be a perfect independent $\Pi^0_1$ subset of $2^\nn$.
Set $P=2^\nn\lsup R$ and $Q=R\lsup 2^\nn$.
Note that $P\oplus Q\equiv^{1}_{1} R$.
Pick $f,g\in R$ such that $f\not=g$.
Then $R^{\leq_Tf}=\{f\}$, $R^{\leq_Tg}=\{g\}$, and $R^{\leq_Tf\oplus g}=R^{\leq_Tf}\sqcup R^{\leq_Tg}=\{f,g\}$.
Since $2^\nn$ is Hausdorff, two points $f$ and $g$ are separated by open sets.
Thus, $P\oplus Q\equiv^1_1R\not\leq^1_1P\cup Q$ by Lemma \ref{lem:2:separ}.
(2) $P\oplus Q\equiv^{<\omega}_{tt,1}P\cup Q<^1_1P\oplus Q$.
\end{proof}

\begin{remark}
One can adopt the unit interval $[0,1]$ as our whole space instead of Cantor space $2^\nn$.
Then, $P_0\cross P_1:=(P_0\times[0,1])\cup([0,1]\times P_1)$ is connected as a topological space.
\index{$P_0\cross P_1$}%
If $P_0\subseteq[0,1]$ is homeomorphic to Cantor space, then the connected space $P_0\cross P_0$ is sometimes called {\em the Cantor tartan}. 
\index{Cantor tartan}%
The above proof shows that every perfect independent $\Pi^0_1$ set $R\subseteq[0,1]$ is not $(1,1)$-reducible to the obtained tartan $R\cross R$, while these sets are $(<\omega,1)$-$tt$-equivalent. 
Note that the tartan plays an important role on the constructive study of Brouwer's fixed point theorem (see \cite{BMP}).
\end{remark}

\subsection{The Disjunction $\cup$ versus the Disjunction $\tie$}

We next separate the strength of the union $\cup$ and the concatenation (the LCM disjunction with mind-change-bound $2$) $\tie$.
Moreover, we also see the LEVEL $3$ separation between $[\mathfrak{C}_{tt}]^{<\omega}_1$ and $[\mathfrak{C}_{T}]^1_{<\omega}$.

\begin{lemma}\label{lem:NRMP:tiecup}
Let $P_0,P_1$ be $\Pi^0_1$ subsets of $2^\nn$, and let $Q$ be a special $\Pi^0_1$ subset of $2^\nn$.
Assume that there exist $f\in P_0$ and $g\in P_1$ such that any $h\in Q^{\leq_Tf}$ and $Q^{\leq_Tg}$ are separated by open sets.
Then $Q\not\leq^{1}_{1}P_0\fr P_1$.
\end{lemma}

\begin{proof}\upshape
Suppose that $Q\leq^{1}_{1}P_0\fr P_1$ via a computable functional $\Phi$.
By our choice of $f\in P_0\subseteq P_0\ntie P_1$, there must exist an open set $U\subseteq 2^\nn$ such that $\Phi(f)\in Q\cap U$ and $Q^{\leq_Tg}\cap U=\emptyset$.
Since $U$ is open there exists a clopen neighborhood $[\sigma]\ni\Phi(f)$ such that $[\sigma]\cap Q\subseteq U$.
We pick an initial segment $\tau\subset f$ with $\Phi(\tau)\supseteq\sigma$.
Since $f\in P_0$ holds, we have that $\tau\in T_{P_0}$, and we pick $\rho\in L_{P_0}$ extending $\tau$.
Then $\rho\fr g\in P_0\fr P_1$, and $\rho\fr g$ is Turing equivalent to $g$.
So, if $Q\leq^{1}_{1}P_0\fr P_1$ via $\Phi$, then $\Phi(\rho\fr g)$ must belong to $Q^{\leq_Tg}$.
However this is impossible because $\Phi(\rho\fr g)\in [\sigma]$, and $[\sigma]\cap Q^{\leq_Tg}\subseteq U\cap Q^{\leq_Tg}=\emptyset$.
\end{proof}

\begin{cor}
There are $\Pi^0_1$ sets $P,Q\subseteq 2^\nn$ such that $P\fr Q<^{1}_{1}P\cup Q<^{1}_{1}P\linf Q$.
\end{cor}

\begin{proof}\upshape
Assume that $R$ be a perfect $\Pi^0_1$ antichain of $2^\nn$.
Set $P=2^\nn\lsup R$ and $Q=R\lsup 2^\nn$.
Pick $f,g\in R$ such that $f\not=g$.
Then $R^{\leq_Tf}=\{f\}$ and $R^{\leq_Tg}=\{g\}$ since $R$ is antichain.
Therefore, $(P\cup Q)^{\leq_TX}\subseteq(\{X\}\otimes 2^\nn)\cup(2^\nn\otimes\{X\})$ for each $X\in\{f,g\}$.
For $h=h_0\oplus h_1\in (P\cup Q)^{\leq_Tf}$, we have $h_0\not=g$ and $h_1\not=g$.
Thus, $h\not\in(2^\nn\otimes\{g\})\cup(\{g\}\otimes 2^\nn)$, and note that $(2^\nn\otimes\{g\})\cup(\{g\}\otimes 2^\nn)$ is closed.
Then, there is an open neighborhood $U\subseteq 2^\nn$ such that $h\in U$ and $U\cap (P\cup Q)^{\leq_Tg}=\emptyset$, since $P\cup Q$ is regular, and $(P\cup Q)^{\leq_Tg}\subseteq(2^\nn\otimes\{g\})\cup(\{g\}\otimes 2^\nn)$.
Namely, any $h\in(P\cup Q)^{\leq_Tf}$ and $(P\cup Q)^{\leq_Tg}$ are separated by some open set.
Consequently, by Lemma \ref{lem:NRMP:tiecup}, we have $P\cup Q\not\leq^1_1P\fr Q$.
\end{proof}

One can establish another separation result for the concatenation.
Recall from \cite{CKWW} that a closed set $P\subseteq\nn^\nn$ is {\em immune} if $T_P^{ext}$ contains no infinite c.e.~subset.
\index{immune}%
In \cite{CKWW} it is shown that the class of non-immune $\Pi^0_1$ subsets of Cantor space is downward closed in the Medvedev degrees $\mathcal{P}^1_1$.
This property also holds in a coarser degree structure.
In Part I \cite{HK_PartI} we have seen that $\mathcal{P}^{<\omega}_{tt,1}$ is an intermediate structure between $\mathcal{P}^1_1$ and $\mathcal{P}^1_{<\omega}$.

\begin{lemma}\label{lem:NRMP:immuclosurett}
Let $P$ and $Q$ be $\Pi^0_1$ subsets of $2^\nn$.
If $P$ is not immune, and $Q\leq^{<\omega}_{tt,1}P$, then $Q$ is not immune.
\end{lemma}

\begin{proof}
Let $V$ be an infinite c.e.~subset of $T_P^{ext}$.
Assume that $Q\leq^{<\omega}_{tt,1}P$ holds via $n$ truth-table functionals $\{\Gamma_i\}_{i<n}$.
Note that every functional $\Gamma_i$ can be viewed as a computable monotone function from $2^{<\omega}$ into $2^{<\omega}$.
Let $V_k$ be the c.e.~set $V\cap\bigcap_{i<k}\Gamma_i^{-1}[2^{<\omega}\setminus T_P^{ext}]$ for each $k\leq n$.
By our assumption, $V_n$ is finite, since otherwise the tree generated from $V$ has an infinite path $f$ such that $\Phi_i(f)\not\in P$ for every $i<n$.
Let $k$ be the least number such that $V_{k+1}$ is finite.
Then, $\Gamma_k[V_k]$ is an infinite c.e.~set, and $\Gamma_k[V_k]$ is included in $T_P^{ext}$ except for finite elements. 
\end{proof}

\begin{cor}\label{cor:2:separ2}
There are $\Pi^0_1$ sets $P,Q\subseteq 2^\nn$ such that $Q<^{<\omega}_{tt,1}P\equiv^1_{<\omega}Q$.
\end{cor}

\begin{proof}
Let $P$ be an immune $\Pi^0_1$ subset of $2^\nn$.
Put $Q=P\fr P$.
As seen in Part I \cite{HK_PartI}, we have $Q\leq^1_1P\equiv^1_{<\omega}Q$.
Then, $Q$ is not immune since $T_Q^{ext}$ include an infinite computable subset $T_P$.
Hence, $P\not\leq^{<\omega}_{tt,1}Q$ by Proposition \ref{lem:NRMP:immuclosurett}.
\end{proof}

We have introduced two concatenation operations $\fr$ and $\tie$, while there are several other concatenation-like operations (see Duparc \cite{Dup}).
Let $P^{\rightarrow}Q$ and $P^{\sqcap} Q$ denote $[\{\sigma\fr\sharp\fr\tau:\sigma\in T_P\;\&\;\tau\in T_Q\}]$ and $[\{\sigma\fr\tau:\sigma\in T_P\;\&\;\tau\in T_Q\}]$, respectively.
\index{$P^{\rightarrow}Q$}\index{$P^{\sqcap} Q$}%
As seen in Part I \cite{HK_PartI}, we have $P\fr Q\equiv^1_1P^{\rightarrow}Q$.
However, there is a $(1,1)$-difference between $P\fr Q$ and $P^{\sqcap}Q$.

\begin{prop}
There are $\Pi^0_1$ sets $P,Q\subseteq 2^\nn$ such that $P^{\sqcap}Q<^1_1P\fr Q$.
\end{prop}

\begin{proof}\upshape
It is easy to see that $P^{\sqcap}Q\leq^1_1P^{\rightarrow}Q$ for any $P,Q\subseteq\nn^\nn$.
Let $R\subseteq 2^\nn$ be a $\Pi^0_1$ antichain.
Then we divide $R$ into four parts, $P_0$, $P_1$, $P_2$, and $P_3$.
Put $P=P_3$, and $Q=(\lrangle{0,1}\fr P_2\fr P_0)\cup(\lrangle{1}\fr P_2\fr P_1)$.
Without loss of generality, we may assume that $\lrangle{0}\in T_P$.
Suppose that $P^{\rightarrow}Q\leq^1_1P^{\sqcap}Q$ via a computable function $\Phi$.
Choose $g\in P_2$.
Then we have $\lrangle{0,1}\fr g\in P^\sqcap Q$.
Therefore, $\Phi(\lrangle{0,1}\fr g)\in P^{\rightarrow}Q$ must contain $\sharp$, since $P=P_3$ has no element computable in $g\in P_2$.
Thus, there is $n\in\nn$ such that $\Phi(\lrangle{0,1}\fr(g\res n))$ contains $\lrangle{\sharp,i}$ as a substring for some $i<2$.
Fix such $i$.
Then, $\Phi(\lrangle{0,1}\fr(g\res n))\in P^{\rightarrow}(Q\cap[\lrangle{i}])$.
We extend $g\res n$ to some leaf $\rho$ of $P_2$.
Choose $h_k\in P_k$ for each $k<2$.
Then, $\lrangle{0,1}\fr\rho\fr h_0\in Q\subseteq P^\sqcap Q$, and $\lrangle{0,1}\fr\rho\fr h_1\in\lrangle{0}\fr Q\subseteq P^{\sqcap}Q$.
Thus, $\Phi(\lrangle{0,1}\fr\rho\fr h_k)$ must belongs to $P^\rightarrow(Q\cap[\lrangle{i}])$, for each $k<2$.
However $P^{\rightarrow}(Q\cap[\lrangle{i}])$ has no element computable in $\lrangle{0,1}\fr\rho\fr h_{1-i}$.
A contradiction.
\end{proof}

\subsection{The Disjunction $\tie$ versus the Disjunction $\lcm$}

Let $\Psi$ be a learner (i.e., a total computable function $\Psi:\nn^{<\nn}\to\nn$).
A point $\alpha\in\nn^\nn$ is said to be {\em an $m$-changing point of $\Psi$} if $\#{\tt mcl}_\Psi(\alpha)\geq m$.
\index{leaner!$m$-changing point of}%
Then, the set of all $m$-changing points\footnote{The set of $m$-changing points is closedly related to the $m$-th derived set obtained from the notion of discontinuity levels (\cite{deBre13,Hemmerling08,Hertling96,Malek06}).
See also Part I \cite[Section 5.3]{HK_PartI} for more information on the relationship between the notion of mind-changes and the level of discontinuity.} of $\Psi$ is denoted by ${\tt mc}_\Psi(\geq m)$.
\index{${\tt mc}_\Psi(\geq m)$}%
A point $\alpha\in\nn^\nn$ is {\em anti-Popperian} with respect to $\Psi$ if $\lim_n\Psi(\alpha\res n)$ converges, but $\Phi_{\lim_n\Psi(\alpha\res n)}(\alpha)$ is partial\footnote{In the sense of the identification in the limit \cite{Gol}, the learner $\Psi$ is said to be Popperian if $\Phi_{\Psi(\sigma)}(\emptyset)$ is total for every $\sigma\in\nn^{<\nn}$ such that $\Psi(\sigma)$ is defined.
This definition indicates that, given any sequence $\alpha\in\nn^\nn$, if the learner makes an incorrect guess $\Phi_{\Psi(\alpha\res s)}(\emptyset)\not=\alpha$ at stage $s$, the leaner will eventually find his mistake $\Phi_{\Psi(\alpha\res s)}(\emptyset;n)\downarrow\not=\alpha(n)$.
In our context, the learner shall be called Popperian if given any falsifiable (i.e., $\Pi^0_1$) mass problem $Q$ and any sequence $\alpha\in\nn^\nn$, the incorrectness $\Phi_{\Psi(\alpha\res s)}(\alpha)\not\in Q$ implies $\Phi_{\Psi(\alpha\res s)}(\alpha)\res n\downarrow\not\in T_Q$ for some $n\in\nn$.
Every anti-Popperial point of $\Psi$ witnesses that $\Psi$ is not Popperian.}.
The set of all anti-Popperian points of $\Psi$ is denoted by ${\rm AP}_\Psi$.
\index{learner!anti-Popperian point of}\index{${\rm AP}_\Psi$}%

\begin{remark}[Trichotomy]
Let $\Gamma$ be a $(1,\omega)$-computable function identified by a learner $\Psi$, and let $P$ be any subset of Baire space $\nn^\nn$.
Then $\nn^\nn\setminus\Gamma^{-1}(P)$ is divided into the following three parts: the set $\Gamma^{-1}(\nn^\nn\setminus P)$; the $\Sigma^0_2$ set ${\rm AP}_\Psi$; and the $\Pi^0_2$ set $\bigcap_{m\in\nn}{\tt mc}_\Psi(\geq m)$.
\end{remark}

We say that $P_0$ and $P_1$ are {\em everywhere $(\omega,1)$-incomparable} if $P_0\cap[\sigma_0]$ is Muchnik incomparable with $P_1\cap[\sigma_1]$ (that is, $P_i\cap[\sigma_i]\not\leq^\omega_1P_{1-i}\cap[\sigma_{1-i}]$ for each $i<2$) whenever  $[\sigma_i]\cap P_i\not=\emptyset$ for each $i<2$.
\index{Pi01 set@$\Pi^0_1$ set!everywhere $(\omega,1)$-incomparable}%

\begin{theorem}\label{thm:NRMP:tielim}
Let $P_0,P_1$ be everywhere $(\omega,1)$-incomparable $\Pi^0_1$ subsets of $2^\nn$, and $\rho$ be any binary string.
For any $(1,\omega)$-computable function $\Gamma$ identified by a learner $\Psi$, the closure of ${\tt mc}_\Psi(\geq m)\cup\Gamma^{-1}(\nn^\nn\setminus P_0\oplus P_1)\cup{\rm AP}_\Psi$ includes $\rho\fr(P_0\tie_nP_1)^{\heartsuit}$ with respect to the relative topology on $\rho\fr(P_0\tie_{m+n}P_1)^{\heartsuit}$ (as a subspace of Baire space $\nn^\nn$).
\end{theorem}

\begin{proof}\upshape
Fix a string $\rho\fr\tau_0$ which is extendible in the heart of $\rho\fr(P_0\tie_nP_1)$.
Then, ${\tt pr}_i(\tau_0)$ must be extendible in $P_i$.
Fix $f_i\in P_i\cap[{\tt pr}_i(\tau_0)]$ witnessing $P_{1-i}\not\leq^\omega_1P_i$ for each $i<2$, i.e., $P_{1-i}$ contains no $f_i$-computable element.
Such $f_i$ exists, by everywhere $(\omega,1)$-incomparability.
Assume that $f_i={\tt pr}_i(\tau_0)\fr f^*_i$ for each $i<2$ and that the last declaration along $\tau_0$ is $j_0$, i.e., $\tau_0=\tau_0^-\fr(j_0,k)$ for some $k<2$.
Then we can proceed the following actions.
\begin{itemize}
\item Extend $\tau_0$ to $g_0=\tau_0\fr{\tt write}(j_0,f^*_{j_0})\in\rho\fr(P_0\tie_nP_1)$.
\item Wait for the least $s_0>|\tau_0|$ such that $\Phi_{\Psi(g_0\res s_0)}(g_0\res s_0;0)=j_0$.
\item Extend $g_0\res s_0$ to $g_1=(g_0\res s_0)\fr{\tt write}(j_1,f^*_{j_1})\in\rho\fr(P_0\tie_{n+1}P_1)$, where $j_1=1-j_0$.
\item Wait for the least $s_1>s_0$ such that $\Phi_{\Psi(g_1\res s_1)}(g_1\res s_1;0)=1-j_0$.
\end{itemize}

If both $s_0$ and $s_1$ are defined, then this action forces the learner $\Psi$ to change his mind.
In other words, $g_1\in{\tt mc}_\Psi(\geq 1)$.
Assume that $s_l$ is undefined for some $l<2$
Note that $g_l\equiv_Tf_{j_l}$, since ${\tt pr}_{j_l}(g_l)=f_{j_l}$ and ${\tt pr}_{1-j_l}(g_l)$ is finite, for each $l<2$.
Therefore, $\Gamma(g_l)\not\in(1-j_l)\fr P_{1-j_l}$ since $P_{1-j_l}$ has no $g_l$-computable element.
In this case, $g_l\in\Gamma^{-1}(\nn^\nn\setminus P_0\oplus P_1)$.
Hence, in $\rho\fr(P_0\tie_{n+1}P_1)^{\heartsuit}$, the closure of ${\tt mc}_\Psi(\geq 1)\cup\Gamma^{-1}(\nn^\nn\setminus P_0\oplus P_1)\cup{\rm AP}_\Psi$ includes $\rho\fr(P_0\tie_nP_1)^{\heartsuit}$.
By iterating this procedure, in $\rho\fr(P_0\tie_{m+n}P_1)^{\heartsuit}$, we can easily see that the closure of ${\tt mc}_\Psi(\geq m)\cup\Gamma^{-1}(\nn^\nn\setminus P_0\oplus P_1)\cup{\rm AP}_\Psi$ includes $\rho\fr(P_0\tie_nP_1)^{\heartsuit}$.
\end{proof}

\begin{cor}\label{cor:2:bl-bel}~
\begin{enumerate}
\item There exists $\Pi^0_1$ sets $P,Q\subseteq 2^\nn$ such that $P\lcm Q<^1_{<\omega}P\tie Q$.
\item There exists $\Pi^0_1$ sets $P,Q\subseteq 2^\nn$ such that $P\equiv^1_{\omega|<\omega}Q$ and $P<^1_{<\omega}Q$.
\end{enumerate}
\end{cor}

\begin{proof}\upshape
(1) Let $P$ be a perfect $\Pi^0_1$ antichain in $2^\nn$ of Theorem \ref{thm:JS}.
Fix a clopen set $C$ such that $P_0=P\cap C\not=\emptyset$, and $P_1=P\setminus C\not=\emptyset$.
Then every $f\in P_0$ and $g\in P_1$ are Turing incomparable.
Therefore, $P_0$ and $P_1$ are everywhere $(\omega,1)$-incomparable.
Let $\rho_n$ denote the $n$-th leaf of the tree $T_{\sf CPA}$ of a Medvedev complete $\Pi^0_1$ set ${\sf CPA}\subseteq 2^\nn$.
Fix a $(1,m)$-computable function $\Gamma$ identified by a learner $\Psi$.
By Theorem \ref{thm:NRMP:tielim}, $\rho_{m+1}\fr(P_0\tie_{m+1}P_1)$ intersects with ${\tt mc}_\Psi(\geq m+1)\cup\Gamma^{-1}(\omega^\omega\setminus P_0\oplus P_1)$.
Thus, $P_0\oplus P_1\not\leq^1_{<\omega}\cmeet_n(P_0\tie_nP_1)$.
Additionally, we easily have $P_0\lcm P_1\leq^1_{<\omega}\cmeet_n(P_0\tie_{n}P_1)$.
(2) $P=\cmeet_n(P_0\tie_nP_1)$ and $Q=P_0\linf P_1$ are $\Pi^0_1$.
\end{proof}

\subsection{The Disjunction $\lcm$ versus the Disjunction $\cls$}

By the similar argument, we can separate the strength of the concatenation $\lcm$ and the classical disjunction $\cls$.

\begin{theorem}\label{thm:NRMP:tielim2}
Let $P_0,P_1$ be everywhere $(\omega,1)$-incomparable $\Pi^0_1$ subsets of $2^\nn$.
For any $(1,\omega)$-computable function $\Gamma$, the complement of $\Gamma^{-1}(P_0\oplus P_1)$ is dense in $(P_0\cls P_1)^{\heartsuit}$ (as a subspace of Baire space $\nn^\nn$).
\end{theorem}

\begin{proof}\upshape
Fix a learner $\Psi$ which identifies the $(1,\omega)$-computable function $\Gamma$.
Fix any clopen set $[\tau]$ intersecting with the heart of $(P_0\cls P_1)$.
Assume that $[\tau]\cap(P_0\cls P_1)^{\heartsuit}$ contains no element of $\Gamma^{-1}(\nn^\nn\setminus P_0\oplus P_1)\cup{\rm AP}_\Psi$.
By Theorem \ref{thm:NRMP:tielim}, ${\tt mc}_\Psi(\geq n)$ is dense and open in the heart of $(P_0\cls P_1)\cap[\tau]=\tau\fr((P_0\cap[{\tt pr}_0(\tau)])\cls(P_1\cap[{\tt pr}_1(\tau)]))$.
As $[\tau]\cap(P_0\cls P_1)^\heartsuit$ is $\Pi^0_1$, the intersection $\bigcap_{n\in\nn}{\tt mc}_\Psi(\geq n)$ is dense in $[\tau]\cap(P_0\cls P_1)^\heartsuit$, by Baire Category Theorem.
Hence, $\nn^\nn\setminus\Gamma^{-1}(P_0\oplus P_1)$ intersects with any nonempty clopen set $[\tau]$ with $[\tau]\cap(P_0\cls P_1)^{\heartsuit}\not=\emptyset$.
\end{proof}

\begin{cor}\label{cor:2:l-b}~
\begin{enumerate}
\item There exist $\Pi^0_1$ sets $P,Q\subseteq 2^\nn$ such that $P\equiv^{<\omega}_{1}Q$ holds but $Q\not\leq^1_{\omega}P$ holds.
\item There exist $\Pi^0_1$ sets $P,Q\subseteq 2^\nn$ such that $P\equiv^{<\omega}_{\omega}Q$ holds but $P<^{1}_{\omega}Q$ holds.
\end{enumerate}
\end{cor}

\begin{proof}\upshape
Let $P$ be a perfect $\Pi^0_1$ antichain in $2^\nn$ of Theorem \ref{thm:JS}.
Fix a clopen set $C$ such that $P_0=P\cap C\not=\emptyset$, and $P_1=P\setminus C\not=\emptyset$.
Then $P_0$ and $P_1$ are everywhere $(\omega,1)$-incomparable.
Fix a $(1,\omega)$-computable function $\Gamma$ identified by a learner $\Psi$.
By Theorem \ref{thm:NRMP:tielim2}, $\nn^\nn\setminus(P_0\oplus P_1)$ is dense in $(P_0\cls P_1)^\heartsuit$.
For $\Pi^0_1$ sets $P_0,P_1\subseteq 2^\nn$, both $P_0\cls P_1$ and $P_0\linf P_1$ are $\Pi^0_1$, and $P_0\lcm P_1\leq^{1}_{1}P_0\linf P_1$.
\end{proof}

\begin{figure}[t]\centering
\begin{center}
\[
	\begin{xy}
	(0,20) *{P\oplus Q\ \ }="A", (25,15) *{\ \ P\cup Q\ \ }="B", (50,10) *{\ \ P\tie Q\ \ }="C", (75,5) *{\ \ P\lcm Q\ \ }="D", (100,0) *{\ \ P\cls Q}="E",
	\ar @<1mm> "A";"B"^{{\rm id}}
	\ar @<1mm> "B";"A"^{(<\omega,1)\text{-truth-table}}
	\ar @/_1pc/ (24,17.5);(1,22.5)|{\varprod}_{(1,1)\text{-computable}}
	\ar @<1mm> "B";"C"^{{\rm id}}
	\ar @<1mm> "C";"B"^{(1,<\omega)\text{-computable}}
	\ar @/_1pc/ (49,12.5);(26,17.5)|{\varprod}_{(1,1)\text{-computable}}
	\ar @<1mm> "C";"D"^{{\rm id}}
	\ar @<1mm> "D";"C"^{(1,\omega|<\omega)\text{-computable}}
	\ar @/_1pc/ (74,7.5);(51,12.5)|{\varprod}_{(1,<\omega)\text{-computable}}
	\ar @<1mm> "D";"E"^{{\rm id}}
	\ar @<1mm> "E";"D"^{(<\omega,1)\text{-computable}}
	\ar @/_1pc/ (99,2.5);(76,7.5)|{\varprod}_{(1,\omega)\text{-computable}}
	\end{xy}
\]
\end{center}
 \vspace{-0.5em}
\caption{The two-tape (bounded-errors) model of disjunctions for independent $\Pi^0_1$ sets $P,Q\subseteq 2^\nn$.}
  \label{fig:arrow1}
\end{figure}
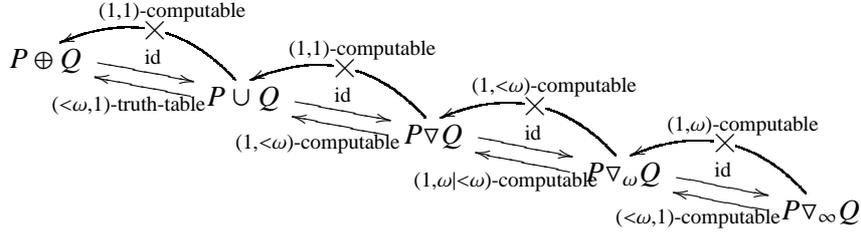

\subsection{The Disjunction $\bigoplus$ versus the Disjunction $\bcls$}

By the similar argument, we can separate infinitary disjunctions.
A sequence $\{x_i\}_{i\in\nn}$ of elements of $\nn^\nn$ is {\em Turing independent} if $x_i$ is not computable in $\bigoplus_{j\not=i}x_j$ for each $i\in\nn$.
\index{Turing independent}%
A collection $\{P_i\}_{i\in I}$ of subsets of $\nn^\nn$ is {\em pairwise everywhere independent} if, for any collection $\{[\sigma_i]\}_{i\in I}$ of clopen sets with $P_i\cap[\sigma_i]\not=\emptyset$ for each $i\in I$, there is a choice $\{x_i\}_{i\in I}\in\prod_{i\in I}(P_i\cap[\sigma_i])$ such that $P_i$ has no element computable in  $\bigoplus_{j\in I\setminus\{i\}}x_j$ for each $i\in I$.
\index{pairwise everywhere independent}%

\begin{theorem}\label{thm:NRMP:tielim3}
Let $\{P_i\}_{i<2^t}$ be a pairwise everywhere independent collection of $\Pi^0_1$ subsets of $2^\nn$, and let $\rho$ be any binary string.
For any $(t,\omega)$-computable function $\Gamma$, the complement of $\Gamma^{-1}(P_0\oplus\dots\oplus P_{2^t-1})$ is dense in the heart of $\rho\fr(P_0\cls\dots\cls P_{2^t-1})$ (as a subspace of Baire space $\nn^\nn$).
Indeed, for any nonempty interval $I$ in the heart of $\rho\fr(P_0\cls\dots\cls P_{2^t-1})$, there is $g\in \rho\fr(P_0\cls\dots\cls P_{2^t-1})^\heartsuit\cap I\setminus\Gamma^{-1}(P_0\oplus\dots\oplus P_{2^t-1})$ which is computable in some $g^*\in\bigotimes_{k<2^t-1}P_k$.
\end{theorem}

\begin{proof}\upshape
Assume that the $(t,\omega)$-computable function $\Gamma$ is identified by a team $\{\Psi_i\}_{i<t}$ of learners.
Fix a string $\rho\fr\tau_0$ which is extendible in the heart of $\rho\fr(P_0\cls\dots\cls P_{2^t-1})$.
Then, ${\tt pr}_i(\tau_0)$ must be extendible in $P_i$ for each $i<2^t$.
Fix $\{f_i\}_{i<2^t}\in\prod_{i<2^t}(P_i\cap[{\tt pr}_i(\tau_0)])$ witnessing the independence of $\{P_i\}_{i<2^t}$, i.e., $P_i$ contains no $\bigoplus_{j\not=i}f_j$-computable element.
Assume that $f_i={\tt pr}_i(\tau_0)\fr f^*_i$ for each $i<2^t$ and that the last declaration along $\tau_0$ is $j_0<2^t$, i.e., $\tau_0=\tau_0^-\fr(j_0,k)$ for some $k<2$.
Fix a computable function $\delta$ mapping $j<2^t$ to a unique binary string $\delta(j)$ satisfying $j=\sum_{e=0}^{t-1}2^e\cdot\delta(j;e)$.
Let $E^e_k$ denote the set $\{j<2^t:\delta(j;e)=k\}$.
Then we can proceed the following actions.
\begin{itemize}
\item Extend $\tau_0$ to $g_0=\tau_0\fr{\tt write}(j_0,f^*_{j_0})\in\rho\fr(P_0\cls\dots\cls P_{2^t-1})$.
\item Wait for the least $s_0>|\tau_0|$ such that $\Phi_{\Psi_e(g_0\res s_0)}(g_0\res s_0;0)\in E^e_{\delta(j_0;e)}$ for some $e<2^t$.
\item 
If such $s_0$ exists, then enumerate all such $e<2^t$ into an auxiliary set ${\rm Ch}_0$, and define $\delta(j_1)$ as follows:
\[
\delta(j_1;e)=
\begin{cases}
\delta(j_0;e) & \mbox{ if }e\not\in {\rm Ch}_0,\\
1-\delta(j_0;e) & \mbox{ if }e\in {\rm Ch}_0.
\end{cases}
\]
\item Extend $g_0\res s_0$ to $g_1=(g_0\res s_0)\fr{\tt write}(j_1,f^*_{j_1})\in\rho\fr(P_0\cls\dots\cls P_{2^t-1})$, where $j_1=\sum_{e=0}^{t-1}2^e\cdot\delta(j_1;e)$.
\end{itemize}

These actions force each learner $\Psi_e$ with $e\in{\rm Ch}_0$ to change his mind whenever the learner $\Psi_e$ want to have an element of $\bigoplus_{i<2^t}P_i$.
Fix $u\in\nn$.
Assume that $j_u$, $g_u$, $s_u$, and ${\rm Ch}_u$ has been already defined, and the following induction hypothesis at stage $u$ is satisfied.

\begin{itemize}
\item ${\tt pr}_e(g_u)\subseteq f_e$ for any $e<2^t$, hence, $g_u\in\rho\fr(P_0\cls\dots\cls P_{2^t-1})^\heartsuit\cap[\rho\fr\tau_0]$.
\item $\{s_v\}_{v\leq u}$ is strict increasing, and ${\rm Ch}_u\not=\emptyset$.
\item For each $e\in{\rm Ch}_u$, if $\Phi_{\Psi_e(g_u\res s_u)}(g_u\res s_u;0)$ converges to some value $k<2^t$, then $k\in E^e_{\delta(j_u;e)}$
\end{itemize}

It is easy to see that $u=0$ satisfies the induction hypothesis.
At stage $u+1\in\nn$, we proceeds the following actions.

\begin{itemize}
\item Define $\delta(j_{u+1})$ as follows:
\[
\delta(j_{u+1};e)=
\begin{cases}
\delta(j_u;e) & \mbox{ if }e\not\in {\rm Ch}_u,\\
1-\delta(j_u;e) & \mbox{ if }e\in {\rm Ch}_u.
\end{cases}
\]
\item Extend $g_u\res s_u$ to $g_{u+1}=(g_u\res s_u)\fr{\tt write}(j_{u+1},f^{(u+1)}_{j_{u+1}})$, where $j_{u+1}=\sum_{e=0}^{t-1}2^e\cdot\delta(j_{u+1};e)$, and $f^{u+1}$ satisfies ${\tt pr}_{j_{u+1}}(g_{u+1})={\tt pr}_{j_{u+1}}(g_u\res s_u)\fr f^{(u+1)}_{j_{u+1}}=\rho\fr f_{j_{u+1}}$.
\item Wait for the least $s_{u+1}>s_u$ such that $\Phi_{\Psi_e(g_{u+1}\res s_{u+1})}(g_{u+1}\res s_{u+1};0)\in E^e_{\delta(j_{u+1};e)}$ for some $e<2^t$.
\item If such $s_{u+1}$ exists, then enumerate all such $e<2^t$ into ${\rm Ch}_{u+1}$,
\end{itemize}

By our action, it is easy to see that $u+1$ satisfies the induction hypothesis.
As the set $\rho\fr(P_0\cls\dots\cls P_{2^t-1})^\heartsuit$ is closed (with respect to the Baire topology) and $\{s_u\}_{u\in\nn}$ is strictly increasing, the sequence $\{g_u\}_{u\in\nn}$ converges to some $g\in\rho\fr(P_0\cls\dots\cls P_{2^t-1})^\heartsuit$.
Let $I(g)\subseteq 2^t$ be the set of all $e<2^t$ such that ${\tt pr}_e(g)$ is total.

\begin{claim}
$g\leq_T\bigoplus_{e\in I(g)}f_e$.
\end{claim}

Note that $g=g[f_0,\dots,f_{2^t-1}]$ is effectively constructed uniformly in a given collection $\{f_k\}_{k<2^t}$.
In other words, there is a (uniformly) computable function $\Theta$ mapping $\{f_k\}_{k<2^t}$ to $\Theta(\{f_k\}_{k<2^t})=g=g[f_0,\dots,f_{2^t-1}]$.
Then, it is easy to see that the function $\Theta$ maps $\{f_e\}_{e\in I(g)}\cup\{{\tt pr}_e(g)\fr 0^\nn\}_{e\in 2^t\setminus I(g)}$ to $g$.
Hence, $g\leq_T\bigoplus_{e\in I(g)}f_e\oplus\bigoplus_{e\in 2^t\setminus I(g)}{\tt pr}_e(g)\fr 0^\nn$.
Therefore, $g\leq_T\bigoplus_{e\in I(g)}f_e$ as desired, since ${\tt pr}_e(g)\fr 0^\nn$ is computable for any $e\in 2^t\setminus I(g)$.

\medskip

Let $\Gamma_e$ denote the $(1,\omega)$-computable function identified by $\Psi_e$, that is, $\Gamma_e(\alpha)=\Phi_{\lim_n\Psi(\alpha\res n)}(\alpha)$ for any $\alpha\in\nn^\nn$.
We consider the following two cases.

\medskip

\noindent
{\bf Case 1} ($e\in{\rm Ch}_u$ for finitely many $u\in\nn$).
Fix $u$ such that $e\not\in{\rm Ch}_v$ for any $v>u$.
For each $v>u$, $\Phi_{\Psi_e(g\res s_u)}(g\res s_u;0)$ does not converges to an element of $E^e_{\delta(j_v;e)}=E^e_{\delta(j_u;e)}$.
By our definition, for each $k\not\in E^e_{\delta(j_u;e)}$, ${\tt pr}_k(g)\subset\rho\fr f_k$ is finite.
By previous claim, $g\leq_T\bigoplus_{e\not=k}f_e$.
Thus, by independence, $P_k$ has no $g$-computable element.
If $\Phi_{\Psi_e(g\res s_u)}(g\res s_u;0)\uparrow$ for any $u\in\nn$, then $g\in{\rm AP}_{\Psi_e}$.
If $\lim_n\Psi_e(g\res n)$ does not converge, then $g\in\bigcap_{m\in\nn}{\tt mc}_{\Psi_e}(\geq m)$.
Otherwise, $\Phi_{\lim_n\Psi_e(g\res n)}(g;0)$ converges to some value $k\not\in E^e_{\delta(j_u;e)}$.
As $\Phi_{\lim_n\Psi_e(g\res n)}(g)$ is $g$-computable, we see $\Phi_{\lim_n\Psi_e(g\res n)}(g)\not\in k\fr P_k$.
Consequently, $g\in\nn^\nn\setminus\Gamma_e^{-1}(\bigoplus_{k<2^t}P_k)$.

\medskip

\noindent
{\bf Case 2} ($e\in{\rm Ch}_u$ for infinitely many $u\in\nn$).
We enumerate an infinite increasing sequence $\{u[n]\}_{n\in\nn}$, where $u[n]$ is the $n$-th element such that $e\in{\rm Ch}_{u}$.
As $e\in{\rm Ch}_{u[n]}$, we have $\Phi_{\Psi_e(g\res u[n])}(g\res u[n];0)\in E^e_{\delta(j_{u[n]};e)}$.
By our action, $\delta(j_{u[n+1]};e)=\delta(j_{u[n]+1};e)\not=\delta(j_{u[n]};e)$.
This implies $E^e_{\delta(j_{u[n+1]};e)}\cap E^e_{\delta(j_{u[n]};e)}=\emptyset$.
However, we must have $\Phi_{\Psi_e(g\res u[n+1])}(g\res u[n+1];0)\in E^e_{\delta(j_{u[n+1]};e)}$, since $e\in{\rm Ch}_{u[n+1]}$.
This forces the learner $\Psi_e$ to change his mind.
By iterating this procedure, we eventually obtain $g\in\bigcap_{m\in\nn}{\tt mc}_{\Psi_e}(\geq m)$.

\medskip

Consequently, $g\in\bigcap_{e\in\nn}(\nn^\nn\setminus\Gamma_e^{-1}(\bigoplus_{k<2^t}P_k))$.
Thus, $g\in\nn^\nn\setminus\Gamma_e^{-1}(\bigoplus_{k<2^t}P_k)$.
For any $\tau_0$ such that $\rho\fr\tau_0$ which is extendible in the heart of $\rho\fr(P_0\cls\dots\cls P_{2^t-1})$, we can construct such $g$ extending $\tau_0$.
Therefore, $\nn^\nn\setminus\Gamma_e^{-1}(\bigoplus_{k<2^t}P_k)$ intersects any nonempty interval in $\rho\fr(P_0\cls\dots\cls P_{2^t-1})^\heartsuit$.
In other words, $\nn^\nn\setminus\Gamma_e^{-1}(P_0\oplus\dots\oplus P_{2^t-1})$ is dense in $\rho\fr(P_0\cls\dots\cls P_{2^t-1})^\heartsuit$ as desired.
\end{proof}

The following theorem by Jockusch-Soare \cite[Theorem 4.1]{JS} is important.

\begin{theorem}[Jockusch-Soare \cite{JS}]\label{thm:JS2}
There is a computable sequence $\{\prod_nP^i_n\}_{i\in\nn}$ of nonempty homogeneous $\Pi^0_1$ subsets of $2^\nn$ such that $\{x_i\}_{i\in\nn}$ is Turing independent for any choice $x_i\in\prod_nP^i_n$, $i\in\nn$.
\end{theorem}

Clearly any such $\Pi^0_1$ set contains no element of a {\sf PA} degree, a Turing degree of a complete consistent extension of Peano Arithmetic.
Accordingly, every element of such a $\Pi^0_1$ set computes no element of a Medvedev complete $\Pi^0_1$ set ${\sf CPA}$.

\begin{cor}\label{thm:2:construction-tl-w}
There are $\Pi^0_1$ sets $P_n\subseteq 2^\nn$, $n\in\nn$, such that $\cmeet_t(P_0\cls\dots\cls P_t)<^{<\omega}_{\omega}\cmeet_t P_t$.
\end{cor}

\begin{proof}\upshape
Fix the computable sequence $\{P_i\}_{i\in\nn}$ of Theorem \ref{thm:JS2}.
Then $\{P_i\}_{i\in\nn}$ is pairwise everywhere independent.
Assume that $\cmeet_t P_t\leq^{<\omega}_{\omega}\cmeet_t(P_0\cls\dots\cls P_t)$ via a $(t,\omega)$-computable function $\Gamma$.
Let $\rho_n$ denote the $n$-th leaf of the tree $T_{\sf CPA}$ of a Medvedev complete $\Pi^0_1$ subset of $2^\nn$.
By Theorem \ref{thm:NRMP:tielim3}, $\Gamma^{-1}(\bigcup_{k<2^t}\rho_k\fr P_k)$ is dense in the heart of $\rho_{2^t}\fr(P_0\cls\dots\cls P_{2^t-1})$.
In particular, there is $g\in\rho_{2^t}\fr(P_0\cls\dots\cls P_{2^t-1})$ such that $\Gamma(g)\not\in\bigcup_{k<2^t}\rho_k\fr P_k$ which is computable in some $g^*\leq_T\bigotimes_{k<2^t}P_k$.
By our choice of $\{P_i\}_{i\in\nn}$, $\Gamma(g)$ computes no element of $\bigcup_{k\geq 2^t}P_k\cup{\sf CPA}$.
Thus, $\Gamma(g)\not\in\cmeet_tP_t$.
\end{proof}

\begin{cor}\label{thm:tl-w}~
\begin{enumerate}
\item There exists a computable sequence $\{P_n\}_{n\in\nn}$ of $\Pi^0_1$ subsets of Cantor space $2^\nn$, such that the condition $\left[\bcls\right]_nP_n<^{<\omega}_{\omega}\binf_nP_n$ is satisfied.
\item There exist $\Pi^0_1$ sets $P,Q\subseteq 2^\nn$ such that $P\equiv^{\omega}_{1}Q$ holds but $P<^{<\omega}_{\omega}Q$ holds.
\end{enumerate}
\end{cor}

\begin{proof}\upshape
(1) By Corollary \ref{thm:2:construction-tl-w}. the condition $\left[\bcls\right]_nP_n\leq^1_1\cmeet_t(P_0\cls\dots\cls P_t)<^{<\omega}_{\omega}\cmeet_t P_t\leq^{1}_{1}\binf_nP_n$ is satisfied.
(2) Put $P=\left[\bcls\right]_nP_n\leq^1_1\cmeet_t(P_0\cls\dots\cls P_t)$ and $Q=\cmeet_t P_t$.
Then $P$ and $Q$ are $(1,1)$-equivalent to $\Pi^0_1$ subsets as seen in Part I \cite{HK_PartI}.
By Theorem \ref{thm:2:construction-tl-w}, $P<^{<\omega}_{\omega}Q$, and $Q\equiv^{\omega}_{1}P$ as seen in Part I \cite{HK_PartI}.
\end{proof}

%% file: NRMP_fullproof3a.tex
\section{Contiguous Degrees and Dynamic Infinitary Disjunctions}

\subsection{When the Hierarchy Collapses}

We have already observed the following hierarchy, for pairwise independent $\Pi^0_1$ sets $P,Q\subseteq 2^\nn$.
\[P\linf Q>^{1}_{1}P\cup Q>^{1}_{1}P\tie Q>^{1}_{<\omega}P\lcm Q>^{1}_{\omega|<\omega}P\cls Q\equiv^{<\omega}_{1}P\linf Q.\]

{\em Homogeneity} is an opposite notion of antichain (and independence).
Recall that $S\subseteq\nn^\nn$ is {\em homogeneous} if $S=\prod_nS_n$ for some $S_n\subseteq\nn$, $n\in\nn$.
\index{Pi01 set@$\Pi^0_1$ set!homogeneous}%
Every antichain is degree-non-isomorphic everywhere.
On the other hand, every homogeneous set $S$ is {\em degree-isomorphic everywhere}, that is to say, $S\cap C$ is degree-isomorphic to $S\cap D$ for any clopen sets $C,D\subseteq\nn^\nn$ with $S\cap C\not=\emptyset$ and with $S\cap D\not=\emptyset$\footnote{The anonymous referee pointed out that the notion of degree-isomorphic everywhere is related to the notion of {\em fractal} in the study of Weihrauch degrees \cite{BGM,lRP13}. The (reverse) lattice embedding $d$ of the Medvedev degrees into the Weihrauch degrees has the property that a subset $P$ of Baire space is degree-isomorphic everywhere if and only if $d(P)$ is a fractal.}.
\index{Pi01 set@$\Pi^0_1$ set!degree-isomorphic everywhere}%

The next observation is that every finite-piecewise computable method of solving a homogeneous $\Pi^0_1$ mass problem can be refined by a finite-$(\Pi^0_1)_2$-piecewise computable method.
That is to say, our hierarchy between $\leq^1_{<\omega}$ and $\leq^{<\omega}_1$ collapses for homogeneous $\Pi^0_1$ sets, modulo the $(1,<\omega)$-equivalence.

\begin{theorem}\label{thm:contig:b-bl}
For every homogeneous $\Pi^0_1$ set $S\subseteq\nn^\nn$ and for any set $Q\subseteq\nn^\nn$, if $S\leq^{<\omega}_{1}Q$ then $S\leq^{1}_{<\omega}Q$.
\end{theorem}

\begin{proof}\upshape
Let $S=\prod_xF_x$ for some $\Pi^0_1$ sets $F_x\subseteq\nn$.
Assume $S\leq^{<\omega}_{1}Q$ via the bound $b$.
That is, for every $g\in Q$ there exists an index $e<b$ such that $\Phi_e(g)\in S$.
Let us begin defining a learner $\Psi$ who changes his mind at most finitely often.
Fix $g\in Q$.
The learner $\Psi$ first sets $A_0=\{e\in\nn:e<b\}$.
By our assumption, we have $\Phi_e(g)\in S$ for some $e\in A_0$.
Then the learner $\Psi$ challenges to predict the solution algorithm $e<b$ such that $\Phi_e(g)\in S$ by using an observation $g\in Q$.
He begins the $1$-st challenge.
On {\em the $(s+1)$-th challenge of $\Psi$}, inductively assume that, the learner have already defined a set $A_s\subseteq A_0$.
Let $v$ be a stage at which the $s+1$-th challenge of $\Psi$ on $g$ begins.
In this challenge, the learner $\Psi$ uses the two following computable functionals $\Gamma$ and $\Delta$.
\begin{itemize}
\item For a given argument $x$, $\Gamma(x,s+1)$ outputs the least $\lrangle{e(x),t(x)}$ such that $e(x)\in A_s$ and $\Phi_{e(x)}(g\res t(x);x)\downarrow$ if such $\lrangle{e(x),t(x)}$ exists.
\item If $\Gamma(x,s+1)=\lrangle{e(x),t(x)}$, then $\Delta(g;x,s+1)=\Phi_{e(x)}(g\res t(x);x)$.
\end{itemize}

Set $\Delta_{s+1}(g;x)=\Delta(g;x,s+1)$.
Clearly, an index $d(s+1)$ of $\Delta_{s+1}$ is calculated from $s+1$.
Then the learner $\Psi(g\res v)$ outputs $d(s+1)$ on the $(s+1)$-th challenge.
Hence $\Phi_{\Psi(g\res v)}(g;x)=\Phi_{d(s+1)}(g;x)=\Phi_{e(x)}(g\res t(x);x)$ for any $x$.
He does not change his mind until the beginning stage $v'$ of the next challenge, i.e., $\Phi_{\Psi(g\res v'')}(g)=\Phi_{\Psi(g\res k)}(g)$ for $k\leq v''<v'$.
{\em The next challenge} might begin when it turns out that $\Psi$'s prediction on his $(s+1)$-th challenge is incorrect, namely:
\begin{itemize}
\item $\Phi_{\Psi(g\res v)}(g\res u)\res n\not\in T_{S,u}$ for some $n<u$ at some stage $u>v$.
\end{itemize}
Here $T_S$ is a corresponding computable tree of $S$.
For each $x\in\nn$, fix a decreasing approximation $\{F_{x,s}\}_{s\in\nn}$ of a $\Pi^0_1$ set $F_x\subseteq\{0,1\}$, uniformly in $x$.
In this case, there exists $x<n$ such that the following condition holds.
\[\Phi_{\Psi(g\res v)}(g\res u;x)=\Delta_{s+1}(g;x)=\Phi_{e(x)}(g;x)\not\in F_{x,s}.\]

For such a least $x$, the learner removes $e(x)$ from $A_s$, that is, let $A_{s+1}=A_s\setminus\{e(x)\}$.
If $A_{s+1}\not=\emptyset$ then the learner $\Psi$ begins {\em the $(s+2)$-th challenge} at the current stage $u$.
The construction of the learner $\Psi$ is completed.
An important point of this construction is that the learner never uses an index rejected on some challenge.
This makes the prediction on $g\in Q$ of the learner $\Psi$ converge.

\begin{claim}
$\Psi$ changes his mind at most $b$ times.
\end{claim}

Whenever $\Psi$ changes, $A_s$ must decrease.
However $\# A_0=b$.

\begin{claim}
For every $g\in Q$ it holds that $\Phi_{\lim_s\Psi(g\res s)}(g)\in S$.
\end{claim}

For $g\in Q$, let $B^g\subseteq A_0$ be the set of all $e\in A_0$ such that $\Phi_e(g)\in S=\prod_xF_x$.
By the definition of $A_0$, clearly $B^g$ is not empty.
Moreover, $B^g\subseteq\bigcap_sA_s$ holds, since $e$ is removed from $\bigcap_sA_s$ only when $\Phi_e(g;x)\not\in F_x$ for some $x$.
Thus, $\Phi_{\Psi(g\res v)}(g):\nn\to\nn$ is total for every stage $v$.
This means that, if $\Phi_{\Psi(g\res v)}(g)\not\in S$, then the learner $\Psi$ will know his mistake at some stage $u$, i.e., $\Phi_{\Psi(g\res v)}(g\res u;x)\not\in F_{x,u}$ for some $x<u$.
Then some index is removed from $\bigcap_sA_s$.
However, this occurs at most $b$ times.
Thus, $\Phi_{\lim_s\Psi(g\res s)}(g)\in S$.
\end{proof}

Let $\alpha,\beta,\gamma\in\{1,<\omega,\omega\}$.
We say that a $(\alpha,\beta|\gamma)$-degree $\dg{a}$ of a nonempty $\Pi^0_1$ subset of $2^\nn$ is {\em $(\alpha,\beta|\gamma)$-complete} if $\dg{b}\leq\dg{a}$ for every $(\alpha,\beta|\gamma)$-degree $\dg{b}$ of a nonempty $\Pi^0_1$ subset of $2^\nn$.
If a $\Pi^0_1$ set $P\subseteq 2^\nn$ has an $(\alpha,\beta|\gamma)$-complete $(\alpha,\beta|\gamma)$-degree, then it is also called {\em $(\alpha,\beta|\gamma)$-complete}.
\index{complete!$(\alpha,\beta\mid\gamma)$-}%

\begin{cor}
A $\Pi^0_1$ subset of $2^\nn$ is $(1,<\omega)$-complete if and only if it is $(1,\omega|<\omega)$-complete if and only if it is $(<\omega,1)$-complete.
\end{cor}

\begin{proof}\upshape
Let ${\sf DNR}_2$ denote the set of all two-valued diagonally noncomputable functions, where a function $f:\nn\to 2$ is {\em diagonally noncomputable} if $f(e)\not=\Phi_e(e)$ for any index $e$.
This set is clearly homogeneous, and $\Pi^0_1$.
\index{diagonally noncomputable}%
Moreover, it is $(1,1)$-complete (hence $(\alpha,\beta|\gamma)$-complete for any $\alpha,\beta,\gamma\in\{1,<\omega,\omega\}$).
Therefore, we can apply Theorem \ref{thm:contig:b-bl} with $S={\sf DNR}_2$.
\end{proof}

\begin{cor}
There are $\Pi^0_1$ sets $P,Q\subseteq 2^\nn$ such that $P\linf Q\equiv^{1}_{<\omega}P\cls Q$.
Indeed, if $P$ is homogeneous and $Q\equiv^1_1P$, then $P\linf Q\equiv^{1}_{<\omega}P\cls Q$ is satisfied.
\end{cor}

\begin{proof}\upshape
Let $P$ be any homogeneous $\Pi^0_1$ subset of $2^\nn$.
Then $P\oplus P$ is also homogeneous.
As seen in Part I \cite[Section 4]{HK_PartI}, there is a $(2,1)$-computable function from $P\cls P$ to $P\oplus P$, hence $P\oplus P\leq^{<\omega}_1P\cls P$.
Thus, by Theorem \ref{thm:contig:b-bl}, $P\oplus P\leq^1_{<\omega}P\cls P$.
Recall from Part I that $Q\equiv^1_1P$ implies $P\cls P\equiv^1_1P\cls Q$.
Hence, $Q\equiv^1_1P$ implies $P\oplus Q\equiv^1_1P\oplus P\leq^1_{<\omega}P\cls P\equiv^1_1P\cls Q$.
\end{proof}

It is natural to ask whether our hierarchy of disjunctive notions for homogeneous $\Pi^0_1$ sets also collapses {\em modulo the $(1,1)$-equivalence}.
The answer is {\em negative}.
We say that a homogeneous set $\prod_nF_n$ is {\em computably bounded} if there is a computable function $l:\nn\to\nn$ such that $F_n\subseteq\{0,\dots,l(n)\}$ for any $n\in\nn$.
\index{computably bounded}%
Clearly, every homogeneous subset of Cantor space $2^\nn$ is computably bounded.
Cenzer-Kihara-Weber-Wu \cite{CKWW} introduced the notion of immunity for closed sets.
A closed subset $P$ of Cantor space $2^\nn$ is {\em immune} if $T_P^{ext}$ has no infinite computable subset.
\index{immune}%

\begin{theorem}\label{thm:chc}
Let $P\subseteq 2^\nn$ be a non-immune $\Pi^0_1$ set, and $S_0,S_1,\dots,S_m\subseteq\nn^\nn$ be special computably bounded homogeneous $\Pi^0_1$ sets.
Then $\bigcup_{i\leq m}S_i\not\leq^{1}_{1}P$.
\end{theorem}

\begin{proof}\upshape
Let $V_0$ be an infinite c.e. subtree of $T_P^{ext}$.
Assume that $\bigcup_{i\leq m}\prod_nF_n^i\leq^{1}_{1}P$ via a computable functional $\Phi$, where, for each $i<m$, $\{F_n^i\}_{n\in\omega}$ is a uniformly $\Pi^0_1$ sequence of subsets of $\{0,1,\dots l_i\}$.
Let $S_i^{ext}$ denotes the corresponding $\Pi^0_1$ tree of $\prod_nF_n^i$, and let $L_i=\{\rho:(\exists\tau\in S_i^{ext})(\exists i)\;\rho=\tau\fr\lrangle{i}\not\in S_i^{ext}\}$, for each $i$.
Note that $L_i$ {\em differs} from the set of leaves of the corresponding {\em computable} tree of $\prod_nF_n^i$.
We first consider the set $L_i^\Phi=\{\rho\in L_i:(\exists\sigma\in V_i)\;\Phi(\sigma)\supseteq\rho\}$, where $V_i$ for $0<i\leq m$ will be defined in the below construction.
Note that $L_0^\Phi$ is computably enumerable.
There are three cases:
\begin{enumerate}
\item $L^\Phi_0$ is infinite;
\item $L^\Phi_0$ is finite, hence $\Phi([V_0])$ is a subset of $\prod_nF_n^0$;
\item otherwise.
\end{enumerate}

(Case 1): For any $n$, there exists $\rho\in L_0^\Phi$ of height $>n+1$, and $\rho(n)\in F_n^0$.
From any computable enumeration of $L_0^\Phi$ we can calculate a computable path of $\prod_nF_n^0$.
This contradicts the specialness of $S_0=\prod_nF_n^0$.

(Case 2): There exists a finite number $k$ such that, for every string $\sigma\in V_0$ of height $>k$, $\Phi(\sigma)$ belongs to $S_i^{ext}$.
This also contradicts the specialness of $S_0=\prod_nF_n^0$.

(Case 3): There exists infinitely many strings $\sigma\in V_0$ such that $\Phi(\sigma)$ extends some string of $L_0^\Phi$.
Since $L_0^\Phi$ is finite, by the pigeon hole principle, there exists $\rho_0\in L_0^\Phi$ such that $\Phi(\sigma)$ extends $\rho$ for infinitely many $\sigma\in V_0$.
Fix such $\rho_0$, and let $V_1=\{\sigma\in V_0:\Phi(\sigma)\supseteq\rho\}$.
Then the downward closure of $V_1$ is an infinite c.e. subtree of $T_P^{ext}$, and $\Phi([V_1])\cap S_0=\emptyset$.

By iterating this procedure, we win the either of the cases 1 or 2 for some $i\leq m$.
The reason is that, if the case 3 occurs for $j$, then $V_{j+1}$ is defined as an infinite c.e. subtree of $T_P^{ext}$ such that $\Phi([V_1])\cap(\bigcup_{i\leq j}S_i)=\emptyset$. 
Since $\bigcup_{i\leq m}\prod_nF_n^i\leq^{1}_{1}P\leq[V_m]$, i.e., $\Phi([V_m])\subseteq\bigcup_{i\leq m}S_i$, the case 3 does not occur for $m$.
\end{proof}

\begin{cor}\label{cor:NRMP:chc}
Let $P,Q$ be any nonempty $\Pi^0_1$ subsets of $2^\nn$, and $S,T$ be computably bounded homogeneous $\Pi^0_1$ sets.
Then $S\cup T\not\leq^{1}_{1}P\tie Q$.
\end{cor}

\begin{proof}\upshape
Clearly $P\tie Q$ is not immune.
Thus, Theorem \ref{thm:chc} implies $S\cup T\not\leq^{1}_{1}P\tie Q$.
\end{proof}

To understand degrees of difficulty of disjunctive notions, and to discover new {\em easier} (possibly infinitary) disjunctive notions, it is interesting to discuss {\em contiguous degrees}.

\begin{definition}
Let $(\alpha,\beta,\gamma),(\alpha^*,\beta^*,\gamma^*)\in\{1,<\omega,\omega\}^3$, and assume that $\leq^{\alpha}_{\beta|\gamma}$ is not finer than $\leq^{\alpha^*}_{\beta^*|\gamma^*}$.
An $(\alpha,\beta|\gamma)$-degree $\dg{a}^{\alpha}_{\beta|\gamma}$ is {\em $(\alpha^*,\beta^*|\gamma^*)$-contiguous} if $\dg{a}^{\alpha}_{\beta|\gamma}$ contains at most one $(\alpha^*,\beta^*|\gamma^*)$-degree, that is to say, for any representatives $A,B\in\dg{a}^\alpha_{\beta|\gamma}$, we have that $A$ is $(\alpha^*,\beta^*|\gamma^*)$-equivalent to $B$.
\index{contiguous!$(\alpha^*,\beta^*\mid\gamma^*)$-}%
\end{definition}

\begin{cor}~
\begin{enumerate}
\item There is a $(1,<\omega)$-contiguous $(<\omega,1)$-degree of $\Pi^0_1$ sets of $2^\nn$.
\item Every $(1,<\omega)$-degree which contains a homogeneous $\Pi^0_1$ set or a $\Pi^0_1$ antichain is not $(1,1)$-contiguous.
\item Every $(1,\omega|<\omega)$-degree of $\Pi^0_1$ antichains is not $(1,<\omega)$-contiguous.
\item Every $(<\omega,1)$-degree of $\Pi^0_1$ antichains is not $(1,\omega)$-contiguous (hence, is not $(1,\omega|<\omega)$-contiguous).
\end{enumerate}
\end{cor}

\begin{proof}\upshape
(1) This follows from Theorem \ref{thm:contig:b-bl}.

(2) If $\dg{d}$ is a $(1,<\omega)$-degree of a homogeneous $\Pi^0_1$ set $S$, then $\dg{d}$ contains $S$ and $S\tie S$, since $S\equiv^{1}_{<\omega}S\tie S$.
However, $S\tie S<^{1}_{1}S\cup S=S$ by Corollary \ref{cor:NRMP:chc}.
If $\dg{d}$ is a $(1,<\omega)$-degree of a $\Pi^0_1$ antichain $P$, then $\dg{d}$ contains $(P\times 2^\nn)\cup(2^\nn\times P)$ and $P\tie P$, since $P\equiv^{1}_{<\omega}(P\times 2^\nn)\cup(2^\nn\times P)$.
However, $P\tie P<^{1}_{1}(P\times 2^\nn)\cup(2^\nn\times P)$ holds by Lemma \ref{lem:NRMP:tiecup}.

(3) Note that, for any $\Pi^0_1$ set $P$ and any clopen set $C$, it holds that $(P\cap C)\linf(P\setminus C)\equiv^{1}_{1}P$.
Let $\dg{d}$ be a $(1,\omega|<\omega)$-degree of a $\Pi^0_1$ antichain $P$.
Fix a clopen set $C$ such that $P_0=P\cap C\not=\emptyset$, and $P_1=P\setminus C\not=\emptyset$.
Then $\dg{d}$ contains $P_0\linf P_1$ and $\binf_n^\rightharpoonup(P_0\tie_{n}P_1)$, since $P_0\linf P_1\equiv^{1}_{\omega|<\omega}\binf_n^\rightharpoonup(P_0\tie_{n}P_1)$.
However, $\binf_n^\rightharpoonup(P_0\tie_{n}P_1)<^{1}_{<\omega}P_0\linf P_1$ holds by Corollary \ref{cor:2:bl-bel}.

(4) Let $\dg{d}$ be a $(<\omega,1)$-degree of a $\Pi^0_1$ antichain $P$.
Fix a clopen set $C$ such that $P_0=P\cap C\not=\emptyset$, and $P_1=P\setminus C\not=\emptyset$.
Then $\dg{d}$ contains $P_0\linf P_1$ and $P_0\tie_{\omega}P_1$, since $P_0\linf P_1\equiv^{<\omega}_{1}P_0\tie_{\omega}P_1$.
However, $P_0\tie_\omega P_1<_{l}P_0\linf P_1$ holds by Corollary \ref{cor:2:l-b}.
\end{proof}




\subsection{Concatenation, Dynamic Disjunctions, and Contiguous Degrees}

We next show the non-existence of nonzero $(1,1)$-contiguous $(1,<\omega)$-degree, that is, we will see the LEVEL 4 separation between $[\mathfrak{C}_T]^1_1$ and $[\mathfrak{C}_T]^1_{<\omega}$.
Indeed, we show the strong anti-cupping result for $(1,1)$-degrees inside every nonzero $(1,<\omega)$-degree via the concatenation operation.
The following theorem is one of the most important and nontrivial results in this paper.

\begin{theorem}\label{thm:Hig}
For any nonempty $\Pi^0_1$ sets $P,Q\subseteq 2^\nn$, $Q\fr P$ does not $(1,1)$-cup to $P$.
That is to say, for any $R\subseteq\nn^\nn$, if $P\leq^{1}_{1} (Q\fr P)\otimes R$ then $P\leq^{1}_{1} R$.
\end{theorem}

\begin{proof}\upshape
Fix a computable tree $V_P$ (resp. $V_Q$) such that $[V_P]=P$ (resp. $[V_Q]=Q$), and let $L_P$ (resp. $L_Q$) denote all leaves of $V_P$ (resp. $V_Q$).
For a tree $T\subseteq 2^{<\nn}$ and $g\in\nn^\nn$, we write $T\lsup\{g\}$ for $\{\sigma\oplus\tau:\sigma\in V_P\;\&\;\tau\subset g\;\&\;|\sigma|=|\tau|\}$.
For computable trees $S$ and $T$, we also write $S\ntie T$ for $S\cup\bigcup_{\rho\in L_S}\rho\fr T$, where $L_S$ denotes the set of all leaves of $S$.
Assume $P\leq^{1}_{1}(Q\ntie P)\otimes R$ via a computable functional $\Phi$.
We construct a computable functional $\Psi$ witnessing $P\leq^{1}_{1} R$.
Fix $g\in R$.
Then we will find a $g$-c.e. tree $D^g\subseteq V_P$ without dead ends.
To this end, we inductively construct a uniformly $g$-computable sequences $\{D^g_i\}_{i\in\omega},\{E^g_i\}_{i\in\omega}$ of $g$-computable trees, as follows.
\begin{align*}
E^g_0&=V_Q\otimes\{g\}; & D^g_0&=\Phi(E^g_0).\\
E^g_{i+1}&=(V_Q\ntie D^g_i)\otimes\{g\}; & D^g_{i+1}&=\Phi(E^g_{i+1}).
\end{align*}

Here $\Phi(E^g_i)$ denotes the image of $E^g_i$ under a functional $\Phi$, namely, $\Phi(E^g_i)=\{\tau\subseteq 2^{<\nn}:(\exists\sigma\in E^g_i)\;\tau\subseteq\Phi(\sigma)\}$.
Finally, we define a $g$-c.e. tree $D^g=\bigcup_nD^g_n$.
Now, we let $W$ be the tree $V_Q\ntie V_P$, and then we observe $[W]=Q\ntie P$ and $V_Q\subseteq W^{ext}$.

\begin{lemma}\label{lem:Hig_ext}
For any $i$, $D^g_i\subseteq V_P^{ext}$ and $E_i^g\subseteq W^{ext}\otimes\{g\}$.
\end{lemma}

\begin{proof}\upshape
This lemma is proved by induction.
First, our assumption $V_Q\subseteq W^{ext}$ ensures $E_0^g=V_Q\otimes\{g\}\subseteq W^{ext}\otimes\{g\}$, and we also have $D_0^g=\Phi(E_0^g)\subseteq V_P^{ext}$ since $\Phi((Q\ntie P)\otimes R)\subseteq [V_P]$ implies $\Phi(W^{ext}\otimes\{g\})\subseteq V_P^{ext}$ for $g\in R$.
Assume the lemma holds for each $j\leq i$.
We now show that the lemma also holds for $i+1$.
By assumption, $V_Q\ntie D^g_i\subseteq V_Q\ntie V_P^{ext}=W^{ext}$.
So by definition of $E^g_{i+1}$, we get $E_{i+1}^g\subseteq W^{ext}\otimes\{g\}$.
Furthermore, we observe $D_{i+1}^g=\Phi(E_{i+1}^g)\subseteq\Phi(W^{ext}\otimes\{g\})\subseteq V_P^{ext}$.
\end{proof}

\begin{lemma}
There is a computable function $\Gamma$ mapping each $g\in R$ to a $g$-computable sequence $\Gamma(g)=\{D^g_n\}_{n\in\omega}$ of $g$-computable trees.
\end{lemma}

\begin{proof}\upshape
We inductively show this lemma.
It suffices to show that we can construct $D^g_i$ from $E^g_i$ by uniformly $g$-computable way, since clearly $E^g_0$ is computable in $g$, and $D^g_i\mapsto E^g_{i+1}$ is uniformly $g$-computable.
Assume that $E_i^g\subseteq 2^{<\nn}\otimes\{g\}$ is given.
For each $\sigma\in 2^\nn$, if $\sigma\oplus(g\res|\sigma|)\in E^g_i$, then put $l(\sigma)=|\Phi(\sigma\oplus(g\res|\sigma|))|$.
If $\sigma\oplus(g\res|\sigma|)\not\in E^g_i$, then put $l(\sigma)=\infty$.
Note that $l:2^{<\nn}\to\nn\cup\{\infty\}$ is $g$-computable, since the notation $\Phi(\sigma)$ just means the computation of $\Phi$ restricted to step $|\sigma|$ with the oracle $\sigma$.
By Lemma \ref{lem:Hig_ext}, $\lim_nl(f\res n)=\infty$ for any $f\in 2^\nn$.
Because, for any $f$ with $f\oplus g\in[E_i^g]$, we have $\Phi(f\oplus g)\in[\Phi(E_i^g)]\subseteq\Phi([W]\otimes\{g\})\subseteq\Phi((Q\fr P)\otimes R)\subseteq P$, hence, $f\oplus g\in{\rm dom}(\Phi)$.
Therefore, by compactness, for each $n\in\nn$, there is $h_n\in\nn$ such that $l(\sigma)\geq n$ for each $\sigma\in 2^{<\nn}$ of length $h_n$.
We can compute $h^g_i(n)=h_n$ with the oracle $g$, since $l$ is $g$-computable.
Here, we can compute a $g$-computable index of $h^g_i$ from an index of $E_i^g$, uniformly in $i\in\nn$ and $g\in\nn^\nn$.
Thus, the relation $\tau\in D_i^g$ is equivalent to the $g$-computable condition that $\tau\subseteq\Phi(\sigma)$ for some $\sigma\in E_i^g$ of length $h^g_i(|\tau|)$, uniformly in $i\in\nn$ and $g\in\nn^\nn$.
Formally, the set $\{(\tau,i,g)\in 2^{<\nn}\times\nn\times\nn^\nn:\tau\in D^g_i\}$ is computable.
\end{proof}

Define $L_{D_i}$ as the set of all leaves of the tree $D_i^g$, and define $L_{E_i}$ as the set of all leaves of the tree $E_i^g$ for each $i$.

\begin{lemma}
Let $X$ be $D$ or $E$, and $i$ be any natural number.
For any $\rho\in L_{X_i}^g$, there are infinitely many nodes $\tau\in L_{X_{i+1}}^g$ which are extensions of $\rho$.
\end{lemma}

\begin{proof}\upshape
This lemma is proved by induction.
First we pick $\rho\in L_{E_0}=L_Q\lsup\{g\}=\{\sigma\oplus\tau:\sigma\in L_Q\;\&\;\tau\subset g\;\&\;|\sigma|=|\tau|\}$.
We note that $V_P$ is an infinite tree since $P$ is special.
By using our assumption $P\leq^{1}_{1}(Q\ntie P)\otimes R$ via $\Phi$ and the property $[V_Q]\otimes\{g\}\subseteq(Q\ntie P)\otimes R$, the tree $D_0^g=\Phi(E_0^g)$ has a path, i.e., it is infinite.
By definition, we have $E_1^g=V_P\ntie D_0^g\supseteq \rho\fr D_0^g$, and so $E_1^g$ has infinitely many extensions of $\rho$.
Now, we assume this lemma for $E$ and any $j\leq i$.
For a given $\rho\in L_{D_i}$, there is a node $\sigma\in E_i^g$ such that $\Phi(\sigma)=\rho$ by our definition of $D_i^g=\Phi(E_i^g)$.
Without loss of generality, we can pick $\sigma$ as a leaf of $E_i^g$.
By induction hypothesis, $\sigma$ has infinitely many extensions in $E_{i+1}^g$.
By lemma \ref{lem:Hig_ext}, we know $E_{i+1}^g\subseteq W^{ext}\otimes\{g\}$.
This implies that $\Phi(E_{i+1}^g(\supseteq\sigma))$ must be infinite whenever $E_{i+1}^g(\supseteq\sigma)$ is infinite.
We now remark that, for any $\sigma'\in\Phi(E_{i+1}^g(\supseteq\sigma))$, $\Phi(\sigma')\supseteq\Phi(\sigma)=\rho$.
Thus, $\Phi(E_{i+1}^g(\supseteq\sigma))$ gives infinitely many extensions of $\rho$, and our definition $D_{i+1}^g=\Phi(E_{i+1}^g)$ clearly implies the lemma for $D$ and $i$.
Now, we will show the lemma for $E$ and $i+1$.
By our definition of $E_{i+1}^g=(V_Q\ntie D_i^g)\otimes\{g\}$, every $\rho\in L_{E_{i+1}}$ must be of form $\rho=(\sigma\fr\tau)\oplus(g\res|\sigma\fr\tau|)$ for some $\sigma\in L_Q$ and $\tau\in L_{D_i}$.
So if $\tau\in L_{D_i}$ has infinitely many extensions in $L_{D_{i+1}}$ then $\rho=(\sigma\fr\tau)\oplus(g\res|\sigma\fr\tau|)$ has infinitely many extensions in $L_{E_{i+2}}$.
Thus, we have established the lemma for $E$ and $i+1$.
Now, the lemma follows by induction.
\end{proof}

\begin{lemma}
$D^g$ is a $g$-c.e. subtree of $V_P$ without dead ends for any $g\in P$.
Hence, $D^g$ has a $g$-computable path of $V_P$.
\end{lemma}

\begin{proof}\upshape
For a given $g$-c.e. tree $D^g=\{\delta_s\}_{s\in\omega}$, we first compute $\Delta(D^g)[0]=\delta_0$.
If $\Delta(D^g)[t]$ has already been defined, then we wait for stage $s$ with $\delta_s\supseteq\Delta(D^g)[t]$.
At this point, we extend $\Delta(D^g)[t+1]$ to $\sigma_s$.
We eventually reach this stage since $V_P$ has no dead ends.
This procedure is $g$-computable and $\Delta(D^g)=\lim_t\Delta(D^g)[t]$ must be a $g$-computable path of $D^g$.
\end{proof}

Then we define a computable functional $\Psi$ as $\Psi(g)=\Delta(\bigcup_n\Gamma(g))$ for any $g\in\nn^\nn$.
This witnesses $P\leq^{1}_{1}R$ as desired.
\end{proof}

\begin{cor}\label{cor:3a:HigHig}
For every special $\Pi^0_1$ set $P\subseteq 2^\nn$, there exists a $\Pi^0_1$ set $Q\subseteq 2^\nn$ such that $Q<^1_1P$ and $Q\equiv^1_{<\omega}P$.
\end{cor}

\begin{proof}\upshape
By Theorem \ref{thm:Hig}, we have $P\fr P<^1_1P\equiv^1_{<\omega}P$.
\end{proof}

\begin{definition}
Fix $\alpha,\beta,\gamma\in\{1,<\omega,\omega\}$.
An $(\alpha,\beta|\gamma)$-degree $\dg{a}\in\mathcal{P}^\alpha_{\beta|\gamma}$ has {\em the strong anticupping property} if there is a nonzero $(\alpha,\beta|\gamma)$-degree $\dg{b}\in\mathcal{P}^\alpha_{\beta|\gamma}$ such that, for any $(\alpha,\beta|\gamma)$-degree $\dg{c}$, if $\dg{a}\leq\dg{b}\vee\dg{c}$, then $\dg{a}\leq\dg{c}$.
\index{strong anticupping property}%
\end{definition}

\begin{cor}
Every nonzero $\dg{a}\in\mathcal{P}^1_1$ has the strong anticupping property.
\end{cor}

\begin{proof}\upshape
Fix $P\in\dg{a}$.
Let $\dg{b}$ be the $(1,1)$-degree of $P\fr P$.
Then, by Theorem \ref{thm:Hig}, for any $(1,1)$-degree $\dg{c}$, if $\dg{a}\leq\dg{b}\vee\dg{c}$, then $\dg{a}\leq\dg{c}$.
\end{proof}

For $\Pi^0_1$ sets, if $P$ and $Q$ are disjoint, then $P\oplus Q$ is equivalent to $P\cup Q$ modulo the $(1,1)$-equivalence, since $\mathcal{P}^1_1=\mathcal{P}/{\rm dec}_{\rm p}^{<\omega}[\Pi^0_1]$.
However, if $P$ and $Q$ are not $\Pi^0_1$, the above claim is false, in general.

\begin{prop}\label{prop:5:refref2}
For any special $\Pi^0_1$ set $P\subseteq 2^\nn$, there exists a $(\Pi^0_1)_2$ set $Q\subseteq 2^\nn$ such that $P\cap Q=\emptyset$ but $P\cup Q<^1_1P\oplus Q$.
\end{prop}

\begin{proof}\upshape
For the item $1$, we assume that $\Gamma_i:Q_i\to P$ is computable for each $i<m$ with a finite $\Pi^0_1$ partition $\{Q_i\}_{i<m}$ of $Q$.
Then we can extend each function $\Gamma_i$ to a total computable function $\Gamma^*_i$.
For a given $g\in Q$, we wait for $g\not\in Q_j$ for $m-1$ many $j$'s.
By this way, we can calculate such $i<m$ such that $g\not\in Q_j$ for each $j\not=i$.
Then $\Gamma_i(g)\in Q$.
For the item $2$, put $Q=(P\tie P)\setminus P$.
For any $g\in Q$, there is a leaf $\rho\in L_P$ such that $\rho\subset g$.
So we wait for such a leaf $\rho\in L_P$.
Then $g^{\shft |\rho|}$ belongs to $P$.
Hence, $P\leq^1_1Q$.
Thus, we have $P\leq^1_1P\oplus Q$, while $P\cup Q=P\tie P<^1_1P$ by Theorem \ref{thm:Hig}.
\end{proof}

\begin{definition}
The operation $\tie:\mathcal{P}(\nn^\nn)\to\mathcal{P}(\nn^\nn)$ is defined as the map sending $P$ to $P^\tie={\sf CPA}\fr P$, where recall that {\sf CPA} denotes the set of all complete consistent extensions of Peano Arithmetic, and it is a $(1,1)$-complete $\Pi^0_1$ subset of $2^\nn$.
\index{$P^\tie$}%
\end{definition}

By the previous theorem, the derived set $P^\tie$ does not $(1,1)$-cup to $P$ whenever $P$ is $\Pi^0_1$.
In particular, we have $P^\tie<^{1}_{1}P$.
Recall from Part I \cite[Proposition 38]{HK_PartI} that the operator $\tie:\mathcal{P}^1_1\to\mathcal{P}^1_1$ introduced by $(\deg^1_1(P))^\tie=\deg^1_1(P^\tie)$ is well-defined.
Moreover, $\mathcal{P}^1_1(\leq\dg{1}^\tie)=\{\dg{a}\in\mathcal{P}^1_1:\dg{a}\leq\dg{1}^\tie\}$ is a principal prime ideal consisting of tree-immune-free Medvedev degrees \cite{CKWW}.
Here, recall that a $\Pi^0_1$ set $P\subseteq 2^\nn$ is tree-immune if $T_P^{ext}$ contains no infinite computable subtree.
\index{tree-immune}%
Then, we also observe the following.

\begin{prop}\label{prop:3a:injective-pre}
Fix $\Pi^0_1$ sets $P_0,P_1,Q_0,Q_1\subseteq 2^\nn$, and assume that $P_0\fr P_1\leq^1_1Q_0\fr Q_1$.
Then, either $P_0\leq^1_1Q_0$ or $P_1\leq^1_1Q_1$ holds.
Moreover, if $P_0$ is tree-immune and $Q_0$ is nonempty, then $P_1\leq^1_1Q_1$.
\end{prop}

\begin{proof}\upshape
Assume that $P_0\fr P_1\leq^1_1Q_0\fr Q_1$ via a computable function $\Phi$.
If $\Phi(\rho)\in T_{P_0}^{ext}$ for any leaf $\rho\in L_{Q_0}$, then $\Phi(g)\in[T_{P_0}]$ for any $g\in[Q_0]$, i.e., $P_0\leq^1_1Q_0$.
If $\Phi(\rho)\not\in T_{P_0}^{\rm ext}$ for some leaf $\rho\in L_{Q_0}$, then there are only finitely many strings of $T_{P_0}$ extending $\Phi(\rho)$.
Thus, $[T_{P_0}\ntie T_{P_1}]\cap[\Phi(\rho)]$ is essentially a sum of finitely many $P_1$'s, hence it is $(1,1)$-equivalent to $P_1$.
Since a computable functional $\Phi$ maps $\rho\fr Q_1$ to the above class, obviously, $P_1\leq^1_1Q_1$.
If $P_0$ is tree-immune, then $\Phi(\rho)\not\in T_{P_0}^{\rm ext}$ for some leaf $\rho\in L_{Q_0}$, since otherwise the image of $T_{Q_0}$ under $\Phi$ is included in $T_{P_0}$, and clearly it is infinite and computable.
Therefore, we must have $P_1\leq^1_1Q_1$.
\end{proof}

\begin{cor}\label{prop:3a:injective}
The operator $\tie:\dg{a}\mapsto\dg{a}^\tie$ is injective.
Hence, $\tie$ provides an order-preserving self-embedding of the $(1,1)$-degrees $\mathcal{P}^1_1$ of nonempty $\Pi^0_1$ subsets of $2^\nn$.
\end{cor}

\begin{proof}\upshape
By Cenzer-Kihara-Weber-Wu \cite{CKWW}, ${\sf CPA}$ is tree-immune.
Therefore, by Proposition \ref{prop:3a:injective-pre}, ${\sf CPA}\fr Q=Q^\tie\leq^1_1P^\tie={\sf CPA}\fr P$ implies $Q\leq^1_1P$.
\end{proof}

It is natural to ask whether the image of $\mathcal{P}^1_1$ under the operator is exactly $\mathcal{P}^1_1(\leq\dg{1}^\tie)$.
Unfortunately, it turns out to be false.

\begin{prop}
There exists a non-tree-immune $\Pi^0_1$ set $Q\subseteq 2^\nn$ such that no nonempty $\Pi^0_1$ sets $P_0,P_1\subseteq 2^\nn$ satisfy $Q\equiv^1_1P_0\fr P_1$.
In particular, the operator $\tie:\mathcal{P}^1_1\to\mathcal{P}^1_1(\leq\dg{1}^\tie)$ is not surjective.
\end{prop}

\begin{proof}
Let $\{Q_n\}_{n\in\nn}$ be a computable sequence of nonempty $\Pi^0_1$ subsets of $2^\nn$ such that $\bigoplus_{n\in\nn}Q_n$ forms a Turing antichain.
Define $Q=Q_0\fr\{Q_{n+1}\}_{n\in\nn}$.
Suppose that there exist nonempty $\Pi^0_1$ sets $P_0,P_1\subseteq 2^\nn$ with $Q\equiv^1_1P_0\fr
 P_1$.
Choose computable functions $\Phi:Q\to P_0\fr P_1$ and $\Psi:P_0\fr P_1\to Q$.
Since $\{Q_n\}_{n\in\nn}$ forms a Turing antichain, $\Psi\circ\Phi$ is an identity function on $Q$.
Consider two cases.

The first case is that $\Phi(Q)\subseteq P_0$.
In this case, $\Psi(P_0)=Q$ since $\Psi\circ\Phi$ is identity.
Thus, every string in $T_Q^{ext}$ is extended by some string in $\Psi(T_{P_0})$.
Moreover, he condition $T_{P_0}\subseteq T_{P_0\fr P_1}^{ext}$ implies $\Psi(T_{P_0})\subseteq T_Q^{ext}$.
Therefore, $\Psi(T_{P_0})=T_Q^{ext}$.
Hence $T_Q^{ext}$ is a computable tree without leaves.
But this is impossible since $Q$ contains no computable elements.

The second case is that $\Phi(Q)\not\subseteq P_0$, that is, there exists $f\in Q$ such that $\Phi(f)\in\rho\fr P_1$, where $\rho$ is a leaf of $T_{P_0}$.
We have $f\equiv_T\Phi(f)$ since $\Psi\circ\Phi$ is identity.
Note that we may assume that $f=\rho_k\fr f_k$ for some leaf $\rho_k\in T_{Q_0}$ and $f_k\in Q_k$, since even if $f\in Q_0$ the string $\Phi(f\res n)$ extends $\rho$ for sufficiently large $n$, and replace $f$ with a string extending $f\res n$ which is removed from $Q_0$.
On the one hand, $f$ is the only element in $Q$ computable in $f$.
On the other hand, every $\sigma\in T_{P_0}$ always extends to an element of $P$ which is Turing equivalent to $f$.
Thus, for every $\sigma\in T_{P_0}$, the string $\Phi(\sigma)$ must be compatible with $\rho_k$.
Hence, $\Psi(P)\subseteq\rho_k\fr Q_k$.
This contradicts the property that $\Psi\circ\Phi(Q)=Q$.
\end{proof}

Recall from Part I that $P^{(a)}$ is the $a$-th derivative of $P$, i.e., the $a$-th iterated concatenation starting from $P$.
\index{derivative}\index{$P^{(a)}$}%

\begin{prop}\label{prop:wfHig}
For any special $\Pi^0_1$ set $P\subseteq 2^\nn$, if $a,b\in\mathcal{O}$ and $a<_{\mathcal{O}}b$, then $P^{(b)}$ does not $(1,1)$-cup to $P^{(a)}$, i.e., for any set $R\subseteq\nn^\nn$, if $P^{(a)}\leq^1_1P^{(b)}\otimes R$ then $P^{(a)}\leq^1_1R$.
\end{prop}

\begin{proof}\upshape
The assumption $a<_\mathcal{O}b$ implies $2^a\leq_\mathcal{O}b$.
Therefore, we have $P^{(b)}\leq^1_1P^{(2^a)}$.
By Theorem \ref{thm:Hig}, $P^{(2^a)}$ does not $(1,1)$-cup to $P^{(a)}$.
Thus, $P^{(b)}$ does not $(1,1)$-cup to $P^{(a)}$.
\end{proof}

Fix any notation ${\tt omega}\in\mathcal{O}$ such that $|\Phi_{\tt omega}(n)|_\mathcal{O}=n$ for each $n\in \nn$.
Note that $|{\tt omega}|_\mathcal{O}=\omega$.

\begin{prop}\label{prop:breductiona}
Let $P$ be a special $\Pi^0_1$ subset of $2^\nn$.
For any $\Pi^0_1$ set $R\subseteq 2^\nn$, if $P\leq^1_{<\omega}P^{({\tt omega})}\otimes R$, then $P\leq^1_{<\omega}R$.
\end{prop}

\begin{proof}\upshape
As seen in Part I \cite{HK_PartI}, for every $\Pi^0_1$ sets $P,Q\subseteq 2^\nn$, $P\leq^1_{<\omega}$ implies $P\leq^1_{tt,<\omega}Q$.
Since $P^{({\tt omega})}\otimes R$ is $\Pi^0_1$, $P\leq^1_{<\omega}P^{({\tt omega})}\otimes R$ implies $P\leq^1_{tt,<\omega}P^{({\tt omega})}\otimes R$, and then there is a $(1,n)$-truth-table function $\Gamma:P^{({\tt omega})}\otimes R\to P$ for some $n\in\nn$.
In particular, $\Gamma:(\rho_{n+1}\fr P^{(\Phi_{\tt omega}(n+1))})\otimes R\to P$, where $\rho_{n+1}$ is the $(n+1)$-th leaf of $T_P$.
By modifying $\Gamma$, we can easily construct a $(1,n)$-truth-table function $\Theta:P^{(n+1)}\otimes R\to P$.

Assume that $\Theta$ is $(1,n)$-truth-table via $n$ many total computable functions $\Theta_0,\dots,\Theta_{n-1}$.
We define a computable function $\gamma:n\times 2^{<\nn}\to 2^{<\nn}$ as follows.
If $\Theta_m(\sigma)\in T_P$, then put $\gamma(m,\sigma)=\Theta_m(\sigma)$.
If $\Theta_m(\sigma)\supsetneq\rho$ for some $\rho\in L_P$, then we define $\gamma(m,\sigma)$ to be such $\rho$.
Let $z(\sigma)=\min\{m<n:\Theta_m(\sigma)\in T_P\}$.
Then, for $\sigma\in 2^{<\nn}$, the value $\Phi(\sigma)$ is defined by $\concat_{m\leq z(\sigma)}\gamma(m,\sigma)$.
Then $\Phi$ ensures that $P^{(n)}\leq^{1}_{1}P^{(n+1)}\otimes R$.
By Theorem \ref{thm:Hig}, we have $P^{(n)}\leq^{1}_{1}R$.
Consequently, $P\leq^1_{<\omega}R$.
\end{proof}


\begin{cor}
For every $a\in\mathcal{O}$ there exists a computable function $g$ such that, for any $\Pi^0_1$ index $e$, if $P_e$ is special then the following properties hold.
\begin{enumerate}
\item $P_{g(e,b)}<^1_1P_{g(e,c)}$ holds for every $c<_\mathcal{O}b<_\mathcal{O}a$, indeed, $P_{g(e,b)}$ does not $(1,1)$-cup to $P_{g(e,c)}$.
\item $P_{g(e,b)}\equiv^1_\omega P_{g(e,c)}$ for every $b,c<_\mathcal{O}a$.
\end{enumerate}
\end{cor}

\begin{proof}\upshape
Let $g(e,b)$ be an index of $P_e^{(b)}$.
Then, the desired conditions follow from Proposition \ref{prop:wfHig}.
\end{proof}

For any reducibility notion $r$, and any ordered set $(I,\leq_I)$, a sequence $\{\dg{a}_i\}_{i\in I}$ of $r$-degrees is {\em $r$-noncupping} if, for any $i<_Ij$, the condition $\dg{a}_i\leq_r\dg{b}$ must be satisfied whenever $\dg{a}_i\leq_r\dg{a}_j\vee\dg{b}$, for any $r$-degree $\dg{b}$.
\index{noncupping}%
In particular, any $r$-noncupping sequence is strictly decreasing, in the sense of $r$-degrees.

\begin{cor}
For any nonzero $(1,\omega)$-degree $\dg{a}\in\mathcal{P}^1_\omega$, there is a $(1,1)$-noncupping computable sequence of $(1,1)$-degrees inside $\dg{a}$ of arbitrary length $\alpha<\omega_1^{CK}$.
\qed
\end{cor}


\subsection{Infinitary Disjunctions along the Straight Line}

We next see the LEVEL 4 separation between $[\mathfrak{C}_T]^{1}_{\omega|<\omega}$ and $[\mathfrak{C}_T]^1_{\omega}$.
Indeed, we show the non-existence of a $(<\omega,1)$-contiguous $(1,\omega)$-degree.
We introduce the {\sf LCM} disjunctions of $\{P_i\}_{i\in\nn}$ as $\btie_{n\in\nn}P_n=\bigcup_{n\in\nn}(P_0\ntie\dots\ntie P_n)$.
\index{LCM disjunction@{\sf LCM} disjunction}\index{$\btie_{n\in\nn}P_n$}%
This is a straightforward infinitary iteration of the concatenations.
If $P_n=P$ for all $n\in\nn$, we write $\btie P$ instead of $\btie_nP_n$.
\index{$\btie P$}%

\begin{prop}
Let $\{P_i\}_{i\in\nn}$ be a computable collection of nonempty $\Pi^0_1$ subsets of $2^\nn$.
Then $\btie_nP_n$ is $(1,1)$-equivalent to a dense $\Sigma^0_2$ set in Cantor space $2^\nn$.
\end{prop}

\begin{proof}\upshape
Let $S$ denote the set $\{g\in(\nn\cup\{\sharp\})^\nn:(\exists n\in\nn)\;({\tt count}(g)=n\;\&\;{\tt tail}(g)\in P_n)\}$, where ${\tt count}(g)=\#\{n\in\nn:g(n)=\sharp\}$.
Then, $S$ is clearly a $\Sigma^0_2$ subset of $\{0,1,\sharp\}^\nn$.
For any $\sigma\in\{0,1,\sharp\}^{<\nn}$, we have $\sigma\fr\lrangle{\sharp}\fr h\in S$ for any $h\in P_{{\tt count}(\sigma)+1}$
Thus, $S$ intersects with any clopen set.
\end{proof}

\begin{example}
Let ${\sf MLR}$ denote the set of all Martin-L\"of random reals.
\index{${\sf MLR}$}%
Then ${\sf MLR}\equiv^1_1\btie P$ for any nonempty $\Pi^0_1$ set $P\subseteq{\sf MLR}$, by Ku\u{c}era-G\'acs Theorem (see \cite{NieR}), while ${\sf MLR}<^1_1P$ for any $\Pi^0_1$ set $P\subseteq{\sf MLR}$ as follows.
\end{example}

\begin{prop}[Lewis-Shore-Sorbi \cite{LSS}]
No somewhere dense set in Baire space $(1,1)$-cup to a closed set in Baire space.
In other words, for any somewhere dense set $D\subseteq\nn^\nn$, any closed set $C\subseteq\nn^\nn$, and any set $R\subseteq\nn^\nn$, if $C\leq^{1}_{1}D\otimes R$ then $C\leq^1_1R$.\qed
\end{prop}

\begin{prop}
For any somewhere dense set $D\subseteq\nn^\nn$ and any special closed set $C\subseteq\nn^\nn$, we have $C\not\leq^{<\omega}_1D$.
\end{prop}

\begin{proof}\upshape
If $\{D_i\}_{i<b}$ is a finite partition of $D$, then $\bigcup_{i<b}{\rm Cl}_{\nn^\nn}(D_i)={\rm Cl}_{\nn^\nn}(D)$, where the topological closure of $D$, in the standard Baire topology on $\nn^\nn$, is denoted by ${\rm Cl}_{\nn^\nn}(D)$.
To show this claim, for every $x\in{\rm Cl}_{\nn^\nn}(D)$ we have a sequence $\{x_k\}_{k\in\nn}\subseteq D$ converging to $x$.
By pigeonhole principle, there is $i<b$ such that there are infinitely many $k$ such that $x_k\in D_i$.
For such $i$, clearly $x\in{\rm Cl}_{\nn^\nn}(D_i)$.
However, since the somewhere density of $D$ implies that ${\rm Cl}_{\nn^\nn}(D)$ contains some clopen set, and hence ${\rm Cl}_{\nn^\nn}(D_i)$ contains a computable element $r$ for some $i$.
Additionally, ${\rm Cl}_{\nn^\nn}(C)=C$ since $C$ is closed.
If $C\leq^{<\omega}_{1}D\otimes R$, then there is a finite partition $\{D_i\}_{i<b}$ of $D$ such that $C\leq^{1}_{1}D_i$ via a computable function $\Phi_{e(i)}$.
Fix $i$ such that ${\rm Cl}_{\nn^\nn}(D_i)$ contains a computable element.
Therefore, $C={\rm Cl}_{\nn^\nn}(C)\leq^{1}_{1}{\rm Cl}_{\nn^\nn}(D_i)\supseteq\{r\}$ via $\Phi^f_{e(i)}$.
Hence, $C$ contains a computable element.
\end{proof}

Especially, if $P$ is a special $\Pi^0_1$ set, then there is no nonzero $(<\omega,1)$-degree of $\Pi^0_1$ subsets of $2^\nn$ below the $(<\omega,1)$-degree of $\btie P$.
We will see that the set $\btie P$ has a stronger property.

\begin{definition}
A sequence $\lrangle{t_n}_{n\in\nn}$ of finite strings is {\em a timekeeper} if there is a uniformly c.e.\ collection of finite sets, $\{V_n\}_{n\in\nn}$, such that, for any $n\in\nn$, $|t_n|=|V_n|$ and $t_n(i)$ is given as the stage at which the $i$-th element is enumerated into $V_n$, for each $i<|t_n|$.
\index{timekeeper}%
\end{definition}

\begin{definition}
For a finite string $\tau\in\nn^{<\nn}$, {\em the $\tau$-delayed $(|\tau|+1)$-derivative} $P^{(\tau)}$ is inductively defined as follows:
\index{derivative!delayed}\index{$P^{(\tau)}$}%
\begin{align*}
P^{(\tau\res 0)}=P;& &P^{(\tau\res i+1)}=\bigcup\{\sigma\fr P:\sigma\in L_{P^{(\tau\res i)}}\;\&\;|\sigma|\geq \tau(i)\}\mbox{ for each }i<|\tau|.
\end{align*}
\end{definition}

\begin{prop}
If $\tau(m)=0$ for each $m<|\tau|$, then $P^{(\tau)}=P^{(|\tau|+1)}$.
\end{prop}

\begin{proof}\upshape
Straightforward from the definition.
\end{proof}

\begin{lemma}
For any timekeeper $\lrangle{t_n}_{n\in\nn}$, the following conditions hold.
\begin{enumerate}
\item $P^{(t_n)}\subseteq P^{(|t_n|+1)}$.
Hence, $P\fr\{P^{(t_n)}\}_{n\in\nn}\subset\btie P$.
\item $P^{(t_n)}$ is $\Pi^0_1$, uniformly in $n$.
Hence, $P\fr\{P^{(t_n)}\}_{n\in\nn}$ is $\Pi^0_1$.
\end{enumerate}
\end{lemma}

\begin{proof}\upshape
(1) Straightforward.
(2) We construct a computable tree $T^{(t_n)}$ corresponding to $P^{(t_n)}$.
Each $\sigma\in 2^\nn$ can be represented as $\sigma=\rho_0\fr\rho_1\fr\dots\fr\rho_k\fr\tau$, where $\rho_m\in L_P$ for any $m\leq k$, and $\lrangle{}\not=\tau\in T_{P}$.
Then $\sigma\in T^{(t_n)}$ if and only if $t_n(k)$ holds by stage $|\rho_0\fr\rho_1\fr\dots\fr\rho_k|$.
Then $T^{(t_n)}$ is a computable tree, and clearly $P^{(t_n)}=[T^{(t_n)}]$.
\end{proof}

\begin{remark}
Timekeeper arises because of {\em finite injury priority argument}.
The delayed derivative construction is useful to bound the complexity of the set, since the recursive meet $P\fr\{P^{(|t_n|+1)}\}_{n\in\nn}$ of the standard derivatives along a timekeeper $\{t_n\}_{n\in\nn}$ is only assured to be $\Pi^{0,\emptyset'}_1$.
\end{remark}

\begin{theorem}\label{thm:contig:b-l}
Let $P$ be any $\Pi^0_1$ subset of $2^\nn$.
Then, for every special $\Pi^0_1$ set $Q\subseteq 2^\nn$, there exists a $\Pi^0_1$ set $\widehat{P}\subseteq\btie P$ such that $Q\not\leq^1_{<\omega}\widehat{P}$.
\end{theorem}

\begin{proof}\upshape
Let $Q$ be a special $\Pi^0_1$ set, and $P$ be a given $\Pi^0_1$ set.
By a uniformly computable procedure, from $P$, we will construct a timekeeper $\{t_n\}_{n\in\nn}$.
The desired class $\widehat{P}$ will be given by $\widehat{P}=P\fr\{P^{(t_n)}\}_{n\in\nn}$.

\begin{req}\upshape
We need to ensure, for all $n\in\nn$, the following:
\[R_n\;:\;Q\leq^1_{<\omega}\widehat{P}\mbox{ via }n\;\rightarrow\;(\exists\Delta_n)\;\Delta_n\in Q.\]
\end{req}

\noindent
{\bf Action of an $R_n$-strategy.}
Fix an effective enumeration $\{\rho_n:n\in\nn\}$ of all leaves of $T_P$.
An $R_n$-strategy uses nodes extending the $n$-th leaf $\rho_n$ of $T_P$, and it constructs a finite sequence $t_n[s]$, a sequence $\tau_n[s]$ of strings, and a computable functional $\Delta_n$.
For any $n$, put $t_n[0]=\lrangle{}$, and $\tau_n[0]=\rho_n$ at stage $0$.
{\em An $R_n$-strategy acts at stage $s+1$} if the following condition holds:
\[(\exists\rho\in T_P^s)(\exists e<n)\ \Phi_e(\tau_n[s]\fr\rho)\in T_Q\ \&\ \Phi_e(\tau_n[s]\fr\rho)\supsetneq\Phi_e(\tau_n[s]).\]
If an $R_n$-strategy acts at stage $s+1$ then, for a witness $\rho\in T_P^s$, we pick $\rho^*\in L_P$ extending $\rho$.
Then let us define $\tau_n[s+1]=\tau_n[s]\fr\rho^*$, $t_n[s+1]=t_n[s]\fr\lrangle{|\tau_n[s+1]|}$, and $\Delta_{e,n}\res l=\Phi_e(\tau_n[s+1])$, where $l$ is the length of $\Phi_e(\tau_n[s+1])$.
Otherwise, $t_n[s+1]=t_n[s]$, $\tau_i[s+1]=\tau_i[s]$.
Note that the mapping $(n,m)\mapsto\tau_n(m)$ is partial computable.
At the end of the construction, set $t_n=\bigcup_st_n[s]$.
As mentioned above, $\widehat{P}$ is defined by $\widehat{P}=P\fr\{P^{(t_n)}\}_{n\in\nn}$.

\begin{claim}
An $R_n$-strategy acts at most finitely often for each $n$.
\end{claim}

Clearly $\tau_n=\bigcup_s\tau_n[s]$ is a computable string.
If $R_n$ acts infinitely often, then $\Delta_{e,n}=\Phi_e(\tau_n)\in Q$ for some $e<n$ by our choice of $\tau_n$.
Since $\Phi_e(\tau_n)$ is computable, $Q$ contains a computable element.
However, this contradicts our assumption that $Q$ is special.
Therefore, we concludes the claim.
As a corollary, $\lrangle{t_n}_{n\in\nn}$ is a timekeeper.

\begin{claim}
$P\not\leq^1_{<\omega}\widehat{P}$.
\end{claim}

Let $\tau_n=\bigcup_s\tau_n[s]$.
By induction we show that $\tau_n\in\rho_n\fr T^{ext}_{P^{(t_n)}}$.
First we have the following observation:
\[\tau_n[0]=\rho_n\in\rho_n\fr T^{ext}_{P}=\rho_n\fr T^{ext}_{P^{(t_n\res 0)}}\subseteq\rho_n\fr T^{ext}_{P^{(t_n[0])}}.\]
Assume $\tau_n[s]\in\rho_n\fr T^{ext}_{P^{(t_n[s])}}$.
If $\tau_n[s+1]=\tau_n[s]\fr\rho^*$ for $\rho^*\in L_P$ then $t_n[s+1]=t_n[s]\fr\lrangle{|\tau_n[s+1]|}$.
In particular $\tau_n[s+1]\in\rho_n\fr L_{P^{(t_n[s+1]\res |t_n[s]|)}}$ and $|\tau_n[s+1]|\geq t_n[s+1](|t_n[s]|)$.
Hence, by the definition of $P^{(t_n[s+1])}$, it is easy to see that $\tau_n[s+1]\fr P\subseteq\rho_n\fr P^{(t_n[s+1])}$.
Thus, $\tau_n[s+1]\in\rho_n\fr T^{ext}_{P^{(t_n[s+1])}}$.
So we obtain $\tau_n\in\rho_n\fr T^{ext}_{P^{(t_n)}}$ and by our construction of $\tau_n$ there is no $\rho\in P$ and $e<n$ such that $\Phi_e(\tau_n\fr\rho)\supsetneq\Phi_e(\tau_n)$.
Since $\Phi_e(\tau_n)$ is a finite string, for any $g\in\rho_n\fr T^{ext}_{P^{(t_n)}}\subset\hat{P}$ extending $\tau_n$, $\Phi_e(g)$ is also a finite string.
Consequently, this $g$ witnesses that $P\not\leq^1_{<\omega}\widehat{P}$.
\end{proof}

\begin{cor}\label{cor:5:l-bref}
\begin{enumerate}
\item For every special $\Pi^0_1$ set $P\subseteq 2^\nn$, we have $\btie P<^{<\omega}_1P\equiv^1_\omega\btie P$.
\item For every special $\Pi^0_1$ set $P\subseteq 2^\nn$ there exists a $\Pi^0_1$ set $Q\subseteq 2^\nn$ with $Q<^{<\omega}_1P\equiv^1_\omega Q$.
\end{enumerate}
\end{cor}

\begin{proof}\upshape
By applying Theorem \ref{thm:contig:b-l} to $Q=P$, we have $P\not\leq^1_{<\omega}\widehat{P}\geq^1_{<\omega}\btie P$.
Moreover, $P\oplus\widehat{P}<^{<\omega}_1P\equiv^1_{\omega}P\oplus\widehat{P}$.
\end{proof}

%% file: NRMP_fullproof3b.tex

\subsection{Infinitary Disjunctions along ill-Founded Trees}

We next show the LEVEL 4 separation between $[\mathfrak{C}_T]^1_\omega$ and $[\mathfrak{C}_T]^{<\omega}_{\omega}$.
The following theorem concerning the hyperconcatenation $\htie$ and the $(1,\omega)$-reducibility $\leq^{1}_{\omega}$ is a counterpart of Theorem \ref{thm:Hig} concerning the concatenation $\tie$ and the $(1,1)$-reducibility $\leq^{1}_{1}$.

\begin{theorem}\label{thm:hyperHig}
For every special $\Pi^0_1$ sets $P,Q\subseteq 2^\nn$, and for any $R$, if $P\leq^{1}_{\omega}(Q\htie P)\lsup R$ then $P\leq^{1}_{\omega}R$ holds.
\end{theorem}

\begin{proof}\upshape
Let $T_{\htie}$ denote the corresponding computable tree for $Q\htie P$.
{\em The heart of $T_{\htie}$}, $T_{\htie}^\heartsuit$, is the set of all strings $\gamma\in T_{\htie}$ such that $\gamma\subseteq\concat_{i<n}(\sigma_i\fr\lrangle{\tau(i)})$ for some $\{\sigma_i\}_{i<n}\subseteq L_P$, and $\tau\in T_Q^{ext}$.
\index{heart!of $T_{\htie}$}\index{$T_{\htie}^\heartsuit$}%
Now we assume $P\leq^{1}_{\omega}(Q\htie P)\lsup R$ via a learner $\Psi$.
To show the theorem it is needed to construct a new learner $\Delta$ witnessing $P\leq^{1}_{\omega}R$.
Fix $g\in R$.

\begin{lemma}\label{lem:3b:hconc1}
There exists a string $\rho\in T_{\htie}^\heartsuit$ such that, for every $\tau\in T_{\htie}^\heartsuit$ extending $\rho$, we have $\Psi(\rho\oplus(g\res|\rho|))=\Psi(\gamma)$ for any $\gamma$ with $\rho\oplus(g\res|\rho|)\subseteq\gamma\subseteq\tau\oplus(g\res|\tau|)$.
\end{lemma}

\begin{proof}\upshape
If Lemma \ref{lem:3b:hconc1} is false, we can inductively define an increasing sequence $\{\tau_i\}_{i\in\omega}$ of strings.
First let $\tau_0=\lrangle{}$, and $\tau_{i+1}$ be the least $\tau\supsetneq\tau_i$ such that $\tau\in T_{\htie}^\heartsuit$ and $\Psi(\tau\oplus(g\res(|\tau|+i)))\not=\Psi(\rho\oplus(g\res(|\rho|+j)))$ for some $i,j<2$.
Since $\bigcup_i\tau_i\in Q\htie P$, clearly $(\bigcup_i\tau_i)\oplus g\in (Q\htie P)\lsup R$.
However, based on the observation $(\bigcup_i\tau_i)\oplus g$, the learner $\Psi$ changes his mind infinitely often.
This means that his prediction $\lim_n\Psi((\bigcup_i\tau_i)\oplus g)$ diverges.
This contradicts our assumption that $P\leq^{1}_{\omega}(Q\htie P)\lsup R$ via the learner $\Psi$.
Thus, our claim is verified.
\end{proof}

Lemma \ref{lem:3b:hconc1} can be seen as an analogy of an observation of Blum-Blum \cite{BlBl} in the theory of inductive inference for total computable functions on $\nn$.
Such $\rho$ is sometimes called {\em a locking sequence}.
\index{locking sequence}%

\begin{lemma}\label{lem:3b:hconc2}
There exist an effective procedure $\Theta:\nn^\nn\times 2^{<\nn}\times 2\to\nn^\nn$ and a $\Pi^0_1$ condition $\varphi$ such that, for any $g\in Q$, $\varphi(g,\rho,m)$ holds for some $\rho\in 2^{<\nn}$, and $m<2$, and that for any $\rho\in 2^{<\nn}$ and $m\in\nn$, if $\varphi(g,\rho,m)$ holds, then $\Theta(g,\rho,m)\in P$.
\end{lemma}

\begin{proof}\upshape
The desired condition $\varphi(g,\rho,m)$ is given by the conjunction of the following three conditions.
\begin{enumerate}
\item $\rho$ is of the form $\concat_{i<n}(\sigma_i\fr\lrangle{\tau(i)})$ for some $\{\sigma_i\}_{i<n}\subseteq L_P$, $\tau\in T_Q^{ext}$, and $|\tau|=n$.
\item $\tau\fr\lrangle{m}\in T_Q^{ext}$.
\item $\Psi(\rho\oplus(g\res|\rho|))=\Psi(\gamma)$ for any $\gamma\in(\rho\fr T_P\fr\lrangle{m}\fr T_P)\lsup\{g\}$.
\end{enumerate}

The first two conditions are clearly $\Pi^0_1$, and since $\Psi$ is total computable, the last condition is also $\Pi^0_1$.
Consequently, $\varphi$ is $\Pi^0_1$.
We first show that $\varphi(g,\rho,m)$ holds for some $\rho\in 2^{<\nn}$ and $m\in\nn$.
Pick a locking sequence $\rho\in T_{\htie}^\heartsuit$ forcing to stop changing the mind of $\Psi$, as in the previous claim.
Without loss of generality, we can assume that $\rho$ satisfies the condition (1).
Since $\tau\in T_Q^{ext}$, there exists $m\in\omega$ such that $\tau\fr\lrangle{m}\in T_Q^{ext}$, and this $m$ satisfies the condition (2).
From conditions (1) and (2), we conclude that $\rho\fr P\fr\lrangle{m}\fr P=(\rho\fr P)\cup(\rho\fr\bigcup_{\sigma\in L_P}\sigma\fr\lrangle{m}\fr P)\subseteq T_{\htie}^\heartsuit$, and so condition (3) is satisfied.
Since we assume that $P\leq^{1}_{\omega}(Q\htie P)\lsup\{g\}$ via the learner $\Psi$, if $\varphi(g,\rho,m)$ is satisfied, then the following holds.
\[P\leq^{1}_{1}(\rho\fr P\fr\lrangle{m}\fr P)\lsup\{g\}\mbox{ via }\Phi_{\Psi(g\res|\rho|\oplus\rho)}.\]

Our proof process in Theorem \ref{thm:Hig} is effective with respect to $g$, $m$, and an index of $\Phi_{\Psi(g\res|\rho|\oplus\rho)}$ which are calculated from $g$, $\rho$, and an index of $\Psi$.
To see this, recall our proof in Theorem \ref{thm:Hig}.
Define $V_P^m=T_P\cup\{\rho\fr\lrangle{m}:\rho\in L_P\}$.
\begin{align*}
E^{g,\rho,m}_0&=V_P^m\otimes\{g\}; & D^{g,\rho,m}_0&=\Phi_{\Psi(g\res|\rho|\oplus\rho)}(E^{g,\rho,m}_0).\\
E^{g,\rho,m}_{i+1}&=(V_P^m\ntie D^{g,\rho,m}_i)\otimes\{g\}; & D^{g,\rho,m}_{i+1}&=\Phi_{\Psi(g\res|\rho|\oplus\rho)}(E^{g,\rho,m}_{i+1}).
\end{align*}

Then, as in the proof of Theorem \ref{thm:Hig}, $D^{g,\rho,m}=\bigcup_{i\in\nn}D^{g,\rho,m}_{i+1}$ is a subtree of $V_P$, and it has no dead ends.
Moreover, this construction is clearly c.e.\ uniformly in $g$, $\rho$, and $m$.
Therefore, we can effectively choose an element $\Theta(g,\rho,m)\in[D^{g,\rho,m}]\subseteq P$, uniformly in $g$, $\rho$, and $m$.
\end{proof}

Now, a procedure to get $P\leq^{1}_{\omega}R$ is follows.
For given $g\in Q$, on {\em the $i$-th challenge} of a learner $\Delta$, the learner $\Delta$ chooses the lexicographically $i$-th least pair $\lrangle{\rho,m}\in 2^{<\nn}\times\nn$, and $\Delta$ calculates an index $e(\rho,m)$ of the computable functional $g\mapsto\Theta(g,\rho,m)$, that is to say, $\Delta(g\res s)=e(\rho,m)$ at the current stage $s$.
At each stage in the $i$-th challenge, the learner $\Delta$ tests whether the $\Pi^0_1$ condition $\varphi(g,\rho,m)$ is refuted.
When $\varphi(g,\rho,m)$ is refuted, $\Delta$ changes his mind, and goes to the $(i+1)$-th challenge.
Clearly $\lim_s\Delta(g\res s)$ converges, and $\Phi_{\lim_s\Delta(g\res s)}(g)\in P$ holds.
\end{proof}

\begin{cor}\label{cor:5:l-tlref}
For every special $\Pi^0_1$ set $P\subseteq 2^\nn$ there exists a $\Pi^0_1$ set $Q\subseteq 2^\nn$ with $Q<^{1}_{\omega}P\equiv^{<\omega}_{\omega}Q$.
\end{cor}

\begin{proof}\upshape
By Theorem \ref{thm:hyperHig}, if $P\leq^1_\omega(P\htie P)\otimes 2^\nn\equiv^1_1 P\htie P$, then $P\leq^1_\omega 2^\nn$, i.e., $P$ contains a computable element.
As $P$ is special, we must have $P\not\leq^1_\omega P\htie P$.
As seen in Part I \cite[Section 4]{HK_PartI}, $P\leq^{<\omega}_\omega P\htie P$.
Therefore, for $Q=P\htie P$, we have $Q<^{1}_{\omega}P\equiv^{<\omega}_{\omega}Q$.
\end{proof}

\begin{cor}
Every nonzero $\dg{a}\in\mathcal{P}^1_\omega$ has the strong anticupping property.
\end{cor}

\begin{proof}\upshape
Fix $P\in\dg{a}$.
Let $\dg{b}$ be the $(1,\omega)$-degree of $P\htie P$.
Then, by Theorem \ref{thm:hyperHig}, for any $(1,\omega)$-degree $\dg{c}$, if $\dg{a}\leq\dg{b}\vee\dg{c}$, then $\dg{a}\leq\dg{c}$. 
\end{proof}

The primary motivation of the second author behind introducing the notions of learnability reduction was to attack an open problem on $\Pi^0_1$ subsets of $2^\nn$.
The problem (see Simpson \cite{Sim}) is whether the Muchnik degrees ($(\omega,1)$-degrees) of $\Pi^0_1$ classes are dense.
Cenzer-Hinman \cite{CH1} showed that the Medvedev degrees ($(1,1)$-degrees) of $\Pi^0_1$ classes are dense.
One can easily apply their priority construction to prove densities of $(1,<\omega)$-degrees and $(<\omega,1)$-degrees.
The reason is that the arithmetical complexity of $A^\alpha_\beta=\{(i,j)\in\nn^2:P_i\leq^\alpha_\beta P_j\}$ is $\Sigma^0_3$ for $(\alpha,\beta)\in\{(1,1),(1,<\omega),(<\omega,1)\}$, where $\{P_e\}_{e\in\nn}$ is an effective enumeration of all $\Pi^0_1$ subsets of $2^\nn$.
It enables us to use priority argument directly.
However, for other reductions $(\alpha,\beta)$, the complexity of $A^\alpha_\beta$ seems to be $\Pi^1_1$.
This observation hinders us from using priority arguments.
Hence it seems to be a hard task to prove densities of such $(\alpha,\beta)$-degrees.
Nevertheless, our disjunctive notions turn out to be useful to obtain some partial results.

\begin{theorem}[Weak Density]
For nonempty $\Pi^0_1$ sets $P,Q\subseteq 2^\nn$, if $P<^1_{\omega}Q$ and $P<^{<\omega}_{\omega}Q$ then there exists a $\Pi^0_1$ set $R\subseteq 2^\nn$ such that $P<^1_{\omega}R<^1_{\omega}Q$.
\end{theorem}

\begin{proof}\upshape
Assume $P<^1_{\omega}Q$ and $P<^{<\omega}_{\omega}Q$.
Let $R=(Q\htie Q)\lsup P$.
Then $P\leq^1_{\omega}R\leq^1_{\omega}Q$.
Moreover $Q\not\leq^1_{\omega}P$ implies $Q\not\leq^1_{\omega}R=(Q\htie Q)\lsup P$, by non-cupping property of $\htie$.
On the other hand, $R=(Q\htie Q)\lsup P\not\leq^1_{\omega}P$ since $Q\htie Q\equiv^{<\omega}_{\omega}Q\not\leq^{<\omega}_{\omega}P$.
Consequently, $P<^1_{\omega}R=(Q\htie Q)\lsup P<^1_{\omega}Q$.
\end{proof}

One can introduce a transfinite iteration $P^{\hjump{a}}$ of hyperconcatenation along $a\in\mathcal{O}$.
\index{$P^{\hjump{a}}$}%

\begin{prop}\label{prop:wfitera-hyperHig}
For any special $\Pi^0_1$ set $P\subseteq 2^\nn$, if $a,b\in\mathcal{O}$ and $a<_{\mathcal{O}}b$, then $P^{\hjump{b}}$ does not $(1,\omega)$-cup to $P^{\hjump{a}}$, i.e., for any set $R\subseteq\nn^\nn$, if $P^{\hjump{a}}\leq^1_\omega P^{\hjump{b}}\otimes R$ then $P^{\hjump{a}}\leq^1_\omega R$.
\end{prop}

\begin{proof}\upshape
The assumption $a<_\mathcal{O}b$ implies $2^a\leq_\mathcal{O}b$.
Therefore, we have $P^{\hjump{b}}\leq^1_\omega P^{\hjump{2^a}}$.
By Theorem \ref{thm:hyperHig}, $P^{\hjump{2^a}}$ does not $(1,\omega)$-cup to $P^{\hjump{a}}$.
Thus, $P^{\hjump{b}}$ does not $(1,\omega)$-cup to $P^{\hjump{a}}$.
\end{proof}

Fix again any notation ${\tt omega}\in\mathcal{O}$ such that $|\Phi_{\tt omega}(n)|_\mathcal{O}=n$ for each $n\in \nn$.
Recall from Part I that a learner $\Psi$ is {\em eventually-Popperian} if, for every $f\in\nn^\nn$, $\Phi_{\lim_s\Psi(f\res s)}(f)$ is total whenever $\lim_s\Psi(f\res s)$ converges.
\index{learner!eventually-Popperian}%

\begin{prop}\label{prop:breduction}
Let $P$ be a special $\Pi^0_1$ subset of $2^\nn$.
For any set $R\subseteq\nn^\nn$, if $P^{\hjump{a}}\leq^{<\omega}_{\omega}P^{\hjump{{\tt omega}}}\otimes R$ by a team of eventually-Popperian learners, then $P\leq^{<\omega}_\omega R$.
\end{prop}

\begin{proof}\upshape
If $P\leq^{<\omega}_{\omega}P^{\hjump{\tt omega}}\otimes R$ via a team of eventually-Popperian learners, then this reduction is also witnessed by a team of $n$ eventually-Popperian learners, for some $n\in\nn$.
In particular, by modifying this reduction, we can easily construct a team of $n$ eventually-Popperian learners witnessing $P\leq^{<\omega}_{\omega}P^{\hjump{n+1}}\otimes R$.
In this case, it is not hard to show $P^{\hjump{n}}\leq^1_1P^{\hjump{n+1}}\otimes R$.
By Theorem \ref{thm:hyperHig}, $P^{\hjump{n}}\leq^1_\omega R$.
Hence, $P\leq^{<\omega}_{\omega}R$ is witnessed by a team of $n$ learners, as seen in Part I.
\end{proof}

\begin{cor}
For every $a\in\mathcal{O}$ there exists a computable function $g$ such that, for any $\Pi^0_1$ index $e$, if $P_e$ is special then the following properties hold.
\begin{enumerate}
\item $P_{g(e,b)}<^1_\omega P_{g(e,c)}$ holds for every $c<_\mathcal{O}b<_\mathcal{O}a$, indeed, $P_{g(e,b)}$ does not $(1,\omega)$-cup to $P_{g(e,c)}$.
\item $P_{g(e,b)}\equiv^\omega_1P_{g(e,c)}$ for every $b,c<_\mathcal{O}a$.
\end{enumerate}
\end{cor}

\begin{proof}\upshape
Let $g(e,b)$ be an index of $P_e^{\hjump{b}}$.
Then the desired conditions follow from Proposition \ref{prop:wfitera-hyperHig}.
\end{proof}

\begin{cor}
For any nonzero $(\omega,1)$-degree $\dg{a}\in\mathcal{P}^\omega_1$, there is a $(1,\omega)$-noncupping computable sequence of $(1,\omega)$-degrees inside $\dg{a}$ of arbitrary length $\alpha<\omega_1^{CK}$.
\qed
\end{cor}

\subsection{Infinitary Disjunctions along Infinite Complete Graphs}

The following is the last LEVEL 4 separation result, which reveals a difference between $[\mathfrak{C}_T]^{<\omega}_\omega$ and $[\mathfrak{C}_T]^\omega_1$.

\begin{theorem}\label{thm:NRMP:contig:tl-w}
For every special $\Pi^0_1$ set $P,Q\subseteq 2^\nn$ there exists a $\Pi^0_1$ set $\widehat{P}\subseteq 2^\nn$ such that $Q\not\leq^{<\omega}_{\omega}\widehat{P}$ and $\widehat{P}$ is $(\omega,1)$-equivalent to $P$.
\end{theorem}

\begin{proof}\upshape
We construct a $\Pi^0_1$ set $P\subseteq 2^\nn$ by priority argument with infinitely many requirements $\{\mathcal{P}_e,\mathcal{G}_e\}_{e\in\nn}$.
Each preservation ($\mathcal{P}_e$-)strategy will injure our coding ($\mathcal{G}$-)strategy of $P$ into $\widehat{P}$ infinitely often.
In other words, for each $\mathcal{P}_e$-requirement, $\widehat{P}$ contains an element $g_e^f$ which is a counterpart of each $f\in P$, but each $g_e^f$ has infinitely many noises.
Indeed, to satisfy the $\mathcal{P}$-requirements, we need to ensure that there is no uniformly team-learnable way to extract the information of $f\in P$ from its code $g_e^f\in\widehat{P}$.
Nevertheless, the global ($\mathcal{G}$-)requirement must guarantee that $f\in P$ is computable in $g_e^f\in\widehat{P}$ via a non-uniform way.
Let $\{\Psi^e_i\}_{i<{b(e)}}$ be the $e$-th team of learners, where $b=b(e)$ is the number of members of the $e$-th team.
\begin{req}
It suffices to construct a $\Pi^0_1$ set $\widehat{P}\subseteq 2^\nn$ satisfying the following requirements.
\begin{align*}
\mathcal{P}_e&\mbox{ : }(\exists g_e\in\widehat{P})(\forall i<b)\;\left(\lim_s\Psi^e_i(g_e\res s)\downarrow\;\rightarrow\;\Phi_{\lim_s\Psi^e_i(g_e\res s)}(g_e)\not\in Q\right).\\
\mathcal{G}_e&\mbox{ : }(\forall f\in P)\;f\leq_Tg_e^f.
\end{align*}
Here, the desired $\Pi^0_1$ set $\widehat{P}\subseteq 2^\nn$ will be of form $P\cup\{g_e^f:e\in\nn\;\&\;f\in P\}$.
\end{req}

\begin{construction}
We will construct a computable sequence of computable trees $\{T_s\}_{s\in\nn}$, and a computable sequence of natural numbers $\{h_s\}_{s\in\nn}$.
The desired set $\widehat{P}$ is defined as $[\bigcup_sT_s]$, and $h_s$ is called {\em active height at stage $s$}.
We will ensure that the tree $T_s$ consists of strings of length $\leq h_s$.
The strategy for the $\mathcal{P}_e$-requirement acts on some string extending the $e$-th leaf $\rho_e$ of $T_{\sf CPA}$.

We will inductively define a string $\gamma_e(\alpha,s)\in T_s$ extending $\rho_e$ for each $s\in\nn$ and $\alpha\in T_P$ of height $\leq s$.
The map $\alpha\mapsto\lim_s\gamma_e(\alpha,s)$ restricted to $T_P^{ext}$ will provide a tree-isomorphism between $T_P^{ext}$ and $(\bigcup_sT_s)^{ext}$, i.e., $\widehat{P}\cap[\rho_e]$ will be constructed as the set of all infinite paths of the tree generated by $\{\lim_s\gamma_e(\alpha,s):\alpha\in T_P\}$.
In other words, $g_e^f$ is defined by $\bigcup_{\alpha\subset f}\lim_s\gamma_e(\alpha,s)$, and each string $\gamma_e(\alpha,s)$ is an approximation of $g_e\in\widehat{P}$ witnessing to satisfy the $\mathcal{P}_e$ requirements.

We will also define a finite set $M_e(\alpha,s)\subseteq b$ for each $s\in\nn$ and $\alpha\in T_P$ of height $\leq s$.
Intuitively, $M_e(\alpha,s)$ contains any index of the learner who have been already changed his mind $|\alpha|$ times along any string extending $\alpha$ of length $s$, and the string $\gamma_e(\alpha,s)$ also plays the role of an active node for learners in $M_e(\alpha,s)$.
To satisfy the $\mathcal{P}_e$-requirement, each learner in $M_e(\alpha,s)$ can act on $\gamma_e(\alpha,s)$ at stage $s+1$, and then he extends $\gamma_e(\alpha,s)$ to some new string $\gamma_e(\alpha,s+1)$ of length $h_s$, and {\em injures} all constructions of $\gamma_e(\beta,s+1)$ for $\beta\supsetneq\alpha$.
We assume that, for any $\alpha\in T_P$ of length $s$, $\{M_e(\beta,s)\}_{\beta\subseteq\alpha}$ is a partition of $\{i\in\nn:i<b\}$.

\medskip

\noindent
{\bf Stage $0$.}
At first, put $T_s=\{\lrangle{}\}$, $h_s=0$, $M_e(\lrangle{},0)=\{i\in\nn:i<b\}$, and $\gamma_e(\lrangle{},0)=\rho_e$.

\medskip

\noindent
{\bf Stage $s+1$.}
At the beginning of each stage $s+1$, assume that $T_s$ and $h_s$ are given, and that $M_e(\beta,s)$ and $\gamma_e(\beta,s)$ have been already defined for each $s\in\nn$ and $\beta\in T_P$ of height $\leq s$.
For each $i,e\in\nn$ and each $\tau\in 2^\nn$, {\em the length-of-agreement function} $l^i_e(\tau)$ is the maximal $l\in\nn$ such that $\Phi_{\Psi^e_i(\tau)}(\tau;x)\downarrow$ for each $x<l$, and $\Phi_{\Psi^e_i(\tau)}(\tau)\in T_Q$.

Fix a string $\alpha\in T_P$ of length $s$, and then each $i$ belongs to some $M_e(\beta,s)$ for $\beta\subseteq\alpha$.
In this case, the learner $\Psi^e_i$ can act on $\gamma_e(\beta,s)$.
Then, we say that {\em the learner $\Psi^e_i$ requires attention along $\alpha$ at stage $s+1$} if there exists $\tau\in T_s$ of length $h_s$ extending $\gamma_e(\beta,s)$ such that either of the following conditions are satisfied.
\begin{enumerate}
\item $\Psi^e_i$ {\em changes on $(\gamma_e(\beta,s),\tau]$}, i.e., there is a string $\sigma$ such that $\gamma_e(\beta,s)\subsetneq\sigma\subseteq\tau$ and $\Psi^e_i(\sigma^-)\not=\Psi^e_i(\sigma)$.
\item or, $l^i_e(\tau)>\max\{l^i_e(\sigma):\sigma\subseteq\gamma_e(\beta,s)\}$.
\end{enumerate}

Let $R_s$ be the set of all $\alpha\in T_P$ of length $s$ such that some learner requires attention along $\alpha$ at stage $s+1$.
For $\alpha\in R_s$, let $m(\alpha)$ be the least $m$ such that there is a string $\beta\subseteq\alpha$ of length $m$ and an index $i\in M_e(\beta,s)$ such that $\Psi^e_i$ requires attention along $\alpha$ at stage $s+1$.
That is to say, some learner $\Psi^e_i$ who has already changed his mind $m(\alpha)$ times requires attention.

\begin{claim}
For any $\alpha,\beta\in R_s$, we have that $\alpha\res m(\alpha)=\beta\res m(\beta)$ holds or $\alpha\res m(\alpha)$ is incomparable with $\beta\res m(\beta)$.
\end{claim}

Put $R_s^*=\{\alpha\res m(\alpha):\alpha\in R_s\}$.
Then, for $\beta\in R_s^*$, let $i(\beta)$ be the least $i\in M(m(\alpha),s)$ such that $\Psi^e_i$ requires attention along some $\alpha\supseteq\beta$ of length $s$ at stage $s+1$.
For $\beta\in R_s^*$, we say that {\em $\Psi^e_{i(\beta)}$ acts at stage $s+1$}.
Moreover, for $\beta\in R_s^*$, let $\tau(\beta)$ be the lexicographically least string $\tau\in T_s$ of length $h_s$ extending $\gamma_e(\beta,s)$ such that $\tau$ witnesses that the learner $\Psi^e_{i(\beta)}$ requires attention along some $\alpha\supseteq\beta$ of length $s$ at stage $s+1$.
Then $R_s^{**}\subseteq R_s^*$ is defined as the set of all $\beta\in R_s^*$ such that $\Psi^e_{i(\beta)}$ changes on $(\gamma_e(\beta,s),\tau(\beta)]$.

For each $\beta\in R_s^{**}$, put $M_e(\beta,s+1)=M_e(\beta,s)\setminus\{i(\beta)\}$, and put $M_e(\beta\fr i,s+1)=M_e(\beta,s)\cup\{i(\beta)\}$ for $\beta\fr i\in T_P$.
For any $\beta\not\in R_s^{**}$, put $M_e(\beta,s+1)=M_e(\beta,s)$.
For each $\beta\in R_s^*$, if $\beta\fr\sigma\in T_P$ is length $\leq s$ for some $\sigma\in 2^{<\nn}$, then put $\gamma_e(\beta\fr\sigma,s+1)=\tau(\beta)\fr\sigma$.
If $\alpha\in T_P$ of length $\leq s$ has no substring $\beta\in R_s^*$, then put $\gamma_e(\alpha,s+1)=\gamma_e(\alpha,s)$.
For each $\alpha\in T_P$ of length $s$, if $|\gamma_e(\alpha,s+1)|<h_s$ then pick the lexicographically least node $\gamma^*_e(\alpha,s+1)\in T_s$ such that $|\gamma^*_e(\alpha,s+1)|=h_s$ and $\gamma^*_e(\alpha,s+1)\supseteq\gamma_e(\alpha,s+1)$.
Otherwise put $\gamma^*_e(\alpha,s+1)=\gamma_e(\alpha,s+1)$.
Then, for each $\alpha\fr i\in T_P$ of length $s$, put $\gamma_e(\alpha\fr i,s+1)=\gamma^*_e(\alpha,s+1)\fr i$.
Put $h_{s+1}=\max\{|\gamma_e(\alpha,s+1)|:\alpha\in T_P\;\&\;|\alpha|=s+1\}$.
Then we define the approximation of $\widehat{P}$ at stage $s+1$ as follows.
\[T_{s+1}=T_s\cup\{\sigma\subseteq\gamma_e(\alpha,s+1)\fr 0^{h_{s+1}-|\gamma_e(\alpha,s+1)|}:\alpha\in T_P\;\&\;|\alpha|=s+1\;\&\;e\in\nn\}.\]

Finally, we set $\widehat{P}=[\bigcup_{s\in\nn}T_s]$.
Clearly, $\widehat{P}$ is a nonempty $\Pi^0_1$ subset of $2^\nn$.
\end{construction}

\begin{lemma}
$\lim_s\gamma_e(\alpha,s)$ converges for any $e\in\nn$ and $\alpha\in T_P$.
\end{lemma}

\begin{proof}\upshape
Note that $\gamma_e(\alpha,s)$ is incomparable with $\gamma_e(\beta,s)$ whenever $\alpha$ is incomparable with $\beta$.
Therefore, $\gamma_e(\alpha,s)$ changes only when some learner in $M_e(\beta,s)$ acts for some $\beta\subseteq\alpha$.
Assume that $\gamma_e(\alpha,s)$ changes infinitely often.
Then there is $\beta\subseteq\alpha$, $t\in\nn$ and $i\in M_e(\beta,t)$ such that $i\in M_e(\beta,s)$ for any $s\geq t$, and $\Psi^e_{i(\beta)}$ acts infinitely often.
However, by our construction, $g_e^\alpha=\lim_s\gamma_e(\alpha,s)$ is computable.
Additionally, since $i\in M_e(\beta,s)$ for any $s\geq t$, $\lim_n\Psi^e_{i(\beta)}(g_e^\alpha\res n)$ exists, and $\Phi_{\lim_n\Psi^e_{i(\beta)}(g_e^\alpha\res n)}(g_e^\alpha)\in Q$.
This contradicts our assumption that $Q$ is special.
\end{proof}

For $f\in P$, put $g_e^f=\bigcup_{\alpha\subset f}\lim_s\gamma_e(\alpha,s)$.
By this lemma, such $g_e^f$ exists, and we observe that $\widehat{P}$ can be represented as $\widehat{P}=P\cup\{g_e^f:e\in\nn\;\&\;f\in P\}$.
For each $e\in\nn$ and $\alpha\in T_P$, we pick $t(e,\alpha)\in\nn$ such that $\gamma_e(\alpha,s)=\gamma_e(\alpha,t)$ for any $s,t\geq t(e,\alpha)$.

\begin{lemma}\label{lem:3b:P}
The $\mathcal{P}$-requirements are satisfied.
\end{lemma}

\begin{proof}\upshape
Assume that $P\leq^{<\omega}_{\omega}\widehat{P}$ via the $e$-th team $\{\Psi_i\}_{i<b}$ of learners.
Then, for any $f\in P$, there is $i<b$ such that $\lim_n\Psi_i(g_e^f\res n)$ exists and $\Phi_{\lim_n\Psi_i(g_e^f\res n)}(g_e^f)\in Q$.
Since $\lim_n\Psi_i(g_e^f\res n)$ exists, there exists $\alpha\subset f$ such that $i\in M_e(\alpha,t(e,\alpha))$.
However, by the previous claim, no learner in $\bigcup_{\beta\subseteq\alpha}M_e(\beta,t(e,\alpha))$ requires attention after stage $t(e,\alpha)$.
This implies $\lim_nl_e^i(g_e^f\res n)<\infty$.
In other words, $\Phi_{\lim_n\Psi_i(g_e^f\res n)}(g_e^f)\not\in Q$.
This contradicts our assumption.
\end{proof}

\begin{lemma}\label{lem:3b:G}
The $\mathcal{G}$-requirements are satisfied.
\end{lemma}

\begin{proof}\upshape
It suffices to show that $f\leq_Tg_e^f$ for any $e\in\nn$ and $f\in P$.
Assume that $\{\Psi_i\}_{i<b}$ is the $e$-th team of learners.
Let $H_e^f$ denote the set of all $i<b$ such that $\lim_n\Psi_i(g_e^f\res n)$ converges.
By our construction and the first claim, if $i\in H_e^f$ then $i\in M_e(\alpha_i,t(e,\alpha_i))$ for some $\alpha_i\subset f$.
If $i\not\in H_e^f$ then for any $\alpha\subset f$ there exists $s$ such that $i\in M_e(\alpha,s)$.
Set $l=\max_{i\in H_e^f}|\alpha_i|$, and $u=\max_{i\in H_e^f}t(e,\alpha_i)$.
For $n>l$, to compute $f(n)$, we wait for stage $v(n)>u$ such that, for every $i\not\in H_e^f$, $i\in M_e(f\res m,v(n))$ for some $m\geq n+1$.
By our construction, it is easy to see that we can extract $f(n)$ from $\gamma_e(f\res n+1,v(n))$, by a uniformly computable procedure in $n$.
\end{proof}

Thus, we have $Q\not\leq^{<\omega}_\omega\widehat{P}$ by Lemma \ref{lem:3b:P}, and $P\subseteq\widehat{P}\subseteq\widehat{\rm Deg}(P)$ by Lemma \ref{lem:3b:G}.
Thus, $\widehat{P}$ is a $\Pi^0_1$ set satisfying $Q\not\leq^{<\omega}_\omega\widehat{P}\equiv^\omega_1P$.
This concludes the proof.
\end{proof}

\begin{cor}
For any nonempty $\Pi^0_1$ sets $P,Q\subseteq 2^\nn$, if $Q\leq^{<\omega}_{\omega}\widehat{\rm Deg}(P)$ then $Q$ contains a computable element.
\end{cor}

\begin{proof}\upshape
Assume that $Q\leq^{<\omega}_{\omega}\widehat{\rm Deg}(P)$ is satisfied.
Suppose that $Q$ has no computable element.
Then, for $P,Q\subseteq 2^\nn$, we obtain $Q\not\leq^{<\omega}_\omega\widehat{P}\equiv^\omega_1 P$ by Theorem \ref{thm:NRMP:contig:tl-w}.
Note that the condition $\widehat{P}\equiv^\omega_1P$ implies $\widehat{P}\subseteq\widehat{\rm Deg}(P)$.
Then, $Q\leq^{<\omega}_{\omega}\widehat{\rm Deg}(P)\leq^1_1\widehat{P}$.
It involves a contradiction.
\end{proof}

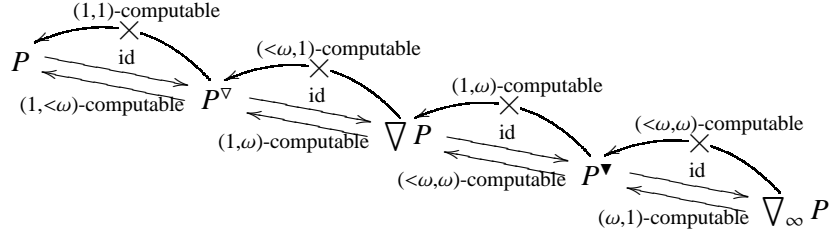
\begin{figure}[t]\centering
\begin{center}
\[
	\begin{xy}
	(0,20) *{P\ \ }="A", (25,15) *{\ \ P^\tie\ \ }="B", (50,10) *{\ \ \btie P\ \ }="C", (75,5) *{\ \ P^\htie\ \ }="D", (100,0) *{\ \ \bcls P}="E",
	\ar @<1mm> "A";"B"^{{\rm id}}
	\ar @<1mm> "B";"A"^{(1,<\omega)\text{-computable}}
	\ar @/_1pc/ (24,17.5);(1,22.5)|{\varprod}_{(1,1)\text{-computable}}
	\ar @<1mm> "B";"C"^{{\rm id}}
	\ar @<1mm> "C";"B"^{(1,\omega)\text{-computable}}
	\ar @/_1pc/ (49,12.5);(26,17.5)|{\varprod}_{(<\omega,1)\text{-computable}}
	\ar @<1mm> "C";"D"^{{\rm id}}
	\ar @<1mm> "D";"C"^{(<\omega,\omega)\text{-computable}}
	\ar @/_1pc/ (74,7.5);(51,12.5)|{\varprod}_{(1,\omega)\text{-computable}}
	\ar @<1mm> "D";"E"^{{\rm id}}
	\ar @<1mm> "E";"D"^{(\omega,1)\text{-computable}}
	\ar @/_1pc/ (99,2.5);(76,7.5)|{\varprod}_{(<\omega,\omega)\text{-computable}}
	\end{xy}
\]
\end{center}
 \vspace{-0.5em}
\caption{The dynamic proof model for a special $\Pi^0_1$ set $P\subseteq 2^\nn$.}
  \label{fig:arrow2}
\end{figure}

%% file: NRMP_fullproof5_II.tex
\section{Applications and Questions}

\subsection{Diagonally Noncomputable Functions}

A total function $f:\nn\to\nn$ is a {\em $k$-valued diagonally noncomputable function} if $f(n)<k$ for any $n\in\nn$ and $f(e)\not=\Phi_e(e)$ whenever $\Phi_e(e)$ converges.
\index{diagonally noncomputable!$k$-valued}%
Let ${\sf DNR}_k$ denote the set of all $k$-valued diagonally noncomputable functions.
\index{${\sf DNR}_k$}%
Jockusch \cite{Joc} showed that every ${\sf DNR}_k$ function computes a ${\sf DNR}_2$ function.
However, he also noted that there is no uniformly computable algorithm finding a ${\sf DNR}_2$ function from any ${\sf DNR}_k$ function.

\begin{theorem}[Jockusch \cite{Joc}]~
\begin{enumerate}
\item ${\sf DNR}_{k}>^1_1{\sf DNR}_{k+1}$ for any $k\in\nn$.
\item ${\sf DNR}_2\equiv^{\omega}_{1}{\sf DNR}_k$ for any $k\in\nn$.
\end{enumerate}
\end{theorem}

\begin{prop}\label{prop:5:DNR1}~
\begin{enumerate}
\item If a $(1,\omega)$-degree $\dg{d}^1_{\omega}$ of subsets of $\nn^\nn$ contains a $(1,1)$-degree $\dg{h}^1_1$ of homogeneous sets, then $\dg{h}^1_1$ is the greatest $(1,1)$-degree inside $\dg{d}^1_{\omega}$.
\item If an $(<\omega,1)$-degree $\dg{d}^{<\omega}_1$ of $\Pi^0_1$ subsets of $2^\nn$ contains a $(1,<\omega)$-degree $\dg{h}^1_{<\omega}$ of homogeneous $\Pi^0_1$ sets, then $\dg{h}^1_{<\omega}$ is the least $(1,<\omega)$-degree inside $\dg{d}^{<\omega}_1$.
\item Every $(<\omega,1)$-degree of $\Pi^0_1$ subsets of $2^\nn$ contains at most one $(1,1)$-degree of homogeneous $\Pi^0_1$ sets.
\end{enumerate}
\end{prop}

\begin{proof}\upshape
For the item $1$, we can see that, for any $P\subseteq\nn^\nn$ and any closed set $Q\subseteq\nn^\nn$, if $P\leq^1_{\omega} Q$ then there is a node $\sigma$ such that $Q\cap[\sigma]$ is nonempty and $P\leq^1_1 Q\cap[\sigma]$.
That is, $\sigma$ is a locking sequence.
If $Q$ is homogeneous, then $P\leq^1_1Q\equiv^1_1Q\cap[\sigma]$.
The item $2$ follows from Theorem \ref{thm:contig:b-bl}.
By combining the item $1$ and $2$, we see that every $(<\omega,1)$-degree of $\Pi^0_1$ subsets of $2^\nn$ contains at most one $(1,<\omega)$-degree of homogeneous $\Pi^0_1$ sets which contains at most $(1,1)$-degree of homogeneous $\Pi^0_1$ sets.
\end{proof}

\begin{cor}\label{cor:5:DNR2}
${\sf DNR}_3<^{<\omega}_1{\sf DNR}_2$, and ${\sf DNR}_3<^1_{\omega}{\sf DNR}_2$.
\end{cor}

\begin{proof}\upshape
By Jockusch \cite{Joc}, we have ${\sf DNR}_3<^1_1{\sf DNR}_2$.
Thus, Proposition \ref{prop:5:DNR1} implies the desired condition.
\end{proof}

By analyzing Jockusch's proof \cite{Joc} of the Muchnik equivalence of ${\sf DNR}_2$ and ${\sf DNR}_k$ for any $k\geq 2$, we can directly establish the $(<\omega,\omega)$-equivalence of ${\sf DNR}_2$ and ${\sf DNR}_k$ for any $k\geq 2$.
However, one may find that Jockusch's proof \cite{Joc} is essentially based on the $\Sigma^0_2$ law of excluded middle.
Therefore, the fine analysis of this proof structure establishes the following theorem.

\begin{theorem}\label{thm:5:DNR3}
${\sf DNR}_k\htie{\sf DNR}_k<^1_1{\sf DNR}_{k^2}$ for any $k$.
\end{theorem}

\begin{proof}\upshape
As Jockusch \cite{Joc}, fix a computable function $z:\nn^2\to\nn$ such that $\Phi_{z(v,u)}(z(v,u))$
$=\lrangle{\Phi_v(v),\Phi_u(u)}$ for any $v,u\in\nn$.
Note that every $g\in(k^2)^\nn$ determines two functions $g_0\in k^\nn$ and $g_1\in k^\nn$ such that $g(n)=\lrangle{g_0(n),g_1(n)}$ for any $n\in\nn$.
We define a uniform sequence $\{\Gamma_v\}_{v\in\nn},\Delta$ of computable functions as $\Gamma_v(g;u)=g_1(z(v,u))$, and $\Delta(g;v)=g_0(z(v,u_v))$, where $u_v=\min\{u\in\nn:g_1(z(v,u))=\Phi_u(u)\downarrow\}$.
Fix $g\in{\sf DNR}_{k^2}$.
Since $\lrangle{g_0(z(v,u)),g_1(z(v,u))}=g(z(v,u))\not=\lrangle{\Phi_v(v),\Phi_u(u)}$, either $g_0(z(v,u))\not=\Phi_v(v)$ or $g_1(z(v,u))\not=\Phi_u(u)$ holds for any $v,u\in\nn$.
We consider the following $\Sigma^0_2$ sentence:
\[(\exists v)(\forall u)\;(\Phi_u(u)\downarrow\;\rightarrow\;g_1(z(v,u))\not=\Phi_u(u)).\]
Let $\theta(g,v)$ denote the $\Pi^0_1$ sentence $(\forall u)\;(\Phi_u(u)\downarrow\;\rightarrow\;g_1(z(v,u))\not=\Phi_u(u))$.
If $\theta(g,v)$ holds, then $\Gamma_v(g;u)=g_1(z(v,u))\not=\Phi_u(u)$ for any $u\in\nn$.
Hence, $\Gamma_v(g)\in{\sf DNR}_k$.
If $\neg\theta(g,v)$ holds, then $u_v$ is defined.
Therefore, $\Delta(g;v)=g_0(z(v,u_v))\downarrow\not=\Phi_v(v)$, since $g_1(z(v,u_v))=\Phi_{u_v}(u_v))\downarrow$.
Thus, $\Delta(g;v)$ is extendible to a function in ${\sf DNR}_k$.
This procedure shows that there is a function $\Gamma:{\sf DNR}_{k^2}\to{\sf DNR}_k$ that is computable strictly along $\Pi^0_1$ sets $\{S_v\}_{v\in\nn}$ via $\Delta$ and $\{\Gamma_v\}_{v\in\nn}$, where $S_v=\{g:\theta(g,v)\}$.
Consequently, ${\sf DNR}_k\htie{\sf DNR}_k\leq^1_1{\sf DNR}_{k^2}$ as seen in Part I \cite{HK_PartI}.

To see ${\sf DNR}_k\htie{\sf DNR}_k\not\geq^1_1{\sf DNR}_{k^2}$, we note that ${\sf DNR}_k\htie{\sf DNR}_k$ is not tree-immune.
By Cenzer-Kihara-Weber-Wu \cite{CKWW}, ${\sf DNR}_k\htie{\sf DNR}_k$ does not cup to the generalized separating class ${\sf DNR}_{k^2}$.
\end{proof}

\begin{cor}\label{cor:5:DNR4}
${\sf DNR}_k\equiv^{<\omega}_{\omega}{\sf DNR}_2$ for any $k\geq 2$. Indeed, for any $k\in\nn$, the direction ${\sf DNR}_k\leq^{<\omega}_{\omega}{\sf DNR}_{k^2}$ is witnessed by a team of a confident learner and a eventually-Popperian learner.
In particular, ${\sf DNR}_k\equiv^{\omega}_{1}{\sf DNR}_2$ for any $k\geq 2$.
\end{cor}

\begin{proof}\upshape
As seen in Part I \cite{HK_PartI}, $P\leq^{<\omega}_\omega P\htie P$ is witnessed by a team of a confident learner and a eventually-Popperian learner.
Thus, Theorem \ref{thm:5:DNR3} implies the desired condition.
\end{proof}

\begin{cor}
There is an $(<\omega,\omega)$-degree which contains infinitely many $(1,1)$-degrees of homogeneous $\Pi^0_1$ sets.
\end{cor}

\begin{proof}\upshape
By Corollary \ref{cor:5:DNR4}, the $(<\omega,\omega)$-degree of ${\sf DNR}_2$ contains ${\sf DNR}_k$ for any $k\in\nn$, while ${\sf DNR}_k\not\equiv^1_1{\sf DNR}_l$ for $k\not=l$.
\end{proof}

\subsection{Simpson's Embedding Lemma}

For a pointclass $\Gamma$ in a space $X$, we say that {\em ${\rm SEL}(\Gamma,X)$ holds for $(\alpha,\beta)$-degrees} holds for $(\alpha,\beta)$-degrees if, for every $\Gamma$ set $S\subseteq X$ and for every nonempty $\Pi^0_1$ set $Q\subseteq 2^\nn$, there exists a $\Pi^0_1$ set $P\subseteq 2^\nn$ such that $P\equiv^\alpha_\beta S\cup Q$.
\index{${\rm SEL}(\Gamma,X)$}%
Jockusch-Soare \cite{JS2} indicates that ${\rm SEL}(\Pi^0_2,\nn^\nn)$ holds for $(\omega,1)$-degrees, and points out that ${\rm SEL}(\Pi^0_3,2^\nn)$ does not hold for $(\omega,1)$-degrees, since the set of all noncomputable elements in $2^\nn$ is $\Pi^0_3$.
Simpson's Embedding Lemma \cite{Sim2} determines the limit of ${\rm SEL}(\Gamma,X)$ for $(\omega,1)$-degrees.

\begin{theorem}[Simpson \cite{Sim2}]
${\rm SEL}(\Sigma^0_3,\nn^\nn)$ holds for $(\omega,1)$-degrees.\qed
\end{theorem}

\begin{theorem}[Simpson's Embedding Lemma for other degree structures]~
\begin{enumerate}
\item ${\rm SEL}(\Sigma^0_2,2^\nn)$ does not hold for $(<\omega,1)$-degrees.
\item ${\rm SEL}(\Sigma^0_2,2^\nn)$ holds for $(1,\omega)$-degrees.
\item ${\rm SEL}(\Pi^0_2,2^\nn)$ does not hold for $(1,\omega)$-degrees.
\item ${\rm SEL}(\Pi^0_2,\nn^\nn)$ holds for $(<\omega,\omega)$-degrees.
\item ${\rm SEL}(\Sigma^0_3,2^\nn)$ does not hold for $(<\omega,\omega)$-degrees.
\end{enumerate}
\end{theorem}

\begin{proof}\upshape
(1)
For any $\Pi^0_1$ set $P\subseteq 2^\nn$, we note that $\btie P\subseteq 2^\nn$ is $\Sigma^0_2$.
By Theorem \ref{thm:contig:b-l}, there is no $\Pi^0_1$ set $2^\nn$ which is $(<\omega,1)$-below $\btie P$.
In particular, there is no $\Pi^0_1$ set $2^\nn$ which is $(<\omega,1)$-equivalent to $P\cup\btie P=\btie P$.

(2)
For a given $\Sigma^0_2$ set $S\subseteq 2^\nn$, there is a computable increasing sequence $\{P_i\}_{i\in\nn}$ of $\Pi^0_1$ classes such that $S=\bigcup_{i\in\nn}P_i$.
We need to show $\bigcup_{i\in\nn}P_i\equiv^1_{\omega}\binf_{i\in\nn}P_i$, since $\binf_{i\in\nn}P_i$ is $(1,<\omega)$-equivalent to the $\Pi^0_1$ class $\cmeet_iP_i$.
Then, it is easy to see $\bigcup_iP_i\leq^1_1\binf_iP_i$.
For given $f\in\bigcup_iP_i$, from each initial segment $f\res n$, a learner $\Psi$ guesses an index of a computable function $\Phi_{\Psi(f\res n)}(g)=i\fr g$ for the least number $i$ such that $f\res n\in T_{P_i}$ but $f\res n\not\in T_{P_{i-1}}$.
For any $f\in\bigcup_iP_i$, for the least $i$ such that $f\in P_i\setminus P_{i-1}$, $\lim_n\Psi(f\res n)$ converges to an index of $\Phi_{\lim_n\Psi(f\res n)}(g)=i\fr g$.
Thus, $\Phi_{\lim_n\Psi(f\res n)}(g)\in i\fr P_i$.
Consequently, $S=\bigcup_iP_i\leq^1_{\omega}\cmeet_iP_i$.

(3)
Fix any special $\Pi^0_1$ set $P\subseteq 2^\nn$.
By Jockusch-Soare \cite{JS2}, there is a noncomputable $\Sigma^0_1$ set $A\subseteq\nn$ such that $P$ has no $A$-computable element.
Then $\{A\}\subseteq 2^\nn$ is a $\Pi^0_2$ singleton, since $A$ is $\Sigma^0_1$.
Therefore, $P\oplus\{A\}$ is $\Pi^0_2$.
It suffices to show that there is no $\Pi^0_1$ set $Q\subseteq 2^\nn$ such that $Q\equiv^1_\omega P\oplus\{A\}$.
Assume that $Q\equiv^1_\omega P\oplus\{A\}$ is satisfied for some $\Pi^0_1$ set $Q\subseteq 2^\nn$.
Then $Q$ must have an $A$-computable element $\alpha\in Q$.
Fix a learner $\Psi$ witnessing $P\oplus\{A\}\leq^1_\omega Q$.
Then, we have $\Phi_{\lim_n\Psi(\alpha\res n)}(\alpha)=1\fr A$, since $P$ has no element computable in $\alpha\leq_TA$.
We wait for $s\in\nn$ such that $\Psi(\alpha\res t)=\Psi(\alpha\res s)$ for any $t\geq s$.
Then, fix $u\geq s$ with $\Phi_{\Psi(\alpha\res u)}(\alpha\res u;0)\downarrow=1$.
Consider the $\Pi^0_1$ set $Q^*=\{f\in Q\cap[\alpha\res u]:(\forall v\geq u)\;\Psi(f\res v)=\Psi(f\res u)\}$.
Then, for any $f\in Q^*$, $\Phi_{\lim_s\Psi(f\res s)}(f)=\Phi_{\Psi(\alpha\res u)}(f)$ must extends $\lrangle{1}$.
Thus, $\{1\fr A\}\leq^1_1Q^*$ via the computable function $\Phi_{\Psi(\alpha\res u)}$.
Since $Q^*$ is special $\Pi^0_1$ subset of $2^\nn$, this implies the computability of $1\fr A$ which contradicts our choice of $A$.

(4)
Fix a $\Pi^0_2$ set $S\subseteq\nn^\nn$.
As Simpson's proof, there is a $\Pi^0_1$ set $\widehat{S}\subseteq\nn^\nn$ such that $S\equiv^1_1\widehat{S}$.
We can find a $\Pi^0_1$ set $\widehat{P}\subseteq\widehat{S}\htie Q$ such that $\widehat{P}\leq^1_1 S\cup Q$, and $\widehat{P}$ is computably homeomorphic to a $\Pi^0_1$ set $P\subseteq 2^\nn$.
Since $S\cup Q\leq^{<\omega}_{\omega}\widehat{S}\htie Q$, we have $S\cup Q\equiv^{<\omega}_{\omega}P$.

(5)
For every $\Pi^0_1$ set $P\subseteq 2^\nn$, the Turing upward closure $\widehat{\rm Deg}(P)=\{g\in 2^\nn:(\exists f\in P)\;f\leq_Tg\}$ of $P$ is $\Sigma^0_3$, and $\widehat{\rm Deg}(P)$ has the least $(<\omega,\omega)$-degree inside $\deg^{\omega}_{1}(P)$.
By Theorem \ref{thm:NRMP:contig:tl-w}, there is no $\Pi^0_1$ subset of $2^\nn$ which is $(<\omega,\omega)$-equivalent to $\widehat{\rm Deg}(P)$.
\end{proof}

\subsection{Weihrauch Degrees}

The notion is piecewise computability could be interpreted as the computability relative to the principle of excluded middle in a certain sense.
Indeed, in Part I \cite[Section 6]{HK_PartI}, we have characterized the notions of piecewise computability as the computability relative to nonconstructive principles in the context of Weihrauch degrees.
Thus, one can rephrase our separation results in the context of Weihrauch degrees as follows.

\begin{theorem}
The symbols $P$, $Q$, and $R$ range over all special $\Pi^0_1$ subset of $2^\nn$, and $X$ ranges over all subsets of $\nn^\nn$.
\begin{enumerate}
\item There are $P$ and $Q\leq^1_1 P$ such that $P\leq_{\Sigma^0_1\text{\sf -LLPO}}Q$ but $P\not\leq^1_1Q$.
\item There are $P$ and $Q\leq^1_1 P$ such that $P\leq_{\Sigma^0_1\text{\sf -LEM}}Q$ but $P\not\leq_{\Sigma^0_1\text{\sf -LLPO}}Q$.
\item For every $P$, there exists $Q\leq^1_1 P$ such that $P\leq_{\Sigma^0_1\text{\sf -LEM}}Q$, whereas, for every $X$, if $P\leq_{\Sigma^0_1\text{\sf -DNE}}Q\otimes X$ then $P\leq^1_1X$.
\item There are $P$ and $Q\leq^1_1 P$ such that $P\leq_{\Delta^0_2\text{\sf -LEM}}Q$ but $P\not\leq_{\Sigma^0_1\text{\sf -LEM}}Q$.
\item There are $P$ and $Q\leq^1_1 P$ such that $P\leq_{\Sigma^0_2\text{\sf -LLPO}}Q$ but $P\not\leq_{\Delta^0_2\text{\sf -LEM}}Q$.
\item There is $P$ such that, for every $Q$, if $P\leq_{\Sigma^0_2\text{\sf -LLPO}}Q$, then $P\leq_{\Sigma^0_1\text{\sf -LEM}}Q$.
\item For every $P$ and $R$, there exists $Q\leq^1_1 P$ such that $P\leq_{\Sigma^0_2\text{\sf -DNE}}Q$ but $R\not\leq_{\Sigma^0_2\text{\sf -LLPO}}Q$.
\item For every $P$, there exists $Q\leq^1_1 P$ such that $P\leq_{\Sigma^0_2\text{\sf -LEM}}Q$, whereas, for every $X$, if $P\leq_{\Sigma^0_2\text{\sf -DNE}}Q\otimes X$ then $P\leq_{\Sigma^0_2\text{\sf -DNE}}X$.
\item For every $P$ and $R$, there exists $Q\leq P$ such that $P\leq_{\Sigma^0_3\text{\sf -DNE}}Q$ but $R\not\leq_{\Sigma^0_2\text{\sf -LEM}}Q$.
\end{enumerate}
\end{theorem}

\begin{proof}
See Part I \cite[Section 6]{HK_PartI} for the definitions of partial multivalued functions and their characterizations.

(1) By Corollary \ref{cor:2:separ}.
(2) By Corollary \ref{cor:2:separ2}.
(3) By Corollary \ref{cor:3a:HigHig}.
(4) By Corollary \ref{cor:2:bl-bel} (2).
(5) By Corollary \ref{cor:2:l-b} (1).
(6) By Theorem \ref{thm:contig:b-bl}.
(7) By Corollary \ref{cor:5:l-bref} (2).
(8) By Corollary \ref{cor:5:l-tlref}.
(9) By Theorem \ref{thm:NRMP:contig:tl-w}.
\end{proof}

\begin{definition}[Mylatz]
The {\em $\Sigma^0_1$ lessor limited principle of omniscience with $(m/k)$ wrong answers}, $\Sigma^0_1\text{-}{\sf LLPO}_{m/k}$, is the following multi-valued function.
\index{lessor limited principle of omniscience!with $(m/k)$ wrong answers}%
\index{$\Sigma^0_1\text{-}{\sf LLPO}_{m/k}$}%
\begin{align*}
&\Sigma^0_1\text{-}{\sf LLPO}_{m/k}:\subseteq\nn^\nn\rightrightarrows k,& & x\mapsto\{l<k:(\forall n\in\nn)\;x(kn+l)=0\}.
\end{align*}
Here, ${\rm dom}(\Sigma^0_1\text{-}{\sf LLPO}_{m/k})=\{x\in\nn^\nn:x(n)\not=0,\mbox{ for at most $m$ many }n\in\nn\}$.
\end{definition}

\begin{remark}
It is well-known that the parallelization of $\Sigma^0_1\text{-}{\sf LLPO}_{1/2}$ is equivalent to Weak K\"onig's Lemma, {\sf WKL} (hence, is Weihrauch equivalent to the closed choice for Cantor space, ${\sf C}_{2^\nn}$).
\end{remark}

\begin{definition}~
\begin{enumerate}
\item (Cenzer-Hinman \cite{CH2}) A set $P\subseteq k^\nn$ is {\em $(m,k)$-separating} if $P=\prod_{n\in\nn}F_n$ for some uniform sequence $\{F_n\}_{n\in\nn}$ of $\Pi^0_1$ sets $F_n\subseteq k$, where $\#(k\setminus F_n)\leq m$ for any $n\in\nn$.
\index{Pi01 set@$\Pi^0_1$ set!$(m,k)$-separating}%
\item A function $f:\nn^m\to k$ is {\em $k$-valued $m$-diagonally noncomputable in $\alpha\in\nn^\nn$} if the value $f(\lrangle{e_0,\dots,e_{m-1}})$ does not belong to $\{\Phi_{e_i}(\alpha;\lrangle{e_0,\dots,e_{m-1}}):i<m\}$ for each argument  $\lrangle{e_0,\dots,e_{m-1}}\in\nn^m$.
\index{diagonally noncomputable!$k$-valued $m$-}%
By ${\sf DNR}_{m/k}(\alpha)$\index{${\sf DNR}_{m/k}(\alpha)$}, we denote the set of all $k$-valued functions which are $m$-diagonally noncomputable in $\alpha$.
\item The {\em $(m/k)$ diagonally noncomputable operation} ${\sf DNR}_{m/k}:\nn^\nn\rightrightarrows k^\nn$ is the multi-valued function mapping $\alpha\in\nn^\nn$ to ${\sf DNR}_{m/k}(\alpha)$.
\index{diagonally noncomputable!$(m/k)$-operation}\index{${\sf DNR}_{m/k}$}%
\end{enumerate}
\end{definition}

\begin{remark}
Clearly ${\sf DNR}_{m/k}(\emptyset)$ is $(m,k)$-separating.
The structure of Medvedev degrees of $(m,k)$-separating sets have been studied by Cenzer-Hinman \cite{CH2}.
Diagonally noncomputable functions are extensively studied in connection with {\em algorithmic randomness}, for example, see Greenberg-Miller \cite{GM}.
\end{remark}

\begin{prop}\label{prop:5:DNR-LLPO}
${\sf DNR}_{m/k}$ is Weihrauch equivalent to $\widehat{\Sigma^0_1\text{-}{\sf LLPO}_{m/k}}$.
\end{prop}

\begin{proof}\upshape
To see $\widehat{\Sigma^0_1\text{-}{\sf LLPO}_{m/k}}\leq_W{\sf DNR}_{m/k}$, for given $(x_i:i\in\nn)$, let $e^i_t$ be an $\bigoplus_{i\in\nn}x_i$-computable index of an algorithm, for any argument, which returns $l$ at stage $s$ if $l\in L_{s+1}\setminus L_s$ and $\#L_s=t$, where $L_s=\{l^*<k:(\exists n)\;kn+l^*<s\;\&\;x_i(kn+l^*)\not=0\}$.
Clearly, $\{e^i_t:i\in\nn\;\&\;t<m\}$ is computable uniformly in $\bigoplus_{i\in\nn}x_i$.
For any $f\in{\sf DNR}_{m/k}(\bigoplus_{i\in\nn}x_i)$, the function $i\mapsto f(\lrangle{e^i_0,\dots,e^i_{m-1}})$ belongs to $\widehat{\Sigma^0_1\text{-}{\sf LLPO}_{m/k}}(\lrangle{x_i:i\in\nn})$.
Conversely, for given $x\in\nn^\nn$, for the $i$-th $m$-tuple $\lrangle{e_0,\dots,e_{m-1}}\in\nn^m$, we set $x_i(ks+l)=1$ if $\Phi_{e_t}(\lrangle{e_0,\dots,e_{m-1}})$ converges to $l<k$ at stage $s\in\nn$ for some $t<m$, and otherwise we set $x_i(ks+l)=0$.
Clearly $\{x_i:i\in\nn\}$ is uniformly computable in $x$.
Then, for any $\lrangle{l_i:i\in\nn}\in\widehat{\Sigma^0_1\text{-}{\sf LLPO}_{m/k}}(\lrangle{x_i:i\in\nn})\subseteq k^\nn$, we have $l_i\not\in\{\Phi_{e_t}(\lrangle{e_0,\dots,e_{m-1}}):t<m\}$ by our construction.
Hence, the $k$-valued function $i\mapsto l_i$ is $m$-diagonally noncomputable in $x$.
\end{proof}

Recall from Part I \cite[Section 6]{HK_PartI} that $\star$ is the operation on Weihrauch degrees such that is defined by $F\star G=\max\{F^*\circ G^*:F^*\leq_WF\;\&\;G^*\leq_WG\}$.
See \cite{lRP13} for more information on $\star$.

\begin{cor}\label{cor:5:DNR-LLPO}
Let $k\geq 2$ be any natural number.
\begin{enumerate}
\item $\widehat{\Sigma^0_1\text{-}{\sf LLPO}_{1/k}}\not\leq_W\Sigma^0_2\text{-}{\sf DNE}\star\widehat{\Sigma^0_1\text{-}{\sf LLPO}_{1/k+1}}$.
\item $\widehat{\Sigma^0_1\text{-}{\sf LLPO}_{1/k}}\not\leq_W\Sigma^0_2\text{-}{\sf LLPO}\star\widehat{\Sigma^0_1\text{-}{\sf LLPO}_{1/k+1}}$.
\item $\widehat{\Sigma^0_1\text{-}{\sf LLPO}_{1/k}}\leq_W\Sigma^0_2\text{-}{\sf LEM}\star\widehat{\Sigma^0_1\text{-}{\sf LLPO}_{1/k+1}}$.
\end{enumerate}
\end{cor}

\begin{proof}\upshape
By Corollary \ref{cor:5:DNR2} and Proposition \ref{prop:5:DNR-LLPO}, the item (1) and (2) are satisfied.
It is not hard to show the item (3) by analyzing Theorem \ref{thm:5:DNR3}.
\end{proof}

\begin{remark}
By combining the results from Cenzer-Hinman \cite{CH2} and our previous results, we can actually show the following.
\begin{enumerate}
\item $\widehat{\Sigma^0_1\text{-}{\sf LLPO}_{n/l}}\not\leq_W\Sigma^0_2\text{-}{\sf DNE}\star\widehat{\Sigma^0_1\text{-}{\sf LLPO}_{m/k}}$, whenever $0<n<l<\lceil k/m\rceil$.
\item $\widehat{\Sigma^0_1\text{-}{\sf LLPO}_{n/l}}\not\leq_W\Sigma^0_2\text{-}{\sf LLPO}\star\widehat{\Sigma^0_1\text{-}{\sf LLPO}_{m/k}}$, whenever $0<n<l<\lceil k/m\rceil$.
\item $\widehat{\Sigma^0_1\text{-}{\sf LLPO}_{n/l}}\leq_W\Sigma^0_2\text{-}{\sf LEM}\star\widehat{\Sigma^0_1\text{-}{\sf LLPO}_{m/k}}$, whenever $0<n<l$ and $0<m<k$.
\end{enumerate}
These results suggest, within some constructive setting, that the $\Sigma^0_2$ law of excluded middle is sufficient to show the formula $\widehat{\Sigma^0_1\text{-}{\sf LLPO}_{m/k}}\rightarrow\widehat{\Sigma^0_1\text{-}{\sf LLPO}_{n/l}}$, whereas neither the $\Sigma^0_2$ double negation elimination nor the $\Sigma^0_2$ lessor limited principle of omniscience is sufficient.
\end{remark}

\begin{cor}
${\sf DNR}_2\leq_{\Sigma^0_2\text{\sf -LEM}}{\sf DNR}_3$; ${\sf DNR}_2\not\leq_{\Sigma^0_2\text{\sf -LLPO}}{\sf DNR}_3$; ${\sf DNR}_2\not\leq_{\Sigma^0_2\text{\sf -DNE}}{\sf DNR}_3$; ${\sf MLR}\leq_{\Sigma^0_2\text{\sf -LEM}}{\sf DNR}_3$; and ${\sf MLR}\not\leq_{\Sigma^0_2\text{\sf -DNE}}{\sf DNR}_3$.
Here, ${\sf MLR}$ denotes the set of all Martin-L\"of random reals.
\end{cor}

\begin{proof}
For the first three statements, see Corollary \ref{cor:5:DNR2} and Theorem \ref{thm:5:DNR3}.
It is easy to see that ${\sf MLR}\leq^1_1{\sf DNR}_2\leq_{\Sigma^0_2\text{-}{\sf LEM}}{\sf DNR}_3$.
It is shown by Downey-Greenberg-Jockusch-Millans \cite{DGJM} that ${\sf MLR}\not\leq^1_1{\sf DNR}_3$.
By homogeneity of ${\sf DNR}_3$ and Proposition \ref{prop:5:DNR1}, we have ${\sf MLR}\not\leq_{\Sigma^0_2\text{-}{\sf DNE}}{\sf DNR}_3$.
\end{proof}

\subsection{Some Intermediate Lattices are Not Brouwerian}

Recall from Medvedev's Theorem \cite{Med}, Muchnik's Theorem \cite{Muc}, and Part I \cite{HK_PartI} that the degree structures $\mathcal{D}^1_1$, $\mathcal{D}^1_\omega$, and $\mathcal{D}^\omega_1$ are Browerian.
Indeed, we have already observed that one can generate $\mathcal{D}^1_\omega$ from a logical principle so called {the \em $\Sigma^0_2$-double negation elimination}.
Though $\mathcal{D}^1_{<\omega}$, $\mathcal{D}^1_{\omega|<\omega}$ and $\mathcal{D}^{<\omega}_1$ are also generated from certain logical principles over $\mathcal{D}^1_1$ as seen before, surprisingly, these degree structures are {\em not} Brouwerian.

\begin{theorem}\label{thm:2:nonBrouwer}
The degree structures $\mathcal{D}^1_{<\omega}$, $\mathcal{D}^1_{\omega|<\omega}$, $\mathcal{D}^{<\omega}_1$, $\mathcal{P}^1_{<\omega}$, $\mathcal{P}^1_{\omega|<\omega}$, and $\mathcal{P}^{<\omega}_1$ are not Brouwerian.
\end{theorem}

Put $\mathcal{A}(P,Q)=\{R\subseteq\nn^\nn:Q\leq^1_{<\omega} P\otimes R\}$, and $\mathcal{B}(P,Q)=\{R\subseteq\nn^\nn:Q\leq^{<\omega}_1 P\otimes R\}$.
Note that $\mathcal{A}(P,Q)\subseteq\mathcal{B}(P,Q)$.
Then we show the following lemma.

\begin{lemma}\label{lem:2:nonBrouwer}
There are $\Pi^0_1$ sets $P,Q\subseteq 2^\nn$, and a collection $\{Z_e\}_{e\in\mathbb{N}}$ of $\Pi^0_1$ subsets of $2^\nn$ such that $Z_e\in\mathcal{A}(P,Q)$, and that, for every $R\in\mathcal{B}(P,Q)$, we have $R\not\leq^{\omega}_1 Z_e$ for some $e\in\nn$.
\end{lemma}

\begin{proof}
By Jockusch-Soare's theorem \ref{thm:JS2}, we have a collection $\{S_i\}_{i\in\nn}$ of nonempty $\Pi^0_1$ subsets of $2^\nn$ such that $x_k\not\leq_T\bigoplus_{j\not=k}x_j$ for any choice $x_i\in S_i$, $i\in\nn$.
Consider the following sets.
\begin{align*}
P&={\sf CPA}\fr\{S_{\lrangle{e,0}}\fr S_{\lrangle{e,1}}\fr\dots\fr S_{\lrangle{e,e}}\}_{e\in\nn},\quad Z_e=S_{\lrangle{e,e+1}},\\
Q&={\sf CPA}\fr\{Q_n\}_{n\in\nn},\quad\text{where }\;Q_{\lrangle{e,i}}=
\begin{cases}
S_{\lrangle{e,i}}\otimes Z_e, & \mbox{ if }i\leq e,\\
(P\setminus[\rho_e])\otimes Z_e, & \mbox{ if }i=e+1,\\
\emptyset, & \mbox{ otherwise.}
\end{cases}
\end{align*}
Here, $\rho_e$ is the $e$-th leaf of the corresponding computable tree $T_{\sf CPA}$ for ${\sf CPA}$.
To see $Z_e\in\mathcal{A}(P,Q)$, choose an element $f\oplus g\in P\otimes Z_e$.
If $f\res n\in T_{\sf CPA}$ or $f\res n$ extends a leaf except $\rho_e$, our learner $\Psi((f\res n)\oplus g)$ guesses an index of the identity function.
If $f\res n$ extends $\rho_e$, then $\Psi$ first guesses $\Phi_{\Psi((f\res n)\oplus g)}(f\oplus g)=(f^{\shft|\rho_e|})\oplus g$.
By continuing this guessing procedure, if $f\res n$ is of the form $\rho_e\fr\tau^0\fr\tau^1\fr\dots\fr\tau^i\fr\tau$ such that $\tau^j$ is a leaf of $S_{\lrangle{e,j}}$ for each $j\leq i$, and $\tau$ does not extend a leaf of $S_{\lrangle{e,j+1}}$, then $\Psi$ guesses $\Phi_{\Psi((f\res n)\oplus g)}(f\oplus g)=(f^{\shft(|\rho_e|+|\tau^0|+\dots+|\tau^i|)})\oplus g$.
Note that $i<e$, since $f\in P$.
It is easy to see that $Q\leq^1_{<\omega}P\otimes Z_e$ via the learner $\Psi$, where $\#\{n\in\nn:\Psi((f\oplus g)\res n+1)\not=\Psi((f\oplus g)\res n)\}\leq e+1$.
Therefore, $Z_e\in\mathcal{A}(P,Q)$.

Fix $R\in\mathcal{B}(P,Q)$.
As $Q\leq^{<\omega}_1P\otimes R$, there is $b\in\mathbb{N}$ such that, for every $f\oplus g\in P\otimes R$, we must have $\Phi_e(f\oplus g)\in Q$ for some $e<b$.
Suppose for the sake of contradiction that $R\leq^{\omega}_1Z_{b+1}$.
Then, for any $h\in Z_{b+1}$, we have $g\in R$ with $g\leq_T h$.
Pick $f_0\in\rho_{b+1}\fr S_{\lrangle{b+1,0}}\subset P\cap[\rho_{b+1}]$.
Since $R\in\mathcal{B}(P,Q)$, there is $e_0<b$ such that $\Phi_{e_0}(f_0\oplus g)\in Q$.
By our choice of $\{S_n\}_{n\in\nn}$ and the property $g\leq_Th\in Z_{b+1}=S_{\lrangle{b+1,b+2}}$, if $e\not=b+1$ or $i\not=0$, then $Q_{\lrangle{e,i}}$ has no $(f_0\oplus g)$-computable element.
Therefore, $\Phi_{e_0}(f_0\oplus g)$ have to extend $\rho_{\lrangle{b+1,0}}$.
Take an initial segment $\sigma_0\subset f_0$ determining $\Phi_{e_0}(\sigma_0\oplus g)\supseteq\rho_{\lrangle{b+1,0}}$.
Extend $\sigma_0$ to a leaf $\tau^0$ of $S_{b+1,0}$, and choose $f_1\in\rho\fr\tau^0\fr S_{b+1,1}\subset P$.
Again we have $e_1<b$ such that $\Phi_{e_1}(f_1\oplus g)\in Q$.
As before, $\Phi_{e_1}(f_1\oplus g)$ have to extend $\rho_{\lrangle{b+1,1}}$.
However, $\rho_{\lrangle{b+1,1}}$ is incomparable with $\rho_{\lrangle{b+1,0}}$.
Hence, we have $e_1\not=e_0$.
Again take an initial segment $\sigma_1\subset f_1$ extending $\sigma_0$ and determining $\Phi_{e_1}(\sigma_1\oplus g)\supseteq\rho_{\lrangle{b+1,1}}$.
By iterating this procedure, we see that $R$ requires at least $b+1$ many indices $e_i$.
This contradicts our assumption.
Therefore, $R\not\leq^{<\omega}_1Z_{b+1}$.
\end{proof}

\begin{proof}[Proof of Theorem \ref{thm:2:nonBrouwer}]
Let $P$, $Q$, and $\{Z_e\}_{e\in\nn}$ be $\Pi^0_1$ sets in \ref{lem:2:nonBrouwer}.
Fix $(\alpha,\beta)\in\{(1,<\omega),(1,\omega|<\omega),(<\omega,1)\}$.
To see $\mathcal{D}^\alpha_\beta$ is not Brouwerian, it suffices to show that there is no $(\alpha,\beta)$-least $R$ satisfying $Q\leq^\alpha_\beta P\otimes R$.
If $R$ satisfies $Q\leq^\alpha_\beta P\otimes R$, then clearly $R\in\mathcal{B}(P,Q)$ since $\leq^\alpha_\beta$ is stronger than or equals to $\leq^{<\omega}_1$.
Then, $R\not\leq^\alpha_\beta Z_e$ for some $e\in\nn$.
Moreover, $Z_e\in\mathcal{A}(P,Q)$ implies $Q\leq^\alpha_\beta P\otimes Z_e$, since $\leq^\alpha_\beta$ is weaker than or equals to $\leq^1_{<\omega}$.
Hence $R$ is not such a smallest set.
By the same argument, it is easy to see that $\mathcal{P}^\alpha_\beta$ is not Brouwerian, since $Z_e$ is $\Pi^0_1$.
\end{proof}

\begin{theorem}\label{thm:5:nonBrouwer2}
$\mathcal{D}^{<\omega}_\omega$ and $\mathcal{P}^{<\omega}_\omega$ are not Brouwrian.
Moreover, the order structures induced by $(\mathcal{P}(\nn^\nn),\leq_{\Sigma^0_2\text{\sf -LEM}})$ and $(\text{the set of all nonempty } \Pi^0_1 \text{ subsets of } 2^\nn,\leq_{\Sigma^0_2\text{\sf -LEM}})$ are not Brouwrian.
\end{theorem}

\begin{lemma}\label{lem:5:n_o-comp}
Let $\{S_i\}_{i\leq n}$ be a collection of $\Pi^0_1$ subsets of $2^\nn$ with the property for each $i\leq n$ that $\bigcup_{k\not=i}S_k$ has no element computable in $x_i\in S_i$.
Then, there is no $(n,\omega)$-computable function from $\bhtie_{i\leq n}S_i$ to $\bigoplus_{i\leq n}S_i$.
\end{lemma}

\begin{proof}
Assume the existence of an $(n,\omega)$-computable function from $\bhtie_{i\leq n}S_i$ to $\bigoplus_{i\leq n}S_i$ which is identified by $n$ many learners $\{\Psi_i\}_{i<n}$.
Let $F_i$ be a partial $(n,\omega)$-computable function identified by $\Psi$, i.e., $F_i(x)=\Phi_{\lim_n\Psi(x\res n)}(x)$.
Note that $\bhtie_{i\leq n}S_i\subseteq\bigcup_{i<n}{\rm dom}(F_i)$.
For each $i<n$, put $D_i={\rm dom}(F_i)\cap F_i^{-1}(\bigoplus_{i\leq n}S_i)$.
Let $T_{S_i}$ denote the corresponding tree for $S_i$, for each $i\leq n$.
Define $S^\heartsuit_E$ for each $E\subseteq n+1$ to be the set of all infinite paths through the following tree $T_E$.
\begin{align*}
&T_E=\bhtie_{\sigma\in T_0^E}\left(\bhtie_{\sigma\in T_1^E}\left(\dots\left(\bhtie_{\sigma\in T_{n-1}^E}[T_n^E]\right)\dots\right)\right).\\
&\text{Here, }
T_i^E=
\begin{cases}
T_{S_i}^{ext},&\mbox{ if }i\in E,\\
\text{some finite subtree of }T_{S_i}^{ext},&\mbox{ otherwise.}
\end{cases}
\end{align*}
Here, the choice of ``some finite subtree of $T_{S_i}^{ext}$'' depends on the context, and is implicitly determined when $E$ is defined.
For each $E\subseteq n+1$, clearly $S^\heartsuit_E$ is a closed subset of $\bhtie_{i\leq n}S_i$.
Divide $S^\heartsuit_{n+1}$ into $n+1$ many parts $\{S^*_i\}_{i\leq n}$, where $S^\heartsuit_{n+1}$ is equal to $\bigcup_{i\leq n}S_i^*$, and each $S^*_i$ is degree-isomorphic to $S_i$.

For each $i\leq n$, check whether there is a string $\sigma$ extendible in $S^\heartsuit_{n+1}$ such that $S^\heartsuit_{n+1}\cap D_i\cap[\sigma]$ is contained in $S^*_j$ for some $j\leq n$.
If yes, for such a least $i\leq n$, choose a witness $\sigma_0=\sigma$, and put $A_0=\{i\}$, and $B_0=\{j\}$.
Then, for such $j\in B_0$, ``some finite subtree of $T_{S_j}^{ext}$'' is choosen as the set of all strings $\eta$ used in $\sigma_0$ as a part of $T_{S_j}^{ext}$ in the sense of the definition of $\bhtie_{i\leq n}S_i$, or successors of such $\eta$ in $T_{S_j}^{ext}$.
Note that $\sigma_0$ is also extendible in $S^\heartsuit_{(n+1)\setminus\{j\}}$.
Inductively, for some $s<n$ assume that $\sigma_s$, $A_s$, and $B_s$ has been already defined.
For each $i\not\in A_s$, check whether there is a string $\sigma\supseteq\sigma_s$ extendible in $S^\heartsuit_{(n+1)\setminus B_s}$ such that $S^\heartsuit_{(n+1)\setminus B_s}\cap D_i\cap[\sigma]$ is contained in $S^*_j$ for some $j\not\in B_s$.
If yes, for such a least $i\not\in A_s$, choose a witness $\sigma_0=\sigma$, and put $A_{s+1}=A_s\cup\{i\}$, and $B_{s+1}=B_s\cup\{j\}$.
As before, for such $j\in B_{s+1}$, ``some finite subtree of $T_{S_j}^{ext}$'' is choosen as the set of all strings $\eta$ used in $\sigma_{s+1}$ as a part of $T_{S_j}^{ext}$, or successors of such $\eta$ in $T_{S_j}^{ext}$.
Note that $\sigma_{s+1}$ is also extendible in $S^\heartsuit_{(n+1)\setminus B_{s+1}}$.
If no such $i\not\in A_s$ exists, finish our construction of $\sigma$, $A$, and $B$.
Then, put $A=A_s$, $B=B_s$, and define $\sigma^*$ to be the last witness $\sigma_s$.

Put $A^-=n\setminus A$ and $B^-=(n+1)\setminus B$.
Note that $\#A^-+1=\#B^-$, since $\#A=\#B$.
Therefore, $B$ contains at least one element.
By our assumption, for any $x\in S^\heartsuit_{B^-}\cap[\sigma^*]\not=\emptyset$, we must have $F_i(x)\in\bigoplus_{i\leq n}S_i$ for some $i\in A^-$.
Thus, $A^-$ is nonempty.

Fix a sequence $\alpha\in(A^-)^{\nn}$ such that, for each $i\in A^-$, there are infinitely many $n\in\nn$ such that $\alpha(n)=i$.
First set $\tau_0=\sigma^*$.
Inductively assume that $\tau_s\supseteq\sigma^*$ has been already defined.
By our definition of $\sigma^*$, $A$ and $B$, if $\xi$ extends $\sigma^*$, then the set $S^\heartsuit_{B^-}\cap D_{\alpha(s)}\cap[\xi]$ intersects with $S_j^*$ for at least two $j\in B^-$.
Therefore, we can choose $x\in S_{B^-}^\heartsuit\cap D_{\alpha(s)}\cap[\tau_s]\cap S_j^*\not=\emptyset$ for some $j\in B^-$.
Then, $F_{\alpha(s)}(x;0)=j$, by our assumption of $\{S_i\}_{i\leq n}$.
Find a string $\tau^*_s$ such that $\tau_s\subseteq\tau^*_s\subset x$ and $F_{\alpha(s)}(\tau^*_s;0)=j$.
Again, we can choose $x^*\in S_{B^-}^\heartsuit\cap D_{\alpha(s)}\cap[\tau^*_s]\cap S_k^*\not=\emptyset$ for some $k\in B^-\setminus\{j\}$.
Then, we must have $F_{\alpha(s)}(x^*;0)=k\not=j$.
Let $\tau_{s+1}$ be a string such that $\tau^*_{s}\subseteq\tau_{s+1}\subset x^*$ and $F_{\alpha(s)}(\tau_{s+1};0)=k$.
Therefore, between $\tau_s$ and $\tau_{s+1}$, the learner $\Psi_{\alpha(s)}$ changes his mind.

Define $y=\bigcup_s\tau_s$.
Then $y$ is contained in $S^\heartsuit_{B^-}$, since $S^\heartsuit_{B^-}$ is closed.
However, for each $i\in A^-$, by our construction of $y$, the value $F_i(y)$ does not converge.
Moreover, for each $i\not\in A^-$, by our definition of $A$, $B$, and $\sigma^*\subset y$, even if $F_i(y)$ converges, $F_i(y)\not\in\bigoplus_{i\leq n}S_i$.
Consequently, there is no $(n,\omega)$-computable function from $\bhtie_{i\leq n}S_i\supset S^\heartsuit_{B^-}$ to $\bigoplus_{i\leq n}S_i$ as desired.
\end{proof}

Put $\mathcal{J}(P,Q)=\{R\subseteq\nn^\nn:Q\leq_{\Sigma^0_2\text{\sf -LEM}}P\otimes R\}$, and $\mathcal{K}(P,Q)=\{R\subseteq\nn^\nn:Q\leq^{<\omega}_\omega P\otimes R\}$.
Note that $\mathcal{J}(P,Q)\subseteq\mathcal{K}(P,Q)$.
Then we show the following lemma.

\begin{lemma}\label{lem:2:nonBrouwer2}
There are $\Pi^0_1$ sets $P,Q\subseteq 2^\nn$, and a collection $\{Z_e\}_{e\in\mathbb{N}}$ of $\Pi^0_1$ subsets of $2^\nn$ such that $Z_e\in\mathcal{J}(P,Q)$, and that, for every $R\in\mathcal{K}(P,Q)$, we have $R\not\leq^{\omega}_1 Z_e$ for some $e\in\nn$.
\end{lemma}

\begin{proof}
By Jockusch-Soare's theorem \ref{thm:JS2}, we have a collection $\{S_i\}_{i\in\nn}$ of nonempty $\Pi^0_1$ subsets of $2^\nn$ such that $x_k\not\leq_T\bigoplus_{j\not=k}x_j$ for any choice $x_i\in S_i$, $i\in\nn$.
Consider the following sets.
\begin{align*}
P&={\sf CPA}\fr\{S_{\lrangle{e,0}}\htie S_{\lrangle{e,1}}\htie\dots\htie S_{\lrangle{e,e}}\}_{e\in\nn},\quad Z_e=S_{\lrangle{e,e+1}},\\
Q&={\sf CPA}\fr\{Q_n\}_{n\in\nn},\quad\text{where }\;Q_{\lrangle{e,i}}=
\begin{cases}
S_{\lrangle{e,i}}\otimes Z_e, & \mbox{ if }i\leq e,\\
(P\setminus[\rho_e])\otimes Z_e, & \mbox{ if }i=e+1,\\
\emptyset, & \mbox{ otherwise.}
\end{cases}
\end{align*}
Here, $\rho_e$ is the $e$-th leaf of the corresponding computable tree $T_{\sf CPA}$ for ${\sf CPA}$.
To see $Z_e\in\mathcal{J}(P,Q)$, for $f\oplus g\in P\otimes Z_e$, by using {\sf $\Sigma^0_1$-LEM}, check whether $f$ does no extend $\rho_e$.
If no, outputs $\rho_{e,e+1}\fr(f\oplus g)$.
If $f$ extends $\rho_e$, it is not hard to see that an finite iteration of {\sf $\Sigma^0_2$-LEM} can divide $(\rho_e\fr S_{\lrangle{e,0}}\htie S_{\lrangle{e,1}}\htie\dots\htie S_{\lrangle{e,e}})\otimes Z_e$ into $\{S_{\lrangle{e,i}}\otimes Z_e\}_{i\leq e}$.

Fix $R\in\mathcal{K}(P,Q)$.
As $Q\leq^{\omega}_{<\omega}P\otimes R$, some $(b,\omega)$-computable function $F$ maps $P\otimes R$ into $Q$.
Suppose for the sake of contradiction that $R\leq^{\omega}_1Z_{b}$.
Then, for any $h\in Z_{b}$, we have $g\in R$ with $g\leq_T h$.
Then, $F$ maps $(P\cap[\rho_{b}])\otimes\{g\}$ into $Q\cap(\bigcup_{i\leq b}\rho_{\lrangle{b,i}})$ by our choice of $\{S_n\}_{n\in\nn}$.
Note that $(P\cap[\rho_{b}])\otimes\{g\}\equiv^1_1(\bhtie_{i\leq b}S_{\lrangle{e,i}})\otimes\{g\}$, and $Q\cap(\bigcup_{i\leq b}\rho_{\lrangle{b,i}})\equiv^1_1(\bigoplus_{i\leq b}S_{\lrangle{e,i}})\otimes Z_e$.
Therefore, by Lemma \ref{lem:5:n_o-comp}, $F$ is not $(b,\omega)$-computable.
\end{proof}

\begin{proof}[Proof of Theorem \ref{thm:5:nonBrouwer2}]
Let $P$, $Q$, and $\{Z_e\}_{e\in\nn}$ be $\Pi^0_1$ sets in \ref{lem:2:nonBrouwer2}.
Then, by the same argument as in the proof of Theorem \ref{thm:2:nonBrouwer}, it is not hard to show the desired statement.
\end{proof}

\begin{cor}
If $(\alpha,\beta)\in\{(1,1),(1,\omega),(\omega,1)\}$, and $(\gamma,\delta)\in\{(1,<\omega),(1,\omega|<\omega),(<\omega,1),(<\omega,\omega)\}$, then, there is an elementary difference between $\mathcal{D}^\alpha_\beta$ and $\mathcal{D}^\gamma_\delta$, in the language of partial orderings $\{\leq\}$.
\end{cor}

\begin{proof}
Recall that the degree structures $\mathcal{D}^1_1$, $\mathcal{D}^1_\omega$, and $\mathcal{D}^\omega_1$ are Browerian, i.e., they satisfy the following elementary formula $\psi$ in the language of partial orders.
\[\psi\equiv(\forall p,q)(\exists r)(\forall s)\;(p\leq q\vee r\;\&\;(p\leq q\vee s\;\rightarrow\;r\leq s)).\]
Here, the supremum $\vee$ is first-order definable in the language of partial orders.
On the other hand, by Theorem \ref{thm:2:nonBrouwer} and \ref{thm:5:nonBrouwer2}, $\mathcal{D}^1_{<\omega}$, $\mathcal{D}_{\omega|<\omega}$, $\mathcal{D}^{<\omega}_{1}$, and $\mathcal{D}^{<\omega}_\omega$ are not Brouwerian, i.e., they satisfy $\neg\psi$.
\end{proof}

\subsection{Open Questions}

\begin{question}[Small Questions]~
\begin{enumerate}
\item Determine the intermediate logic corresponding the degree structure $\mathcal{D}^1_{\omega}$, where recall that $\mathcal{D}^1_1$ and $\mathcal{D}^\omega_1$ are exactly Jankov's Logic.
\item Does there exist $\Pi^0_1$ sets $P,Q\subseteq 2^\nn$ with $P\leq^1_\omega Q$ such that there is no $|a|$-bounded learnable function $\Gamma:Q\to P$ for any $a\in\mathcal{O}$?
For a $\Pi^0_1$ set $\widehat{P}$ in Theorem \ref{thm:contig:b-l}, does there exist a function $\Gamma:\widehat{P}\to P$ $(1,\omega)$-computable via an $|a|$-bounded learner for some notation $a\in\mathcal{O}$?
\item Does there exist a pair of special $\Pi^0_1$ sets $P,Q\subseteq 2^\nn$ with a function $\Gamma:Q\htie P\to Q\oplus P$ (or $\Gamma:Q\cls\btie P\to Q\oplus P$) which is learnable by a team of confident learners (or a team of eventually-Popperian learners)?
\item Let $P_0$, $P_1$, $Q_0$, and $Q_1$ be $\Pi^0_1$ subsets of $2^\nn$ with $Q_0\leq^1_\omega Q_1$ and $P_0\leq^1_\omega P_1$.
Then, does $\bhk{P_0\vee Q_0}_{\Sigma^0_2}\leq^1_\omega\bhk{P_1\vee Q_1}_{\Sigma^0_2}$ hold?
Moreover, if $Q_0\leq^1_\omega Q_1$ is witnessed by an eventually Lipschitz learner, then does $P_0\htie Q_0\leq^1_\omega P_1\htie Q_1$ hold?
\item Compare the reducibility $\leq^\omega_{tt,1}$ and other reducibility notions (e.g., $\leq^{<\omega}_{tt,1}$, $\leq^1_{<\omega}$, $\leq^1_{\omega|<\omega}$, $\leq^{<\omega}_1$ and $\leq^1_\omega$) for $\Pi^0_1$ subsets of Cantor space $2^\nn$.
\end{enumerate}
\end{question}

\begin{question}[Big Questions]~
\begin{enumerate}
\item Are there elementary differences between any two different degree structures $\mathcal{D}^\alpha_{\beta|\gamma}$ and  $\mathcal{D}^{\alpha'}_{\beta'|\gamma'}$ ($\mathcal{P}^\alpha_{\beta|\gamma}$ and  $\mathcal{P}^{\alpha'}_{\beta'|\gamma'}$)?
\item Is the commutative concatenation $\tie$ first-order definable in the structure $\mathcal{D}^1_1$ or $\mathcal{P}^1_1$?
\item Is each local degree structure $\mathcal{P}^{\alpha}_{\beta|\gamma}$ first-order definable in the global degree structure $\mathcal{D}^{\alpha}_{\beta|\gamma}$?
\item Is the structure $\mathcal{P}^1_\omega$ dense?
\item Investigate properties of $(\alpha,\beta|\gamma)$-degrees $\dg{a}$ assuring the existence of $\dg{b}>\dg{a}$ with the same $(\alpha',\beta'|\gamma')$-degree as $\dg{a}$.
\item Investigate the nested nested model, the nested nested nested model, and so on.
\item Does there exist a natural intermediate notion between $(<\omega,\omega)$-computability (team-learnability) and $(\omega,1)$-computability (nonuniform computability) on $\Pi^0_1$ sets?
\item (Ishihara) Define a uniform (non-adhoc) interpretation (such as the Kleene realizability interpretation) translating each intuitionistic arithmetical sentence (e.g., $(\neg\neg\exists n\forall m A(n,m))\rightarrow(\exists n\forall m A(n,m))$) into a partial multi-valued function (e.g., $\Sigma^0_2\text{-}{\sf DNE}:\subseteq\nn^\nn\rightrightarrows\nn^\nn$).
\end{enumerate}
\end{question}